\documentclass[11pt,reqno]{amsart}
\usepackage[margin=3.5cm]{geometry}

\usepackage{amsmath,amsfonts,amssymb,amscd,amsthm,amsbsy,amscd,bbm, epsf,calc,enumerate,color,graphicx,verbatim,cite,setspace}
\usepackage{thmtools}
\usepackage{thm-restate}
\usepackage[margin=1cm]{caption}
\usepackage{mathtools}
\usepackage[colorlinks=true, linkcolor=blue, citecolor=red, pdfborder={0 0 0}]{hyperref}
\usepackage{accents}
\usepackage{wrapfig}
\usepackage{multirow}
\usepackage{float}
\usepackage{blkarray}
\theoremstyle{plain}
 \newtheorem{theorem}{Theorem}[section]
 \newtheorem{lemma}[theorem]{Lemma}
 \newtheorem{prop}[theorem]{Proposition}
 \newtheorem{cor}[theorem]{Corollary}
 \newtheorem{conjecture}[theorem]{Conjecture}

\theoremstyle{definition}
 \newtheorem{definition}[theorem]{Definition}

\theoremstyle{remark}
 \newtheorem{remark}[theorem]{Remark}

\numberwithin{equation}{section}
\numberwithin{theorem}{section}

\newcommand\nc\newcommand
\nc\dmo\DeclareMathOperator
\nc\bs\boldsymbol

\nc{\Red}[1]{\textcolor[rgb]{1,0,0}{[#1]}}
\nc{\Green}[1]{\textcolor[rgb]{0,1,0}{[#1]}}
\nc{\Blue}[1]{\textcolor[rgb]{0,0,1}{[#1]}}

\nc{\abbr}[1]{{\sc{\lowercase{#1}}}}
\nc{\abbrev}[1]{{\raisebox{-0.0pt}{\tt{\uppercase{\scalebox{0.75}{#1}}}}}}

\nc{\ii}{\mathrm{i}}

\dmo{\sign}{sign}
\dmo{\spt}{spt}
\dmo{\supp}{supp}
\dmo{\sym}{Sym}
\nc{\R}{\mathbb{R}}
\nc{\C}{\mathbb{C}}
\nc{\N}{\mathbb{N}}
\nc{\Z}{\mathbb{Z}}
\nc{\erdos}{Erd\H os }
\nc{\er}{Erd\H os--R\'enyi } 
\dmo{\ls}{\lesssim}
\dmo{\gs}{\gtrsim}
\def \<{\langle}
\def \>{\rangle}

\def \lf {\lfloor}
\def \rf {\rfloor}
\nc{\expo}[1]{\exp \left( #1 \rule{0mm}{3mm}\right)}
\DeclarePairedDelimiter\parentheses{\lparen}{\rparen}

\nc{\dd}{d}%{\mathrm{d}}

\dmo{\e}{\mathbb{E}}
\dmo{\var}{Var}
\dmo{\pr}{\mathbb{P}}
\dmo{\un}{\mathbbm{1}}
\dmo{\1}{\mathbf{1}}
\nc{\eqd}{\,{\buildrel d \over =}\,}

\nc{\bad}{\mathcal{B}}
\nc{\event}{\mathcal{E}}
\nc{\good}{\mathcal{G}}
\nc{\pro}[1]{\mathbb{P}\parentheses*{#1 \rule{0mm}{0mm}}}
\nc{\set}[1]{\left\{ #1 \right\}}

\dmo{\I}{I}
\dmo{\tr}{tr}
\dmo{\rank}{rank}
\dmo{\rk}{Rank}
\dmo{\corank}{corank}
\def \tran {\mathsf{T}}
\def \HS {\mathrm{HS}}
\nc{\Span}{\operatorname{span}}
\dmo{\per}{\text{per}}

\nc{\eps}{\varepsilon}
\nc{\ep}{\epsilon}

\nc{\mA}{\mathcal{A}}
\nc{\mB}{\mathcal{B}}
\nc{\mC}{\mathcal{C}}
\nc{\mD}{\mathcal{D}}
\nc{\mE}{\mathcal{E}}
\nc{\mF}{\mathcal{F}}
\nc{\mG}{\mathcal{G}}
\nc{\mH}{\mathcal{H}}
\nc{\mI}{\mathcal{I}}
\nc{\mJ}{\mathcal{J}}
\nc{\mK}{\mathcal{K}}
\nc{\mL}{\mathcal{L}}
\nc{\mM}{\mathcal{M}}
\nc{\mN}{\mathcal{N}}
\nc{\mO}{\mathcal{O}}
\nc{\mP}{\mathcal{P}}
\nc{\mQ}{\mathcal{Q}}
\nc{\mR}{\mathcal{R}}
\nc{\mS}{\mathcal{S}}
\nc{\mT}{\mathcal{T}}
\nc{\mU}{\mathcal{U}}
\nc{\mV}{\mathcal{V}}
\nc{\mW}{\mathcal{W}}
\nc{\mX}{\mathcal{X}}
\nc{\mY}{\mathcal{Y}}
\nc{\mZ}{\mathcal{Z}}

\nc \wY {\widetilde{Y}}
\nc \wA {\widetilde{A}}
\nc \tA {\widetilde{A}}
\nc \tG {\widetilde{G}}
\nc \stA {A^\star}
\nc \wR {\widetilde{R}}
\nc \wM {\widetilde{M}}
\nc \tnu {\check{\nu}}
\nc \hnu {\hat{\nu}}
\dmo \Sparse {Sparse}
\dmo \Comp {Comp}
\dmo \Flat {Flat}
\dmo \Proj	{Proj}
\dmo \circular {circ}
\nc{\muc}{\mu_{\circular}}
\dmo{\Her}{\mathbf{H}}
\dmo{\Res}{\mathbf{R}}
\nc{\tS}{\widetilde{S}}
\nc{\tT}{\widetilde{T}}
\dmo{\dist}{dist}
\dmo{\Switch}{Switch}
\dmo{\eye}{\mathbf{I}_2}
\dmo{\jay}{\mathbf{J}_2}
\dmo{\codeg}{codeg}
\dmo{\discrep}{discrep}
\dmo{\edge}{disc}
\dmo{\expand}{exp}

\dmo{\II}{\mathbb{I}}
\dmo{\indic}{\1}	%{\mathbbm{1}}
\nc{\sph}{\mathbb{S}^{n-1}}
\nc{\ball}{\mathbb{B}^n}
\dmo{\Id}{Id}
\nc{\oneperp}{{\langle \1 \rangle^{\perp}}}
\nc{\iid}{i.i.d.}
\nc{\me}{{m}}
\nc{\mo}{{m}}
\nc{\ka}{k}
\nc{\Abar}{\bar{A}}
\nc{\bxi}{{\bs{\xi}}}
\nc{\bi}{\bs{i}}

\begin{document}

\title{The Circular Law for random regular digraphs}

\author{Nicholas Cook}
\address{Department of Mathematics, University of California, Los Angeles}
\email{nickcook@math.ucla.edu}
\thanks{The author was partially supported by NSF grant DMS-1266164 and an NSF postdoctoral fellowship.} %DMS-1606310.}

\date{\today}

\begin{abstract}
Let $\log^Cn\le d\le n/2$ for a sufficiently large constant $C>0$ and let $A_n$ denote the adjacency matrix of a uniform random $d$-regular directed graph on $n$ vertices.
We prove that as $n$ tends to infinity, the empirical spectral distribution of $A_n$, suitably rescaled, is governed by the Circular Law.
A key step is to obtain quantitative lower tail bounds for the smallest singular value of additive perturbations of $A_n$.
\end{abstract}


\maketitle

%Indented subsection numbering in TOC
\let\oldtocsubsection=\tocsubsection
\renewcommand{\tocsubsection}[2]{\hspace*{1cm}\oldtocsubsection{#1}{#2}}
\let\oldtocsubsubsection=\tocsubsubsection
\renewcommand{\tocsubsubsection}[2]{\hspace*{2.5cm}\oldtocsubsubsection{#1}{#2}}

\setcounter{tocdepth}{2}
%\tableofcontents

\section{Introduction}

\subsection{Convergence of ESDs and the Circular Law}

For an $n\times n$ matrix $M$ with complex entries and eigenvalues $\lambda_1,\dots, \lambda_n\in \C$ (counted with multiplicity and labeled in some arbitrary fashion), denote the \emph{empirical spectral distribution (ESD)} 
\begin{equation}	\label{def:esd}
\mu_M = \frac1n \sum_{i=1}^n \delta_{\lambda_i}.
\end{equation}
We give the space of probability measures on $\C$ the vague topology.
Thus, a sequence of random probability measures $\mu_n$ over $\C$ converges to another measure $\mu$ in probability  if for every $f\in C_c(\C)$ and every $\eps>0$, 
\begin{equation}	\label{def:weakprob}
\lim_{n\to\infty}\pro{ \left|\int_\C f \,\dd \mu_n - \int_\C f\,\dd \mu\right|>\eps} = 0,
\end{equation}
and $\mu_n$ converges to $\mu$ almost surely if for every $f\in C_c(\C)$, $\int_\C f\dd \mu_n\to \int_\C f \dd\mu$ almost surely. We say that $\mu_n$ converges to $\mu$ in expectation if $\e \int_\C fd\mu_n \to \e \int_\C fd\mu$ for every $f\in C_c(\C)$. 

A well-studied class of non-Hermitian random matrices is the \emph{iid matrix} $X_n$, which has iid centered entries of unit variance.
A seminal result in the theory of non-Hermitian random matrices is the Circular Law for iid matrices, which was established in various forms over several decades. We denote by $\muc$ the normalized Lebesgue measure on the unit disk in $\C$.

\begin{theorem}[Strong Circular Law for iid matrices \cite{TaVu:esd}]	\label{thm:circ}
Fix a complex random variable $\xi$ with zero mean and unit variance, and for each $n\ge1$ form an $n\times n$ random matrix $X_n = (\xi_{ij}^{(n)})$ with entries that are iid copies of $\xi$. 
Then the rescaled ESDs $\mu_{\frac1{\sqrt{n}}X_n}$ converge to $\muc$ almost surely.
\end{theorem}

The above strong form of the Circular Law due to Tao and Vu, and is the culmination of the work of many authors.
Previous works had obtained the Circular Law under additional assumptions on the \emph{atom variable} $\xi$, or with convergence in probability or expectation rather than almost-sure convergence (the above result is called a ``strong law" in analogy with the strong law of large numbers). 
The earliest result was by Ginibre, who established the Circular Law (with convergence in expectation) for the \emph{Ginibre ensemble}, where the atom variable $\xi$ is a standard complex Gaussian \cite{Ginibre65} (see also \cite{Mehta}); the harder case of real Gaussian entries was handled by Edelman \cite{Edelman:circ}. These results relied on explicit formulas available for Gaussian ensembles for the joint density of eigenvalues.
Following influential work of Girko \cite{Girko84}, Bai was the first to rigorously establish the Circular Law for a general class of atom variables, assuming that $\xi$ has bounded density and finite sixth moment \cite{Bai97}.
Following breakthrough work of Rudelson \cite{Rudelson:inv}, Tao--Vu \cite{TaVu:cond} and Rudelson--Vershynin \cite{RuVe:ilo} on the smallest singular value for random matrices with independent entries, the assumptions on the atom variable were progressively relaxed in works of G\"otze--Tikhomirov \cite{GoTi:circ}, Pan--Zhou \cite{PaZh}, and Tao--Vu \cite{TaVu:circ, TaVu:esd}.

Theorem \ref{thm:circ} is an instance of the \emph{universality phenomenon} in random matrix theory, exhibiting an asymptotic behavior of the spectrum which is insensitive to all but a few details of the atom variable (in this case the first two moments). 
In fact, it is a consequence of a more general ``universality principle" established in \cite{TaVu:esd}, which states that if $X_n, X_n'$ are iid matrices generated from atom variables $\xi$ and $\xi'$, respectively, and $M_n$ is a deterministic matrix satisfying $\frac1n\|M_n\|_{\HS} = O(1)$ (where $\|M\|_{\HS}$ is the Hilbert--Schmidt norm), then 
\[
\mu_{M_n+X_n} - \mu_{M_n+X_n'} \to 0\quad \text{in probability}.
\]
(Almost sure convergence is also obtained under an additional technical assumption that we do not state here.) The Circular Law for general iid matrices can then be deduced from the universality principle (taking $M_n=0$) and the Circular Law for the Ginibre ensemble. The perturbations $M_n$ can also give rise to limiting measures different from $\muc$.

Since the work of Tao and Vu the Circular Law has been strengthened and extended in several directions.
In a sequence of works, Bourgade, Yau and Yin \cite{BYY:loccirc1, BYY:loccirc2, Yin:loccirc3} have established the \emph{local Circular Law}, showing that $\muc$ provides a good estimate for the number of eigenvalues of $\frac1{\sqrt{n}}X_n$ in a fixed small ball $B(z,r)$ down to the optimal scale $r\sim n^{-1/2+\eps}$ for arbitrary fixed $\eps>0$, assuming an exponential decay condition for the tails of the atom variable $\xi$. A weaker local law was obtained by Tao and Vu in \cite{TaVu:iidloc} as part of their proof of universality for local eigenvalue statistics. 

We will informally say that a random matrix $Y_n$  (that is, a sequence of $n\times n$ random matrices $(Y_n)_{n\ge1}$) lies in the \emph{Circular Law universality class} if, after rescaling, the ESDs $\mu_{Y_n}$ converge in probability to $\muc$.
Theorem \ref{thm:circ} shows this class contains all iid matrices $X_n$, but in recent years various works have shown it to be somewhat larger.
In \cite{GoTi:circ, TaVu:circ, Wood:sparse, BaRu:sparse} it has been shown that the Circular Law is robust under sparsification, i.e.\ that matrices of the form $Y_n=A_n\circ X_n$ lie in the Circular Law universality class, where $\circ$ denotes the Hadamard (entrywise) product, $X_n$ is an iid matrix, and $A_n$ is a 0--1 matrix of iid Bernoulli($p$) variables, independent of $X_n$, with $p=o(n)$ and $pn$ growing at some speed. In particular, Wood \cite{Wood:sparse} showed the Circular Law holds with convergence in probability if $pn=n^{\eps}$ for any fixed $\eps\in (0,1)$, while the recent work \cite{BaRu:sparse} allows $pn = \omega(\log^2n)$ under higher moment assumptions on the atom variable. 

There has also been extensive work on non-Hermitian matrices with dependent entries. 
In \cite{BCC:markov}, Bordenave, Caputo and Chafa\"i showed the Circular Law class includes random Markov matrices obtained by normalizing the entries of a matrix with iid nonnegative entries of finite variance by the row sums.
Nguyen and Vu obtained the Circular Law for random $\pm1$ matrices with prescribed row sums $|s|\le (1-\eps)n$ for some fixed $\eps\in (0,1]$ \cite{NgVu}. 
Later, Nguyen proved the Circular Law for random doubly stochastic matrices (drawn uniformly from the Birkhoff polytope), which do not enjoy independence between rows or columns \cite{Nguyen:uds}.
In \cite{AdCh}, Adamczak and Chafa\"i  showed that random real matrices having unconditional log-concave distribution obey the Circular Law, extending Edelman's result for real Gaussian matrices.
Adamczak, Chafa\"i and Wolff proved the Circular Law for random matrices with exchangeable entries having finite moments of order $20+\eps$ (if not for the moment assumption this result would generalize the Circular Law for iid matrices) \cite{ACW:exchangeable}. 
In \cite{Cook:circpm}, the author obtained the Circular Law for adjacency matrices of dense random regular digraphs with random edge weights, i.e.\ matrices of the form $A_n\circ X_n$ with $X_n$ an iid matrix and $A_n$ a 0--1 matrix constrained to have all rows and columns sum to $\lf p n\rf$ for some fixed $p\in (0,1)$. Such matrices $A_n$, which are the focus of the present work, can be seen as a discrete version of the doubly stochastic matrices considered by Nguyen.

The second moment hypothesis in Theorem \ref{thm:circ} is sharp. Indeed, in \cite{BCC:heavy}, Bordenave, Caputo and Chafa\"i established a different limiting law for matrices with iid entries lying in the domain of attraction of an $\alpha$-stable distribution for $\alpha\in (0,2)$. 
In \cite{BCCP:markov.heavy} the same authors together with Piras have considered random stochastic matrices obtained by normalizing the entries of an iid heavy-tailed matrix with $\alpha\in (0,1)$ by the row sums, proving convergence of the ESDs to deterministic measure supported on a compact disk (while they do not obtain an expression for the limiting density, simulations indicate that it is not the uniform measure on the disk). 

Finally, a natural question is whether the Circular Law extends to matrices with independent but non-identically distributed entries having finite second moment. 
If the entries all have unit variance then one can replace the assumption of identical distribution with some more general technical hypotheses; see \cite{TaVu:circ}, \cite[p.\ 428]{BaSi10:book}.
Recently, the work \cite{CHNR} studied the asymptotic ESDs for matrices of the form $Y_n = \frac1{\sqrt{n}} A_n\circ X_n$, with $X_n$ an iid matrix having entries with finite $(4+\eps)$-th moment, and $A_n=(\sigma_{ij})$ a fixed ``profile" of standard deviations $\sigma_{ij}\in [0,1]$. In particular, it was shown that the Circular Law holds if the standard deviations $\sigma_{ij}$ are uniformly bounded and the variance profile $(\frac1n\sigma_{ij}^2)$ is doubly stochastic. Examples were also provided of variance profiles leading to limiting measures different from $\muc$, though they are always compactly supported and rotationally symmetric.
Another recent work \cite{AEK:inhom} has obtained a local law (analogous to the above-mentioned local Circular Law of Bourgade--Yau--Yin) for $Y_n$ as above, but under stronger assumptions: that the entries have smooth distribution and the variances are uniformly bounded above and below by positive constants. \\

The Circular Law and its extensions have been applied to the stability analysis of complex dynamical systems, in theoretical ecology \cite{May72} and neuroscience \cite{SCS:chaos}. In the latter work, an iid matrix was used to model the \emph{synaptic matrix} for a large neural network. 
There has since been significant effort to extend the results of \cite{SCS:chaos} to various random matrix models incorporating additional structural features of natural neural networks such as the brain, using both rigorous and non-rigorous methods \cite{RaAb:neural, ASS, ARS:block,AFM:neural}. 
However, a key feature that has not been covered by these works is \emph{sparsity}. While the aforementioned works such as \cite{Wood:sparse, BaRu:sparse} can be used to extend the analysis of \cite{SCS:chaos} to sparse iid matrices, it would also be interesting to treat networks where each node has a specified valence. However, such constraints destroy the independence between entries, making the analysis of such models challenging. 

In the present work we make a first step in this direction by extending the Circular Law to adjacency matrices of random regular digraphs. 
For integers $n\ge 1$ and $d\in [n]$ denote
\begin{equation}	\label{def:And}
\mA_{n,d} = \left\{ A \in \{0,1\}^{n\times n} :\; A\1 = A^\tran\1 = d\1 \right\},
\end{equation}
which is the set of 0--1 adjacency matrices for $d$-regular directed graphs (digraphs) on $n$ vertices, allowing self-loops.
(Here and throughout, $\1=\1_n$ denotes the column vector of all 1s.) 
Given $A\in \mA_{n,d}$, we denote the normalized matrix
\begin{equation}	\label{def:Abar}
\Abar = \frac{1}{\sqrt{d\left(1-d/n\right)}} A.
\end{equation}
The main result of this paper is the following:

\vspace{.1cm}
\begin{theorem}[Circular Law for random regular digraphs]	\label{thm:main}
Assume $d=d(n)$ satisfies $\min(d,n-d)\ge \log^{C_0}n$ for a sufficiently large constant $C_0>0$.
For each $n\ge1$ let $A_n$ be a uniform random element of $\mA_{n,d}$. Then
$\mu_{\Abar_n} \to \muc$ in probability.
\end{theorem}
\vspace{.01cm}

\begin{remark}	\label{rmk:C0}
The proof shows we can take $C_0=96$, but we have not tried to optimize this constant. Various parts of the argument work for smaller degree, and for the interested reader we state the required range for $d$ in the statements of the lemmas, sometimes indicating how the range might be improved by longer arguments.
\end{remark}

\begin{remark}
The methods in this paper can also be used to prove the Circular Law for random regular digraphs with random edge weights, extending the result of \cite{Cook:circpm} to the sparse setting (in fact this is somewhat easier than Theorem \ref{thm:main}). Specifically, letting $X_n$ be an iid matrix as in Theorem \ref{thm:circ} with entries having finite fourth moment, and putting $Y_n = \frac1{\sqrt{d}}A_n\circ X_n$, it can be shown that $\mu_{Y_n}\to \muc$ in probability if $\min(d,n-d)\ge \log^{C_0}n$.  We do not include the proof in order to keep the article of reasonable length.
\end{remark}

\begin{remark}[Reduction to $d\le n/2$]	\label{rmk:reduced}
From \eqref{def:And} we have that $\1$ is an eigenvector of $A_n$ with eigenvalue $d$ (this is the Perron--Frobenius eigenvalue).
A routine calculation shows that if $\lambda\in \C\setminus\{d-n\}$ is an eigenvalue of $A_n$ with nonzero eigenvector $v\in \C^n$, then putting $A_n':=\1\1^\tran -A_n$ we have $A_n'w=-\lambda w$, where
\[
w=w(\lambda,v): =
v- \frac{\langle v, \1\rangle}{\lambda+n-d} \1. 
\]
Note that $w(\lambda,v)=0$ only if $\lambda=d$ and $v\in \langle \1\rangle$.
Thus, each eigenvalue of $-A_n$ (counting multiplicity) is an eigenvalue of $A_n'$, with at most one exception, and conversely. In particular,
$\mu_{\overline{A'}_n}$ and the reflected ESD $
\mu_{-\Abar_n}=\mu_{\Abar_n}(-\cdot)$ differ in total variation distance by at most $2/n$. Since $\muc$ is invariant under the refection $\lambda\mapsto -\lambda$, in the proof of Theorem \ref{thm:main} we may and will assume that $d\le n/2$, as we can replace $A_n$ with $A_n'$ if necessary.
\end{remark}

We conjecture that Theorem \ref{thm:main} still holds if $\min(d,n-d)$ tends to infinity with $n$ at any speed. For fixed degree we have the following well-known conjecture.

\begin{conjecture}[\cite{BoCh:survey}]	\label{conj:fixedd}
Fix $d\ge 3$ and let $A^{(d)}_{n}\in \mA_{n,d}$ be drawn uniformly at random. Then $\mu_{A^{(d)}_{n}}\to \mu^{(d)}_{\text{KM}}$ in probability, where $\mu^{(d)}_{\text{KM}}$ is the \emph{oriented Kesten--McKay law} on $\C$ with density
\begin{equation}	\label{def:OKM}
f^{(d)}_{\text{KM}}(z) = \frac1\pi \frac{d^2(d-1)}{(d^2-|z|^2)^2} 1_{\{|z|\le \sqrt{d}\}}
\end{equation}
with respect to Lebesgue measure. 
\end{conjecture}

For some numerical evidence supporting this conjecture the reader is referred to \cite{Cook:circpm}.
The measure \eqref{def:OKM} is the Brown measure for the free sum of $d$ Haar unitary operators; see \cite[Example 5.5]{HaLa}. Basak and Dembo established the conclusion of Conjecture \ref{conj:fixedd} with $A^{(d)}_{n}$ replaced by the sum of $d$ independent Haar unitary or orthogonal matrices \cite{BaDe}. Conjecture \ref{conj:fixedd} would follow from an extension of their proof to the sum $S_n^{(d)}$ of $d$ independent Haar \emph{permutation} matrices. Indeed, note that $S_n^{(d)}$ is a random element of $\mM_{n,d}$, the set of adjacency matrices for $d$-regular directed \emph{multi-}graphs on $n$ vertices. 
While the law of $S_n^{(d)}$ and the uniform distribution on $\mA_{n,d}\subset \mM_{n,d}$ are different measures,  \emph{contiguity} results \cite{Janson:contiguity} state any sequence of events $\event_n\subset \mM_{n,d}$ with probability $o(1)$ under the former distribution must also have probability $o(1)$ under the latter, provided $d$ is fixed independent of $n$. This allows one to deduce asymptotic results for $A_{n}^{(d)}$ from results for $S_n^{(d)}$.

A proof of Conjecture \ref{conj:fixedd} would require a significantly different approach than the one we take to prove Theorem \ref{thm:main}. For instance, to prove Theorem \ref{thm:main} we will need to understand the asymptotic empirical distribution of singular values $\mu_{\sqrt{(\Abar_n-z)^*(\Abar_n-z)}}$ for arbitrary fixed $z\in \C$ (see Section \ref{sec:highlevel} for additional explanation). In the present work we do this by comparing with a Gaussian matrix, for which the asymptotics are well-understood. However, when $\Abar_n$ is replaced by $A_n^{(d)}$ the conjectured asymptotic singular value distributions are different, and in particular we cannot compare with a Gaussian matrix or any other well-understood model.

We mention that in the recent work \cite{BCZ} with Basak and Zeitouni we established the Circular Law for the permutation model $S_n^{(d)}$ (rescaled by $\sqrt{d}$) under the assumption that $d$ grows poly-logarithmically as in the present work. Parts of the proof in \cite{BCZ} follow a significantly different approach from the present paper. In particular, in the present work the singular value distributions of $\Abar_n-z$ for $z\in \C$ are analyzed by first replacing $A_n$ with an iid Bernoulli matrix $B_n$ using a lower bound for the number of 0--1 matrices with constrained row and column sums, and then replacing the Bernoulli matrix with a Gaussian matrix using a Lindeberg exchange-type argument (see Section \ref{sec:compare.proofs}). Such a comparison is unavailable for the sum of permutation matrices. In \cite{BCZ} we derive and analyze the \emph{Schwinger--Dyson loop equations} for the Stieltjes transform of empirical singular value distributions, implementing a discrete analogue of techniques that were used in \cite{GKZ,BaDe} for the unitary group.

\subsection{The smallest singular value}

A key challenge for proving convergence of the ESDs of non-normal random matrices is to deal with possible \emph{spectral instability} for such matrices. This can be quantified in terms of the pseudospectrum.
Recall that the $\eps$-pseudospectrum of a matrix $M_n\in \mM_n(\C)$ is the set
\[
\Lambda_\eps(M_n) = \Lambda(M_n)\cup \left\{ z\in \C\setminus \Lambda(M_n): \|(M_n-z)^{-1}\| \ge \eps^{-1} \right\}
\]
where $\Lambda(M_n)$ is the set of eigenvalues of $M_n$.
(Here and throughout, $\|\cdot \|$ denote the operator norm when applied to matrices.)
If the pseudospectrum of a matrix is much larger than the spectrum itself, then the ESD can vary wildly under perturbations of small norm. 
For a random matrix $M_n$ the pseudospectrum is a random subset of the complex plane, and we will need it to be small in the sense that
\begin{equation}	\label{pseudo}
\pro{ z\in \Lambda_\eps(M_n)} = o(1) \qquad \text{ for $a.e.$ } z\in \C
\end{equation}
for some $\eps\ge \expo{ -n^{o(1)}}$. (Here the rate of convergence in the $o(1)$ terms may depend on $z$.)
Establishing \eqref{pseudo} is a key step
in all known approaches to proving the Circular Law for a given random matrix ensemble $(M_n)_{n\ge 1}$ (except in the case of integrable models such as the Ginibre ensemble).
See the survey \cite{BoCh:survey} for additional discussion of the pseudospectrum and its role in proving the Circular Law for random matrices.

Denote the singular values of a matrix $M$ by $s_1(M)\ge \cdots \ge s_n(M)\ge 0$. 
We can alternatively express our goal \eqref{pseudo} as showing that for $a.e.$ $z\in \C$, 
\begin{equation}	\label{sn}
\pro{ s_n(M_n-z) \le \eps} = o(1).
\end{equation}
Establishing \eqref{sn} is an extension of the \emph{invertibility problem}, which is to show 
\begin{equation}	\label{singprob}
\pro{ s_n(M_n) = 0} =\pro{ \det(M_n)=0} = o(1).
\end{equation}
The problem of proving \eqref{singprob}, along with quantifying the rate of convergence, has received much attention for the case of random matrices with discrete distribution, such as matrices with iid uniform $\pm1$ entries -- see \cite{Komlos67, KKS, TaVu:sing, BVW}.

The problem of proving \eqref{sn} with $M_n=X_n$ as in Theorem \ref{thm:circ} (i.e.\ having iid entries with zero mean and unit variance) was addressed in works of Rudelson \cite{Rudelson:inv}, Tao--Vu \cite{TaVu:cond, TaVu:circ} and Rudelson--Vershynin \cite{RuVe:ilo}. In particular, through new advances in the Littlewood--Offord theory from additive combinatorics, for certain random discrete matrices \cite{TaVu:cond} established bounds of the form $\pr(s_n(X_n)\le n^{-A}) = O(n^{-B})$ for arbitrary $B>0$ and $A=O_B(1)$. This result was extended in \cite{TaVu:circ} to allow general entry distributions with finite second moment and deterministic perturbations (such as scalar perturbations as in \eqref{sn}). 
The work \cite{RuVe:ilo} obtained the optimal dependence $A=B+1/2$ under a stronger subgaussian hypothesis for the entry distributions. Recently, this optimal bound was obtained for centered real iid matrices only assuming finite second moment \cite{ReTi}.

The invertibility problem for adjacency matrices of random regular digraphs $A_n$ as in Theorem \ref{thm:main} was first addressed in \cite{Cook:sing}, where it was shown that if $\min(d,n-d)\ge C\log^2n$, then
\begin{equation}	\label{cook:sing}
\pro{ s_n(A_n) = 0 } = O(d^{-c})
\end{equation}
for some absolute constants $C,c>0$. The main difficulties in proving \eqref{cook:sing} over the case of, say, iid $\pm1$ matrices are the lack of independence among entries and the sparsity of the matrix. 
The author introduced an approach based on a combination of strong graph regularity properties and the method of \emph{switchings}.

In its most basic form, one performs a \emph{simple switching} on a regular digraph by replacing directed edges $i_1\to j_1, i_2\to j_2$ with edges $i_1\to j_2, i_2\to j_1$ when this is allowed (i.e.\ when this does not create parallel edges); see Figure \ref{fig:switching}. The switching preserves the degrees of all vertices. One can create coupled pairs $(A_n,\tA_n)$ of random elements of $\mA_{n,d}$ by first drawing $A_n$ uniformly at random, and then applying several switchings at different $2\times 2$ submatrices of $A_n$ independently at random to form $\tA_n$. Taking care to do this in a way that $\tA_n$ is also uniformly distributed, one can condition on $A_n$ (perhaps restricted to a ``good" event on which $A_n$ enjoys certain graph regularity properties) and proceed using only the randomness of the independent switchings. In particular one gains access to tools of Littlewood--Offord theory. See \cite{Cook:sing} for additional motivation of the switchings method for the invertibility problem.

\begin{figure}
\centering
  \includegraphics[width=100mm]{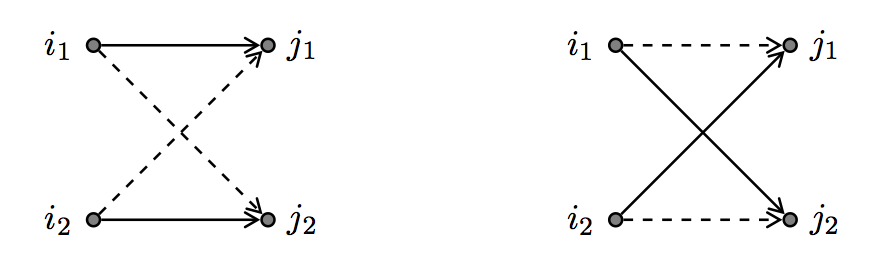}

\caption{Two possible configurations of directed edges passing from a pair of vertices $\{i_1,i_2\}$ to a pair of vertices $\{j_1,j_2\}$ in a digraph, where a dashed arrow indicates the \emph{absense} of a directed edge. Note that there may be edges passing from $\{j_1,j_2\}$ to $\{i_1,i_2\}$, or between $i_1$ and $i_2$ or $j_1$ and $j_2$. A simple switching replaces the configuration on the left with the one on the right and vice versa; for any other configuration the simple switching leaves the graph unchanged. }

\label{fig:switching}
\end{figure}

The switchings method has long been a popular tool for analyzing random regular graphs -- for additional background see the survey \cite{Wormald}. It has also recently been applied in the random matrix setting in \cite{BKY:semicircle, BHKY:rrg, BHY:km} to prove local laws for the empirical spectral distribution and universality of local spectral statistics for undirected random regular graphs.

Following the work \cite{Cook:sing}, it was shown in \cite{LLTTY} that for $C\le d\le cn/\log^2n$ we have
\begin{equation}	\label{LLTTY:bound}
\pro{ s_n(A_n)=0} \ll \frac{\log^3d}{\sqrt{d}}
\end{equation}
for some absolute constants $C,c>0$. Together with \eqref{cook:sing} this shows that $A_n$ is invertible with probability $1-o(1)$ as soon as $\min(d,n-d) = \omega(1)$. The work \cite{LLTTY} also follows the approach of using random switchings and graph regularity properties. The key new ingredients are finer regularity properties that apply for smaller degree, as well as an efficient averaging argument to improve the probability bound. (We make use of a variant of this averaging argument in the proof of Theorem \ref{thm:ssv} below.)
A natural conjecture is that
$
\pro{ s_n(A_n)= 0} =o(1)
$
for all $3\le d\le n-3$, which mirrors a conjecture by Vu for adjacency matrices of undirected random regular graphs \cite{Vu:discrete}.

In the present work we extend the approaches of \cite{Cook:sing, LLTTY} to obtain lower tail bounds on the smallest singular value of $\Abar_n - z$ for arbitrary scalar shifts $z\in \C$. It turns out that we can handle a more general class of perturbations; to describe them we need some notation.
First, note that $\1$ is a left and right eigenvector of $A_n$ with eigenvalue $d$. By a standard argument using the Cauchy--Schwarz inequality we have $\|A_n\|\le  d$,  so that in fact
\begin{equation}	\label{pf}
A_n\1 = A_n^* \1 =\|A_n\|\1 = d\1.
\end{equation}
We will be able to handle perturbations $Z\in \mM_n(\C)$ that also preserve the space $\1$, and which have polynomially-bounded norm on $\langle \1\rangle^\perp$. 
For instance, $Z$ could be the adjacency matrix of another regular digraph, either fixed or random and independent of $A_n$.
For a subspace $W\subset \C^n$ we write
\begin{equation}	\label{def:oneperp}
\|M\|_W := \|M:W\to \C^n\|=\sup_{u\in \sph\cap W} \|Mu\|.
\end{equation}

In the statement and proof of the following result we regard $n$ as a sufficiently large fixed integer and give quantitative bounds. As such, we will generally suppress the subscript $n$ from our matrices.

\begin{theorem}[The smallest singular value]		\label{thm:ssv}
Let $1\le d\le n/2$ and let $A$ be a uniform random element of $\mA_{n,d}$.
Fix $\gamma\ge 1$ and let $Z$ be a deterministic $n\times n$ matrix with $\|Z\|_{\oneperp}\le n^\gamma$ and such that $Z\1=\zeta\1$, $Z^*\1=\overline{\zeta}\1$ for some $\zeta\in \C$ with $|d+\zeta|\ge n^{-10}$.
There exists $\Gamma =O( \gamma \log_dn)$ such that
\begin{equation}	\label{ssv:bd}
\pro{ \,s_n(A+Z)\le n^{-\Gamma}\,}	
\ll_{\gamma} \frac{\log^{C_1}n}{\sqrt{d}}
\end{equation}
for some absolute constant $C_1>0$.
\end{theorem}

\begin{remark}	\label{rmk:C1}
We note that \eqref{ssv:bd} only gives a nontrivial bound if $d\ge C\log^{2C_1}n$ for some $C=C(\gamma)$ sufficiently large.
The proof shows we can take $C_1=11/2$, though there is certainly room for improvement. For instance, this exponent could be lowered using some refined graph expansion and discrepancy lemmas -- see Remark \ref{rmk:expand}. 
\end{remark}

For our purposes of proving Theorem \ref{thm:main} we only need the following consequence (recall the notation \eqref{def:Abar}):

\begin{cor}	\label{cor:ssv}
Assume $\log^{C_1'}n\le d\le n/2$ for a sufficiently large constant $C_1'>0$, and let $A$ be a uniform random element of $\mA_{n,d}$.
Fix $z\in \C$. 
There exists $\Gamma=o( \log n)$ such that
\begin{equation}
\pro{ s_n\left( \Abar_n -z\right) \le n^{-\Gamma}} =o_z(1).
\end{equation}
\end{cor}

\begin{proof}
We may assume $n$ is sufficiently large depending on $z$.
Up to perturbing $\Gamma$ by a constant factor, it suffices to verify the matrix
\[
Z= -z\sqrt{d(1-d/n)} \I_n
\]
satisfies the conditions of Theorem \ref{thm:ssv}. The condition $\|Z\|_{\oneperp}\le n^\gamma$ holds with $\gamma=0.51$, say, when $n$ is sufficiently large. Taking $\zeta= z\sqrt{d(1-d/n)}$ we have
\[
|d-\zeta| \ge d - |z|\sqrt{d} = d(1-o_z(1)) \to \infty
\]
and the condition on $\zeta$ easily holds when $n$ is sufficiently large. The result now follows from Theorem \ref{thm:ssv} and taking $C_1'=2C_1+1$, say.
\end{proof}

Recently (a few months after this paper was first posted to arXiv) \cite{LLTTY17} obtained an improvement of Theorem \ref{thm:ssv} (for the case of scalar shifts), showing that for some constants $C,c>0$ and any fixed $z\in \C$ with $|z|\le d/6$, if $C\le d\le cn/(\log n)(\log\log n)$ then $s_n(A_n-z)\ge n^{-6}$ with probability $1-O(\log^2d/\sqrt{d})$.

\subsection{Overview of the paper}	\label{sec:overview}

The first part of the paper (Sections \ref{sec:graph}--\ref{sec:unstructured}) is devoted to the proof of Theorem \ref{thm:ssv}.
In Section \ref{sec:graph} we recall some concentration inequalities for random regular digraphs from \cite{Cook:discrep} and use these to show that a random element of $\mA_{n,d}$ satisfies certain graph regularity properties with high probability.
In Section \ref{sec:partition} we describe the general approach to Theorem \ref{thm:ssv}, which proceeds by partitioning the sphere $\sph_0$ into sets whose elements have a similar level of ``structure" (in a certain precise sense that we do not describe here), and separately controlling $\inf_{v\in S} \|(A+Z)v\|$ for each part $S$ of the partition. We then establish bounds on covering numbers for sets of highly structured vectors, and prove anti-concentration properties for unstructured vectors. 
In Section \ref{sec:high} we establish uniform control from below on $\|(A+Z)v\|$ for ``highly structured" vectors $v$, and in Section \ref{sec:increment} we boost this to control for less structured vectors by an iterative argument. In Section \ref{sec:unstructured} we obtain control over the remaining unstructured vectors. We mention that in each of Sections \ref{sec:high}, \ref{sec:increment} and \ref{sec:unstructured} we make use of a different graph regularity property from Section \ref{sec:graph}, and all three sections use coupling arguments based on switchings.

In the remainder of the paper we prove Theorem \ref{thm:main}.
In Section \ref{sec:highlevel} we recall the approach to proving the Circular Law via the logarithmic potential, and give a high-level proof of Theorem \ref{thm:main} using Theorem \ref{thm:ssv} and two propositions concerning the empirical singular value distributions for certain perturbations of $\Abar_n$.
In Sections \ref{sec:comparison} and \ref{sec:compare.proofs} we prove these propositions by a two-step comparison approach, first comparing $A_n$ with a matrix $B_n$ having iid Bernoulli entries, and then comparing $B_n$ (suitably centered and rescaled) with an iid Gaussian matrix $G_n$, for which the desired results are known.
The comparison of singular value distributions for $A_n$ with those of $B_n$ is accomplished using a conditioning argument of Tran, Vu and Wang from \cite{TVW}, together with a new estimate for the probability that $B_n$ lies in $\mA_{n,d}$, proved in Appendix \ref{app:band}.
For the comparison between $B_n$ and $G_n$ we use the Lindeberg replacement strategy, through an invariance principle of Chatterjee (Theorem \ref{thm:invar}). 
In the appendix we prove Lemma \ref{lem:wegner.gaussian}, which gives a near-optimal estimate on local density of small singular values for perturbed Gaussian matrices.

\subsection{Notation}	\label{sec:notation}

%Asymptotic notation
$C,c,c',c_0$, etc.\ denote unspecified constants whose value may change from line to line, understood to be absolute unless otherwise stated.
$f=O(g)$, $f\ll g$ and $g\gg f$ are synonymous and mean that $|f|\le Cg$ for some absolute constant $C<\infty$. $f\asymp g$ means $f\ll g$ and $g\ll f$. $f=o(g)$ and $g=\omega(f)$ mean that $f/g\to 0$ as $n\to \infty$.
We indicate dependence of implied constants with subscripts, e.g.\ $f\ll_\alpha g$; by $f=o_\alpha(g)$ we mean $f/g\to 0$ when $\alpha$ is fixed and $n\to \infty$, where the rate of convergence may depend on $\alpha$.

%Matrices
$\mM_{n}(\C)$ denotes the set of $n\times n$ matrices with complex entries. 
For $M=(m_{ij})\in \mM_n(\C)$ it will sometimes be convenient to denote the $(i,j)$-th entry by $M(i,j)=m_{ij}$.
For $i_1,\dots i_k,j_1,\dots, j_l\in [n]$ we write
\begin{equation}
M_{(i_1,\dots, i_k)\times (j_1,\dots, j_l)} := (m_{i_{k'} j_{l'}})_{k'\in [k], l'\in [l]}.
\end{equation}
If one of the sequences $(i_1,\dots, i_k)$, $(j_1,\dots, j_l)$ is replaced by an unordered set $J\subset[n]$ then we interpret $J$ as a sequence with the natural ordering inherited from $[n]$.
We also write $M^{(i_1,i_2)}$ for the $(n-2)\times n$ matrix obtained by removing rows $i_1$ and $i_2$ (assuming $i_1\ne i_2$).
%spectral and singular value distributions
We label the singular values of $M$ in non-increasing order: 
\[
s_1(M)\ge \cdots\ge s_n(M)\ge 0.
\] 
In addition to our notation \eqref{def:esd} for the empirical spectral distribution, we denote the empirical singular value distribution by 
\begin{equation}	\label{def:esvd}
\nu_M:= \frac1n\sum_{i=1}^n \delta_{s_i(M)}.
\end{equation}

%Norms, subspaces, projections, balls and spheres
$\|\cdot\|$ denotes the Euclidean norm when applied to vectors and the $\ell_2^n\to \ell_2^n$ operator norm when applied to elements of $\mM_n(\C)$. Other norms are indicated with subscripts; in particular, $\|M\|_{\HS}$ denotes the Hilbert--Schmidt (or Frobenius) norm of a matrix $M$. 
We denote the (Euclidean) closed unit ball in $\C^n$ by $\ball$ and the unit sphere by $\sph$. We write $\C^J$ for the subspace of vectors supported on $J\subset[n]$, and write $\mathbb{B}^J, \mathbb{S}^J$ for the unit ball and sphere in this subspace.
Given $v\in \C^n$ and $J\subset[n]$, $v_J$ denotes the projection of $v$ to $\C^J$.
$\1=\1_n$ denotes the $n$-dimensional vector with all components equal to one, and consequently $\1_J$ denotes the vector with $j$th component equal to 1 for $j\in J$ and 0 otherwise. 
We will frequently consider the unit sphere in $\langle \1\rangle^\perp$, which we denote
\begin{equation}
\sph_0:= \sph\cap \oneperp = \big\{u\in \C^n: \|u\|=1, \; \langle u, \1\rangle =0\big\}.
\end{equation}

%Graph theoretic notation
It will be conceptually helpful to associate a 0--1 $n\times n$ matrix $A=(a_{ij})$ to a directed graph $\Gamma_A=([n],E_A)$, which we do in the natural way, i.e.\ $E_A=\{(i,j)\in [n]^2: a_{ij}=1\}$.
Given a vertex $i\in [n]$ we denote its set of out-neighbors by
\begin{equation}	\label{def:nbhd}
\mN_A(i):= \{ j\in [n]: a_{ij}=1\}.
\end{equation}
Its set of in-neighbors is consequently given by $\mN_{A^\tran}(i)$. 
For $i\in [n]$ and $L\subset[n]$ we sometimes abbreviate
\begin{equation}	\label{def:LA}
L_A(i) := \mN_A(i) \cap L.
\end{equation} 
We denote the out-neighorhood of a set $I\subset[n]$ by
\begin{equation}	
\mN_A(I):= \bigcup_{i\in I} \mN_A(i).
\end{equation}
Given $I,J\subset[n]$, we denote by
\begin{equation}	\label{def:edge}
e_A(I,J) = \sum_{i\in I, j\in J} a_{ij}
\end{equation}
the number of directed edges which start in $I$ and end in $J$.
For sets $J\subset[n]$ we will frequently abbreviate $J^c:= [n]\setminus J$.

\subsection{Acknowledgement}

The author thanks Terry Tao for his encouragement and support.

\section{Graph regularity properties}	\label{sec:graph}

Recall the graph theoretic notation from Section \ref{sec:notation}.
In this section we define three collections of ``good" subsets of $\mA_{n,d}$, namely
\[
\mA^{\codeg}(i_1,i_2)\,,\; \mA^{\edge}(n_0,\delta)\,,\; \text{and } \mA^{\expand}(\kappa)
\]
whose elements are associated to digraphs enjoying certain graph regularity properties. 
We will show that for appropriate values of the parameters, each of these sets constitutes most of $\mA_{n,d}$.
The key tools to establish this are sharp tail bounds for codegrees and edge densities for random regular digraphs that were proved in \cite{Cook:discrep}.

For $A\in \mA_{n,d}$, the number of common out-neighbors $|\mN_A(i_1)\cap \mN_A(i_2)|$ of a pair of vertices $i_1,i_2\in [n]$ in the associated digraph is called the \emph{out-codegree} of $i_1,i_2$. By a routine calculation, for a fixed pair $\{i_1,i_2\}\subset [n]$ and $A\in \mA_{n,d}$ drawn uniformly at random we have
\[
\e |\mN_A(i_1)\cap \mN_A(i_2)| = \frac{d(d-1)}{n-1}\approx \frac{d^2}{n}.
\]
In the proof of Theorem \ref{thm:ssv} we will want to restrict attention to those $A\in \mA_{n,d}$ whose out-codegree at a fixed pair of vertices is not too large. 
For distinct $i_1,i_2\in [n]$ and $K>0$ define the set of elements of $\mA_{n,d}$ having \emph{good codegrees}
\begin{equation}	\label{def:goodcodeg}
\mA^{\codeg}(i_1,i_2) := \set{ A\in \mA_{n,d}: \; \big| \mN_A(i_1)\cap \mN_A(i_2)\big| \le d/4}.
\end{equation}

\begin{lemma}[Control on codegrees, cf.\ {\cite[Proposition 4.1]{Cook:discrep}}]	\label{lem:codeg} 
Let $1\le d\le n$ and let $A\in \mA_{n,d}$ be drawn uniformly at random. For any distinct $i_1,i_2\in [n]$ and $K>0$,
\begin{equation}
\pro{ \big| \mN_A(i_1)\cap \mN_A(i_2)\big| \ge (1+K)\frac{d^2}{n}} \le\expo{ -\frac{K^2}{4+ 2K}\frac{d^2}{n}}.
\end{equation}
In particular, taking $K$ to be a sufficiently large constant multiple of $n/d$, we have
\begin{equation}
\pro{ A\notin \mA^{\codeg}(i_1,i_2) } \le e^{-cd}
\end{equation}
for some constant $c>0$.
\end{lemma}

For fixed sets $I,J\subset[n]$ and $A\in \mA_{n,d}$ drawn uniformly at random, the expected number of directed edges passing from $I$ to $J$ in the digraph associated to $A$ is
\[
\e e_A(I,J) = \frac{d}{n} |I||J|.
\]
In the random graphs literature, a graph for which the edge densities $e_A(I,J)/|I||J|$ (for all sufficiently large sets $I,J$) do not deviate too much from the overall density $d/n=e_A([n],[n])/n^2 $ is said to satisfy a \emph{discrepancy property}. 
For $n_0\in [n]$ and $\delta>0$ we define the set of elements of $\mA_{n,d}$ enjoying a discrepancy property:
\begin{equation}	\label{event:discrep}
\mA^{\edge}(n_0,\delta) = \bigcap_{\substack{I,J\subset[n]:\\ |I|,|J|>n_0}} \left\{A\in \mA_{n,d}:\ \left|e_A(I,J) - \frac{d}{n}|I||J| \right| < \delta \frac{d}{n}|I||J|\right\}.
\end{equation}

The following is an easy corollary of the main result in \cite{Cook:discrep}.

\begin{lemma}[Discrepancy property]		\label{lem:discrep}
Assume $1\le d\le n/2$. 
Let $\delta\in (0,1)$. If
$
(C/\delta)nd^{-1/2}\le n_0\le n
$
for a sufficiently large constant $C>0$, then for a uniform random element $A\in \mA_{n,d}$ we have
\[
\pro{ A\in\mA^{\edge}(n_0,\delta)} = 1- n^{O(1)} \expo{ -c\delta \min(d, \delta n)}.
\]

\end{lemma}

\begin{proof}
By \cite[Theorem 1.5]{Cook:discrep}, there is a set $\good_0\subset\mA_{n,d}$ with $\pr(A\in \good_0) = 1-n^{O(1)} \exp(-c\delta \min(d,\delta n))$ such that for any fixed $I,J\subset[n]$, 
\begin{equation}
\pro{ A\in \good_0, \; \left|e_A(I,J) - \frac{d}{n}|I||J|\right| \ge  \delta \frac{d}{n}|I||J|} \le 2\expo{ -c\delta^2 \frac{d}{n}|I||J|}.
\end{equation}
Let $n_0$ be as in the statement of the lemma.
Applying the union bound over choices of $I,J$ with $|I|,|J|>n_0$, 
\begin{align*}
\pro{ A\in \good_0 \cap \mA^{\edge}(n_0,\delta)^c}\le 4^n\times 2\expo{ -c\delta^2 \frac{d}{n}n_0^2}\le 2^{-n}
\end{align*}
where in the last bound we took the constant $C$ in the lower bound on $n_0$ sufficiently large. 
Thus,
\[
\pro{ A\in \mA^{\edge}(n_0,\delta)} \ge 1-\pro{ A\notin\good_0} - 2^{-n} = 1-n^{O(1)} \exp(-c\delta \min(d,\delta n))
\]
as desired.
\end{proof}

Finally, we will need to show that most elements of $\mA_{n,d}$ satisfy a certain neighborhood expansion property. By the $d$-regularity constraint, for any $I\subset[n]$ and $A\subset \mA_{n,d}$ we have
\[
|\mN_A(I)| \le d|I|. 
\]
It turns out that for random regular digraphs and $|I|<n/d$, this upper bound is not far from the truth. 
For $\kappa\in (0,1)$ we define the set of elements of $\mA_{n,d}$ enjoying the ``good expansion property":
\begin{equation}	\label{def:goodexp}
\mA^{\expand}(\kappa) = \bigcap_{\substack{J\subset [n]:\\ |J|\le  n/2\kappa d }}\Big\{ A\in \mA_{n,d}: \ \big|\mN_{A^\tran}(J)\big| > \kappa d|J| \Big\}.
\end{equation}

\begin{lemma}[Expansion property]	\label{lem:expand}
There are absolute constants $C,c>0$ such that if $C\log n\le d\le n/2$, then
\[
\pro{ A\in \mA^{\expand}\left( \frac{c}{\log n}\right)} = 1- O(e^{-cd}).
\]
\end{lemma}

\begin{remark}	\label{rmk:expand}
While the above will be sufficient for our purposes, we note that Litvak et al.\ obtained a stronger ``log-free" version in \cite{LLTTY} (see Theorem 2.2 there), showing that with high probability one has $|\mN_{A^\tran}(J)|\ge \kappa d|J|$ with $\kappa$ arbitrarily close to one, uniformly over $|J|\le c\kappa n/d$ (in fact they allow $\kappa\to 1$ at a certain rate with $d$). Moreover, their result holds for all $d$ at least a sufficiently large constant. 
Using their result in place of Lemma \ref{lem:expand} would lower the power of log by 2 in our assumption on $d$ in Theorem \ref{thm:ssv}. 
\end{remark}

\begin{proof}
This is essentially a restatement of \cite[Corollary 3.7]{Cook:sing}, taking the parameter $\gamma$ there to be $\kappa\log n$. While it was assumed there that $d=\omega(\log n)$, the proof actually only assumes $d\ge C\log n$ for a sufficiently large constant $C>0$.
\end{proof}

By definition, for $A\in \mA^{\expand}(\kappa)$ and a sufficiently small set $J\subset[n]$, the number of rows of $A$ whose support overlaps with $J$ is within a factor $\kappa$ of its maximum value $d|J|$.
However, we will also need lower bounds on the number of rows whose overlap with $J$ has cardinality within a specified range.  
For $J\subset [n]$ and $r\ge 1$ write 
\begin{align*}
\mN_{A^\tran}^{\le r}(J) &= \{i\in [n]: 1\le |\mN_A(i)\cap J|\le r\},\\
\mN_{A^\tran}^{\ge r}(J) &= \{i\in [n]: |\mN_A(i)\cap J|\ge r\},
\end{align*}
and similarly define $\mN_{A^\tran}^{< r}(J), \mN_{A^\tran}^{> r}(J)$.

\begin{lemma}		\label{lem:goodexpand}
Let $1\le d\le n/2$, $\kappa\in(0,1)$ and $A\in\mA^{\expand}(\kappa)$. Then the following hold:
\begin{enumerate}
\item For all $J\subset[n]$ such that $|J|\le n/2\kappa d$,
\begin{equation}	\label{bd:goodexp1}
\big|\mN_{A^\tran}^{\le r}(J)\big|\ge\left( \kappa - \frac1{r+1}\right) d|J|
\end{equation}
for all $r\ge 1$.
\item For all $J\subset[n]$ such that $|J|>  n/2\kappa d$,
\begin{equation}	\label{bd:goodexp2}
\big|\mN_{A^\tran}^{> r}(J)\big|>n/8
\end{equation}
for all $1\le r\le \kappa d|J|/4n$.
\end{enumerate}
\end{lemma}

\begin{proof}
We begin with (1). 
Fix such a set $J$ and let $r\ge 1$.
We have
\begin{align*}
d|J| &= e_A(\mN_{A^\tran}(J),J) \\
&= \sum_{i\in \mN_{A^\tran}^{\le r}(J)} |\mN_A(i)\cap J| +  \sum_{i\in \mN_{A^\tran}^{> r}(J)} |\mN_A(i)\cap J| \\
&\ge (r+1) \big|\mN_{A^\tran}^{> r}(J)\big| \\
&\ge (r+1) \left( \kappa d |J| - \big|\mN_{A^\tran}^{\le r}(J)\big|\right).
\end{align*}
and \eqref{bd:goodexp1} follows upon rearranging.

We turn to (2). 
Let $J$ and $r$ be as in the statement of the lemma.
Let $m\in (\frac{ n}{4\kappa d}, \frac{n}{2\kappa d}]$ be an integer, put $k=\lf |J|/m\rf$, and let $J_1,\dots,J_k$ be pairwise disjoint subsets of $J$ of size $m$.
Denote $J'=\bigcup_{l=1}^m J_l$, and note that $|J'|\ge |J|/2$.
By our restriction to $\mA^{\expand}(\kappa)$, for each $l\in [m]$ we have
\begin{equation}	\label{goodexp:above}
|\mN_{A^\tran}(J_l)|\ge \kappa dm.
\end{equation}
Let $B$ be the adjacency matrix of the bipartite graph with vertex parts $U=\mN_{A^\tran}(J')$, $V=[k]$ which puts an edge at $(i,l)$ when $|\mN_A(i)\cap J_l|\ge 1$.
From \eqref{goodexp:above} we have that the number of edges $e_B(U,V)$ in this graph is bounded below by $\kappa dmk$.
On the other hand, 
\begin{align*}
e_B(U,V)
&\le k|\{ i\in U: |\mN_B(i)|>r\}| + r|\{i\in U: 1\le |\mN_B(i)|\le r\}|\\
&\le k\big|\mN_{A^\tran}^{>r}(J)\big| + rn.
\end{align*}
Combining these bounds on $e_B(U,V)$ and rearranging, we have
\[
\big|\mN_{A^\tran}^{>r}(J)\big| \ge \kappa dm - \frac{rn}{k} \ge \kappa dm\left( 1 - \frac{|J|}{4km}\right)
\]
where in the second inequality we applied our assumption on $r$.
Now since $km=|J'|\ge |J|/2$ we conclude
\[
\big|\mN_{A^\tran}^{>r}(J)\big| \ge\frac12 \kappa dm >  \frac{n}{8}
\]
as desired.
\end{proof}

\section{Partitioning the sphere}	\label{sec:partition}

In this section we begin the proof of Theorem \ref{thm:ssv}. Throughout this section $A$ denotes a uniform random element of $\mA_{n,d}$.

\subsection{Structured and unstructured vectors}

A well-known approach to bounding the probability a random matrix $M$ is singular is to classify potential null vectors $v\ne 0$ as ``structured" or ``unstructured", and use different arguments to bound 
\[
\pr(\exists \text{ structured } v: Mv=0) \quad \text{and}\quad \pr(\exists \text{ unstructured } v: Mv=0).
\]
This approach goes back to the work of Koml\'os on iid Bernoulli matrices $X=(\xi_{ij})$, where an integer vector $v$ is said to be structured if it is \emph{sparse}, i.e.\ $|\supp(v)|\le n/10$, say. The key observation is that unstructured vectors $v$ enjoy good anti-concentration for random walks 
\begin{equation}	\label{randomwalk}
R_i\cdot v=\sum_{j=1}^n  \xi_{ij}v_j,
\end{equation}
in the sense that $R_i\cdot v$ is unlikely to be zero, where $R_i$ denotes the $i$th row of $X$. On the other hand, while structured vectors only have crude anti-concentration properties, the set of structured vectors has low entropy (i.e.\ cardinality), which allows one to obtain uniform control via the union bound. 
Later, in \cite{TaVu:sing} Tao and Vu used more complicated classifications of potential null vectors $v\in\Z^n\setminus \{0\}$ by relating concentration properties of the random walks \eqref{randomwalk} to arithmetic structure in the components of $v$ using tools from additive combinatorics (such as Freiman's theorem). 

This approach carries over to the problem of bounding the smallest singular value. From the variational formula
\[
s_n(M) = \inf_{u\in \sph} \|Mu\|, 
\]
if $S_1\cup\cdots \cup S_N$ is a partition of $\sph$, then
\begin{equation}	\label{bd:classify}
\pro{ s_n(M) \le \eps} \le \pro{ \inf_{u\in S_1} \|Mu\|\le \eps} + \cdots + \pro{ \inf_{u\in S_N} \|Mu\|\le \eps}.
\end{equation}
The analogue of Koml\'os's argument for the invertibility problem was accomplished by Rudelson for a general class of iid matrices in \cite{Rudelson:inv} (see also \cite{RuVe:ilo}). There a unit vector is said to be structured if it is \emph{close} to a sparse vector. Specifically, for $m\in[n]$ we denote the set of \emph{$m$-sparse vectors}
\begin{equation}	\label{def:sparse}
\Sparse(m) = \big\{ v\in \C^n: |\supp(v)|\le m\big\}
\end{equation}
where $\supp(v) = \{j\in [n]:v_j\ne 0\}$, 
and for $m\in [n],\rho\in (0,1)$ we define the set of \emph{compressible vectors}
\begin{equation}
\Comp(m,\rho) = \sph \cap \left( \Sparse(m) + \rho \ball\right),
\end{equation}
where we recall that $\ball$ denotes the closed unit ball in $\C^n$, so that $E+\rho \ball$ denotes the closed $\rho$-neighborhood of a set $E\subset \C^n$. In \cite{Rudelson:inv},
\eqref{bd:classify} is applied with $N=2$, taking $S_1$ to be the set of compressible vectors (with an appropriate choice of paramers) and $S_2$ the complementary set of ``incompressible" vectors.
As in the invertibility problem, incompressible vectors enjoy good anti-concentration properties for the associated random walks \eqref{randomwalk}, while the set of compressible vectors has low \emph{metric entropy}, which allows one to obtain uniform control using nets and the union bound.
Later works of Tao--Vu \cite{TaVu:cond, TaVu:circ} and Rudelson--Vershynin \cite{RuVe:ilo} used larger partitions based on arithmetic structural properties.

In the present work, the distribution of $A_n$ calls for a different notion of structure than those discussed above.
In the work \cite{Cook:sing} on the invertibility problem for $A_n$ an integer vector was said to be structured if it had a large level set. Thus, for controlling the smallest singular value, we consider a unit vector $u\in \sph$ to be structured if it is close to a vector with a large level set, and call such vectors ``flat". 
Alternatively, non-flat vectors are those $u\in \sph$ for which the empirical measure $\frac1n\sum_{i=1}^n \delta_{u_i}$ of components enjoys some anti-concentration estimate; this perspective will be expanded upon in Section \ref{sec:ac}.

Formally, for $m\in [n]$ and $\rho\in (0,1)$, define the set of \emph{$(m,\rho)$-flat vectors} 
\begin{align}
\Flat(m,\rho) &= \sph\cap \left(  \rho \ball + \bigcup_{\lambda\in \C} \big(\lambda\1+ \Sparse(m)\big) \right) \notag\\
&= \big\{ u\in \sph:  \exists\, v\in \Sparse(m),\, \lambda\in \C \text{ with } \|u-v-\lambda \1\| \le \rho \big\}	.\label{def:Aflat}
\end{align}
We denote the mean-zero flat vectors by
\begin{equation}	\label{def:Aflat0}
\Flat_0(m,\rho) =  \Flat(m,\rho)\cap \langle \1\rangle^\perp.
\end{equation}
For non-integral $x\ge0$ we will sometimes abuse notation and write $\Sparse(x)$, $\Flat(x,\rho)$, etc.\ to mean $\Sparse(\lf x\rf)$, $\Flat(\lf x\rf,\rho)$.

Now we state our main proposition controlling the event that $\|(A+Z)u\|$ is small for some structured vector $u$.
For $K\ge 1$ we denote the \emph{boundedness event}
\begin{equation}	\label{def:bdd}
\mB(K)= \left\{ \|A+Z\|_{\oneperp}\le K\sqrt{d}\right\}
\end{equation}
(recall the notation \eqref{def:oneperp}).
From \eqref{pf}, our assumptions on $Z$ and the triangle inequality, $\mB(K)$ holds with probability one for $K\sqrt{d}= d+n^\gamma$.
(Sharper bounds than this hold with high probability, but are not necessary for our purposes.) 
For much of the proof we will leave the parameter $K$ generic.
For $K\ge1$ and $m\in [n],\rho\in (0,1)$, denote
\begin{equation}	\label{def:event}
\event_K(m,\rho) = \mB(K)\wedge\Big\{\, \exists u\in \Flat_0(m,\rho): \|(A+Z)u\|\le \rho K\sqrt{d}\,\Big\}.
\end{equation}

\begin{prop}[Control on flat vectors]		\label{prop:flat}
Assume $\log^{4}n\le d\le n/2$ and $1\le K\le n^{\gamma_0}$ for some fixed $\gamma_0\ge1/2$.
There exists $\Gamma_0\ll \gamma_0\log_dn$ such that for all $n$ sufficiently large depending on $\gamma_0$,
\begin{equation}	\label{bd:propflat}
\pro{ \event_K\left(\frac{cn}{\gamma_0\log^3n}\,, \;n^{-\Gamma_0}\right)} \le e^{-cd}
\end{equation}
where $c>0$ is an absolute constant.
\end{prop}

We briefly outline some of the ideas of the proof.
As in prior works controlling invertibility over structured vectors, we will reduce to controlling the size of $\|(A+Z)u\|$ for $u$ ranging over a $\rho$-net for the set $\Flat_0(m,\rho)$ -- that is, a finite set $\Sigma_0(m,\rho)\subset \Flat_0(m,\rho)$ whose $\rho$-neighborhood contains $\Flat_0(m,\rho)$. 
The restriction to $\mB(K)$ allows us to argue
\[
\pro{ \event_K(m,\rho)} \le \pro{ \exists u\in \Sigma_0(m,\rho): \|(A+Z)u\|\le 2\rho K\sqrt{d}}.
\]
Applying the union bound,
\[
\pro{ \event_K(m,\rho)} \le |\Sigma_0(m,\rho)| \max_{u\in \Sigma_0(m,\rho)} \pro{  \|(A+Z)u\|\le 2\rho K\sqrt{d}}.
\]
In the next subsection we construct such $\rho$-nets of controlled cardinality. Our task is then to obtain a strong enough lower tail bound for $\|(A+Z)u\|$, holding uniformly over fixed $u\in \Sigma_0(m,\rho)$, to beat the cardinality of the net. 

It turns out that with no additional information on $u$ we can only beat the cardinality of the net when $m$ is fairly small (of size $m\ll d/\log n$). However, once we have shown $\event_K(m_0,\rho_0)$ is small for some $m_0,\rho_0$, then when trying to control vectors in a net $\Sigma_0(m,\rho)$ for $\Flat_0(m,\rho)$ with $m>m_0$, we can restrict the net to the complement of $\Flat_0(m_0,\rho_0)$:
\begin{align*}
\pr\Big( \event_K(m,\rho)\setminus \event_K(m_0,\rho_0)\Big) \le |\Sigma_0(m,\rho)| \max_{u\in \Sigma_0(m,\rho)\setminus \Flat_0(m_0,\rho_0)} \pro{  \|(A+Z)u\|\le 2\rho K\sqrt{d}}.
\end{align*}
This additional information that $u\notin \Flat_0(m_0,\rho_0)$ allows us to get an improved lower tail bound for $\|(A+Z)u\|$, which beats the cardinality of the net $ |\Sigma_1(m,\rho)|$ for $m\ll (d/\log^{O(1)} n)m_0$. Assuming $d$ grows at an appropriate poly-logarithmic rate, we can iterate this argument along a sequence $(m_k,\rho_k)$ until $m_k\gg n/\log^3n$ as in \eqref{bd:propflat}. The values $\rho_k$ will degrade by a polynomial factor at each step, but this is acceptable for our purposes.

We mention that this iterative approach is similar to arguments from \cite{RuZe, Cook:ssv}, and to a lesser extent \cite{TaVu:cond,RuVe:ilo}. However, those works concerned matrices with independent entries; consequently, our proofs of lower tail bounds for $\|(A+Z)u\|$ with fixed $u$  are substantially different, making use of coupled pairs $(A,\tA)$ formed by applying random switching operations, and relying heavily on the graph regularity properties from Section \ref{sec:graph}.

We prove Proposition \ref{prop:flat} in Sections \ref{sec:high} and \ref{sec:increment}.
In the remainder of this section we develop some useful lemmas concerning flat and non-flat vectors.

\subsection{Metric entropy of flat vectors}	\label{sec:entropy}

In this section we bound the metric entropy of the sets $\Flat_0(m,\rho)$ -- that is, we find efficient coverings of these sets by Euclidean balls. 
The following is a standard fact on the existence of such coverings of controlled cardinality, and is established by a well-known volumetric argument; see for instance \cite{MiSc:notes}, \cite[Lemma 2.2]{Cook:ssv}.

\begin{lemma}		\label{lem:net}
Let $V\subset \C^n$ be a subspace of (complex) dimension $m$, and let $\rho\in (0,1)$. There exists a set $\Sigma_V(\rho)\subset V\cap \ball$ with $|\Sigma_V(\rho)| = O(1/\rho)^{2m}$ such that $\Sigma_V(\rho)$ is a $\rho$-net for $V\cap \ball$ -- i.e.\ for every $x\in V\cap \ball$ there exists $y\in \Sigma_V(\rho)$ such that $\|x-y\|\le \rho$.
\end{lemma}

Using this we can show:

\begin{lemma}[Metric entropy for flat vectors]		\label{lem:flatnet}
Let $1\le m\le n/10$ and $\rho\in (0,1)$. 
There exists $\Sigma_0 = \Sigma_0(m,\rho)\subset \Flat_0(m,\rho)$ such that $|\Sigma_0| = O\big( \frac{n}{m\rho^2}\big)^m$ and $\Sigma_0$ is a $\rho$-net for $\Flat_0(m,\rho)$.
\end{lemma}

\begin{proof}	
Note we may assume that $\rho$ is smaller than any fixed constant.
Consider an arbitrary element $u\in \Flat_0(m,\rho)$.
We may write 
\[
u = v + \lambda \1 + w
\]
for some $v\in \Sparse(m)$, $\lambda\in \C$ and $\|w\|\le \rho$. 

We begin by crudely bounding $\|v\|$ and $|\lambda|$. 
Suppose $v$ is supported on $J\subset[n]$, $|J|=m$.
By the triangle inequality, $\|v+\lambda\1\|\le 1+\rho$, and by Pythagoras's theorem
\begin{equation}	\label{net:above}
\|v + \lambda \1_J\|^2 + \|\lambda \1_{J^c} \|^2 \le (1+\rho)^2.
\end{equation}
Ignoring the first term gives
\begin{equation}	\label{net:lambda}
|\lambda| \le (1+\rho)\sqrt{\frac{1}{n-m}}\le \frac{2}{\sqrt{n}}
\end{equation}
by our bound on $m$ and assuming $\rho\le 1/2$.
Ignoring the second term in \eqref{net:above} and applying the triangle inequality and \eqref{net:lambda} yields
\begin{equation}	\label{net:v}
\|v\| \le 1+ \rho + \|\lambda \1_J\| \le (1+\rho) \left( 1+ \sqrt{\frac{m}{n-m}}\right) \le 2.
\end{equation}

Denote $\overline{v} := \frac1n \sum_{j=1}^n v_j$ and similarly for $w$. 
Since $u\in \sph_0$, we have
\[
0 = \overline{v} + \lambda + \overline{w}.
\]
Rearranging and applying Cauchy--Schwarz gives
\[
|\overline{v}+ \lambda| = |\overline{w}| \le \rho/\sqrt{n}.
\]
Thus we may alternatively express
\begin{align}
u &= v - \overline{v} \1 + \left[ (\overline{v} + \lambda)\1 + w\right] 	\notag\\
&= v - \overline{v} \1 + w'	\label{net:ualt}
\end{align}
with $\|w'\|\le 2\rho$. 

In the remainder of the proof we will obtain a $3\rho$-net for $\Flat_0(m,\rho)$ from a $\rho$-net for the projection to $\langle \1\rangle^\perp$ of $\Sparse(m) \cap 2\ball$. Then we will show that we can rescale the resulting vectors to lie in $\Flat_0(m,\rho)$ to obtain a $6\rho$-net. The result will then follow by replacing $\rho$ with $\rho/6$. 

By Lemma \ref{lem:net}, for each $J\subset[n]$ with $|J|=m$ there is a $\rho$-net $\Sigma_J(\rho)$ for $2\mathbb{B}^J$ of cardinality $O(1/\rho)^{2m}$ (by dilating a $\rho/2$-net for $\mathbb{B}^J$). Taking the union of these nets over all choices of $J\in {[n]\choose m}$ yields a $\rho$-net $\Sigma(m,\rho)$ for $\Sparse(m)\cap 2\ball$ of cardinality at most ${n\choose m} O(1/\rho)^{2m} = O(n/m\rho^2)^m$. 

Let $\Sigma'(m,\rho)$ denote the projection of $\Sigma(m,\rho)$ to the subspace $\langle \1\rangle^\perp$. 
Since projection to a subspace can only contract distances and cardinalities, $\Sigma'(m,\rho)$ is a $\rho$-net for $\Proj_{\langle \1\rangle^\perp}(\Sparse(m)\cap 2\ball)$ and $|\Sigma'(m,\rho)|\le |\Sigma(m,\rho)|$.
Now by \eqref{net:ualt} and \eqref{net:v}, any element $u\in \Flat_0(m,\rho)$ is within distance $2\rho$ of an element of $\Proj_{\langle \1\rangle^\perp}(\Sparse(m)\cap 2\ball)$, and hence is within distance $3\rho$ of an element of $\Sigma'(m,\rho)$. 

Finally, we rescale every element of $\Sigma'(m,\rho)$ to be of unit length and denote the resulting set by $\Sigma_0(m,\rho)$. Note that the rescaling leaves $\Sigma_0(m,\rho) \subset \Proj_{\langle \1\rangle^\perp}(\Sparse(m))$. 
In particular, $\Sigma_0(m,\rho)\subset \Flat_0(m,\rho)$.
Moreover, for any $u\in \Flat_0(m,\rho)$ and $y\in \Sigma'(m,\rho)$ with $\|u-y\|\le 3\rho$, it follows that $\|y\|\in [1-3\rho,1+3\rho]$, so by the triangle inequality $u$ is within distance $6\rho$ of $y/\|y\|\in \Sigma_0(m,\rho)$. 
Thus, $\Sigma_0(m,\rho)$ is a $6\rho$-net for $\Flat_0(m,\rho)$ of cardinality $O(n/m\rho^2)^m$.
The result now follows by replacing $\rho$ with $\rho/6$. 
\end{proof}

\subsection{Anti-concentration properties of non-flat vectors}		\label{sec:ac} 

In this section we observe that the property of a unit vector $u\in \sph$ not lying in $\Flat(m,\rho)$ implies an anti-concentration property for the empirical measure $\frac1n\sum_{i=1}^n\delta_{u_i}$ of its components. We then show that we can find a large ``bimodal" piece of the empirical measure for such a vector; specifically, we can find two well-separated subsets of the plane that each capture a large portion of the total measure. 

For $v\in \C^n$, $\lambda\in \C$ and $\rho>0$ we write
\begin{equation}
E_v(\lambda,\rho):= \left\{j\in [n]: \left|v_j - \frac{\lambda}{\sqrt{n}}\right| < \frac{\rho}{\sqrt{n}}\right\}
\end{equation}
and define the \emph{concentration function}
\begin{equation}
Q_v(\rho) := \sup_{\lambda\in \C} \frac1n|E_v(\lambda,\rho)|
\end{equation}
(as $|E_v(\lambda,\rho)|$ takes values in the discrete set $[n]$ the supremum is attained). 
We remark that $Q_v(\rho)$ is the classical L\'evy concentration function for the empirical measure $\frac1n\sum_{i=1}^n \delta_{v_i}$.

\begin{lemma}[Anti-concentration for non-flat vectors]	\label{lem:ac}
Let $u\in \sph\setminus \Flat(m,\rho)$. Then $Q_u(\rho)<1-\frac{m}n$.
\end{lemma}

\begin{proof}
If this were not the case then there would exist $\lambda\in \C$ such that $|E_u(\lambda,\rho)|\ge n-m$. Then taking $v=(u-\frac{1}{\sqrt{n}}\lambda\1)_{E_u(\lambda,\rho)^c}$ we have that $v\in \Sparse(m)$, and
\[
\left\| u - v -\frac1{\sqrt{n}}\lambda\1\right\| = \left\|\left(u-\frac1{\sqrt{n}}\lambda\1\right)_{E_u(\lambda,\rho)} \right\| <\rho
\]
which implies $u\in \Flat(m,\rho)$, a contradiction. 
\end{proof}

\begin{lemma}[Existence of a large bimodal component]		\label{lem:gap}
Let $u\in \sph\setminus \Flat(m,\rho)$. 
There exist disjoint sets $J_1,J_2\subset[n]$ such that $|J_1|\ge m$, $|J_2|\gg n-m$ and 
\begin{equation}	\label{twolevels.weak}
|u_{j_1}-u_{j_2}| \ge \frac{\rho}{2\sqrt{n}} \qquad \forall\; j_1\in J_1,\;j_2\in J_2.
\end{equation}
Moreover, there is a set $J_1'\subset J_1$ with $|J_1'|\gg m$ such that for any $1\le r\le \min(|J_1'|,|J_2|)$, 
\begin{equation}	\label{twolevels.strong}
\left| \frac1 r \sum_{j\in J_1''} u_j - \frac1r \sum_{j\in J_2''} u_j\right| \ge \frac{\rho}{4\sqrt{n}}\qquad \forall\; J_1''\in {J_1'\choose r},\; J_2''\in {J_2\choose r}.
\end{equation}
\end{lemma}

\begin{remark}
We will later apply the above lemma when studying random variables of the form $W_{u,\pi}=\sum_{j=1}^m \xi_j(u_j-u_{\pi(j)})$, where $\bxi=(\xi_1,\dots, \xi_m)$ is a sequence of iid Bernoulli($1/2$) variables and $j\mapsto \pi(j)$ is a pairing between $2m$ distinct elements of $[n]$.
We will think of $W_{u,\pi}$ as a random walk on $\C$ with steps $u_j-u_{\pi(j)}$, and use \eqref{twolevels.weak} to argue that for certain $u$ this walk takes many large steps and is thus unlikely to concentrate significantly in any small ball. 
At some point we will also consider a random walk whose steps are differences between \emph{averages} of the components of $u$ over sets of equal size, rather than differences between individual components, in which case we will need \eqref{twolevels.strong}.
\end{remark}

\begin{proof} 	
We observe that for any $\eps>0$, 
\begin{equation}	\label{q:cont}
Q_u(\eps/2) \ge c \,Q_u(\eps)
\end{equation}
for some universal constant $c>0$. 
Indeed, letting $\lambda\in \C$ be such that $Q_u(\eps) = \frac1n|E_u(\lambda,\eps)|$, we can cover the ball $\{w\in \C: |w-\lambda|<\eps\}$ with $O(1)$ balls of radius $\eps/2$, and the claim follows from the pigeonhole principle.

Write $q_k= Q_u(2^{-k})$ and consider the non-increasing sequence $\{q_k\}_{k\in \Z}$. 
Since all components of $v$ lie in the unit disk we have $q_k=1$ for $k<-\frac12\log_2n$.
Let
\[
k_0= \min\{k: q_k<1-m/n\}.
\] 
Then $k_0\ge -\frac12\log_2n$, and from Lemma \ref{lem:ac} we have $k_0\le \lceil\log_2(1/\rho)\rceil$.
From \eqref{q:cont},
\begin{equation}
q_{k_0+1} \ge c^2 q_{k_0-1} \ge c^2(1-m/n).
\end{equation}
Let $\lambda_0\in \C$ such that $q_{k_0+1}=\frac1n|E_u(\lambda_0,2^{-k_0-1})|$. 
We have
\[
|E_u(\lambda_0,2^{-k_0})|\le nq_{k_0} <n-m.
\]
Taking $J_1 = E_u(\lambda_0,2^{-k_0})^c$ and $J_2 = E_u(\lambda_0,2^{-k_0-1})$
we have $|J_1|\ge m$, $|J_2|\ge c^2(n-m)$, and for all $j_1\in J_1,j_2\in J_2$, 
$$|u_{j_1}-u_{j_2}| \ge  \frac{2^{-k_0-1}}{\sqrt{n}} \ge\frac{\rho}{2\sqrt{n}},$$
and \eqref{twolevels.weak} follows.

For \eqref{twolevels.strong}, let $c_0>0$ be a sufficiently small constant, and divide the complement of the ball $B(\lambda_0/\sqrt{n},2^{-k_0}/\sqrt{n})\subset \C$ into $\lceil 1/c_0\rceil$ congruent angular sectors.
By the pigeonhole principle one of which must contain at least $c_0m$ of the components of $J_1$. Taking $J_1'$ to be the set of corresponding indices, we can take $c_0$ smaller if necessary to ensure that for some open halfspace $H\subset \C$, 
\[	
\dist\big( H, \{u_j: j\in J_2\}\big),\quad \dist\big(H^c, \{u_j: j\in J_1'\}\big) \ge \frac{\rho}{8\sqrt{n}}.
\]
\eqref{twolevels.strong} now follows from the above lower bounds, the convexity of $H$, and the triangle inequality.
\end{proof}

\section{Invertibility over very flat vectors}	\label{sec:high}

In this section we prove the following lemma, which already implies Proposition \ref{prop:flat} for $d\gg n/\log^2n$, but is weaker for smaller values of $d$.

\begin{lemma}[Very flat vectors]	\label{lem:very}
Let $1\le d\le n/2$ and $1\le K\le n^{\gamma_0}$ for some fixed $\gamma_0\ge1/2$. Recall the events \eqref{def:event} with $A\in \mA_{n,d}$ drawn uniformly at random.
For all $1\le m\le cd/\gamma_0\log n$ we have
\begin{equation}
\pro{ \event_K\left(m,\,\frac{c}{K\sqrt{m}} \right) } =O( e^{-cd})
\end{equation}
where $c>0$ is an absolute constant.
\end{lemma}

Here the key graph regularity property will be the control on codegrees enjoyed by elements of the set $\mA^{\codeg}(i_1,i_2)$ from \eqref{def:goodcodeg}. 
We need the following variant of a lemma of Rudelson and Vershynin (see \cite[Lemma 2.2]{RuVe:ilo},\cite[Lemma 2.10]{Cook:ssv}).
\begin{lemma}[Tensorization of anti-concentration]	\label{lem:tensorize}
Let $\zeta_1,\dots, \zeta_n$ be independent non-negative random variables.
Suppose that for some $\eps_0,p_0>0$ and all $j\in [n]$, $\pro{\zeta_j\le \eps_0}\le p_0$. There are $c_1,p_1\in (0,1)$ depending only on $p_0$ such that
\begin{equation}
\pr\bigg( \sum_{j=1}^n \zeta_j^2 \le c_1\eps_0^2n\bigg) \le p_1^n.
\end{equation}
\end{lemma}

\begin{proof}[Proof of Lemma \ref{lem:very}]
Let $m$ be as in the statement of the lemma and denote
\[
\rho = \frac{c}{K\sqrt{m}}.
\]

First we consider an arbitrary fixed vector $u\in \Flat_0(m,\rho)$.
By definition, there exists $\lambda\in \C$, $v\in \Sparse(m)$ and $w\in \rho \ball$ such that $u=v+\frac{\lambda}{\sqrt{n}}\1+w$.
We note that 
\begin{equation}	\label{very:vwlb}
\|v+w\|\ge 1/2.
\end{equation}
Indeed, by the triangle inequality,
\begin{equation}	\label{very:1}
|\lambda| = \left\|\frac{\lambda}{\sqrt{n}}\1\right\| \ge \|u\|-\|v+w\| = 1-\|v+w\|.
\end{equation}
On the other hand by the assumption $u\in \sph_0$ and Cauchy--Schwarz,
\[
|\lambda|\sqrt{n} = \left| \sum_{j=1}^n v_j+w_j\right| \le \|v+w\|\sqrt{n}
\]
and so
\[
|\lambda| \le \|v+w\|.
\]
Combined with \eqref{very:1} this gives \eqref{very:vwlb}.

Let $J\subset[n]$ with $|J|=m$ such that $\supp(v)\subset J$.
From \eqref{very:vwlb},
\[
\frac14 \le \|v+w\|^2 \le m \max_{j\in J} \left| u_j -\frac{\lambda}{\sqrt{n}}\right|^2.
\]
It follows that there exists $j_1\in J$ with 
\begin{equation}	\label{vj1}
\left| u_{j_1}-\frac{\lambda}{\sqrt{n}}\right| \ge \frac{1}{2\sqrt{m}}.
\end{equation}
On the other hand, since $\sum_{j\in J^c} |w_j|^2 \le \|w\|^2\le \rho^2$
it follows from the pigeonhole principle that there exists $j_2\in J^c$ such that 
\[
\left| u_{j_2}-\frac{\lambda}{\sqrt{n}}\right| =|w_{j_2}| \le \frac{\rho}{\sqrt{n-m}}.
\]
By the previous displays and the triangle inequality we have
\begin{equation}	\label{very:du}
|u_{j_1} - u_{j_2}| \ge \frac{1}{2\sqrt{m}} - \frac{\rho}{\sqrt{n-m}}\ge \frac{1}{4\sqrt{m}}.
\end{equation}

Let $A\in \mA_{n,d}$ be drawn uniformly at random. We form a coupled matrix $\wA\in \mA_{n,d}$ as follows.
Conditional on $A$, we fix some arbitrary bijection $\pi: \mN_{A^\tran}(j_1)\setminus \mN_{A^\tran}(j_2) \to \mN_{A^\tran}(j_2)\setminus \mN_{A^\tran}(j_1)$ (if these sets are empty we simply set $\wA=A$). We do this in some measurable fashion with respect to the sigma algebra generated by $A$.
We let $\bxi=(\xi_i)_{i=1}^n\in \{0,1\}^n$ be a sequence of iid Bernoulli(1/2) indicators, independent of $A$, and form $\wA$ by replacing the submatrix $A_{(i,\pi(i))\times (j_1,j_2)}$ with 
\begin{equation}	\label{very:switchings}
\wA_{(i,\pi(i))\times (j_1,j_2)} =  \begin{pmatrix} 1 & 0\\ 0 & 1\end{pmatrix} + \xi_i \begin{pmatrix}-1& +1 \\ +1 & -1\end{pmatrix}
\end{equation}
for each $i \in \mN_{A^\tran}(j_1)\setminus \mN_{A^\tran}(j_2)$.

We claim that $\tA\eqd A$. 
It is clear that for any realization of the signs $\bxi$ the replacements \eqref{very:switchings} do not affect the row and column sums, so $\tA\in \mA_{n,d}$.
Now note that $A$ and $\tA$ agree on all entries $(i,j)\notin \mN_{A^\tran}(j_1)\triangle \mN_{A^\tran}(j_2) \times \{j_1,j_2\}$.
Conditional on any realization of the entries of $A$ outside this set, from the constraints on row and column sums we have that the remaining entries of $A$ are determined by the set $ \mN_{A^\tran}(j_1)\setminus \mN_{A^\tran}(j_2) $, and similarly the remaining entries of $\tA$ are determined by $ \mN_{\tA^\tran}(j_1)\setminus \mN_{\tA^\tran}(j_2) $.  $\mN_{A^\tran}(j_1)\setminus \mN_{A^\tran}(j_2)$ is uniformly distributed over subsets of $\mN_{A^\tran}(j_1)\triangle \mN_{A^\tran}(j_2)$ of cardinality $|\mN_{A^\tran}(j_1)\triangle \mN_{A^\tran}(j_2)|/2$. One then notes that for any fixed realization of $\bxi$, the set $\mN_{\tA^\tran}(j_1)\triangle \mN_{\tA^\tran}(j_2)$ is also uniformly distributed over subsets of $\mN_{A^\tran}(j_1)\triangle \mN_{A^\tran}(j_2)$ of the same cardinality. The claim then follows from the independence of $\bxi$ from $A$.

Denote the rows of $\wA+Z$ by $\wR_i$. We have
\begin{equation}	\label{very:oneswitch}
\wR_i\cdot u = R_i \cdot u+ \xi_i (u_{j_2}-u_{j_1}).
\end{equation}
Recall the sets $\mA^{\codeg}(i_1,i_2)\subset\mA_{n,d}$ from Lemma \ref{lem:codeg}. Since $A\eqd \wA$ we have
\begin{align}
&\pro{ \|(A+Z)u\| \le c\sqrt{\frac{d}{\me}} } =\pro{ \|(\wA+Z) u\| \le c\sqrt{\frac{d}{\me}} }		\notag\\
&\qquad\le  \pro{ A^\tran\notin\mA^{\codeg}(j_1,j_2)} + \pro{ \|(\wA+Z) u\| \le c\sqrt{\frac{d}{\me}},\;  A^\tran\in\mA^{\codeg}(j_1,j_2) }		\notag\\
&\qquad =  \pro{ A^\tran\notin\mA^{\codeg}(j_1,j_2)} + \e \pr_\bxi  \left(\|(\wA+Z) u\| \le c\sqrt{\frac{d}{\me}}\right)  \un\Big(A^\tran\in\mA^{\codeg}(j_1,j_2)\Big).		\label{very:split}
\end{align}
For $A$ such that $A^\tran\in\mA^{\codeg}(j_1,j_2)$ we have 
\[
|\mN_{A^\tran}(j_1)\setminus \mN_{A^\tran}(j_2)| = d- |\mN_{A^\tran}(j_1)\cap \mN_{A^\tran}(j_2)| \ge 3d/4.
\]
Now for any $i\in  \mN_{A^\tran}(j_1)\setminus \mN_{A^\tran}(j_2) $, from \eqref{very:du} and \eqref{very:oneswitch} we have
\begin{equation}
\pr_{\xi_i}\left(|\wR_i\cdot u| \le \frac{c}{\sqrt{\me}}\right)\le \frac12
\end{equation}
if $c>0$ is a sufficiently small constant. 
From Lemma \ref{lem:tensorize} it follows that
\begin{align*}
\pr_\bxi\left( \sum_{i\in  \mN_{A^\tran}(j_1)\setminus \mN_{A^\tran}(j_2) } |\wR_i\cdot u|^2 \le \frac{c'}{\me}| \mN_{A^\tran}(j_1)\setminus \mN_{A^\tran}(j_2) | \right) \le e^{-c | \mN_{A^\tran}(j_1)\setminus \mN_{A^\tran}(j_2) |}
\end{align*}
so 
\begin{align*}
&\pr_\bxi \left(\|(\wA+Z) u\| \le c\sqrt{\frac{d}{\me}}\right)  \un\Big(A^\tran\in\mA^{\codeg}(j_1,j_2) \Big)\\
&\quad \le \pr_\bxi \left(\sum_{i=1}^n |\wR_i\cdot u|^2 \le \frac{c^2d}{m}\right)  \un\Big( |\mN_{A^\tran}(j_1)\setminus \mN_{A^\tran}(j_2)|  \ge 3d/4 \Big)\\
& \quad \le \pr_\bxi\left( \sum_{i\in  \mN_{A^\tran}(j_1)\setminus \mN_{A^\tran}(j_2) } |\wR_i\cdot u|^2 \le \frac{c'}{\me}| \mN_{A^\tran}(j_1)\setminus \mN_{A^\tran}(j_2) | \right)\un\Big( |\mN_{A^\tran}(j_1)\setminus \mN_{A^\tran}(j_2)|  \ge 3d/4\Big)\\
&\quad \le e^{-cd}.
\end{align*}
Combined with \eqref{very:split} and Lemma \ref{lem:codeg} (applied to $A^\tran$, which is also uniform over $\mA_{n,d}$) we conclude
\begin{equation}	\label{very:ufixed}
\sup_{u\in\Flat_0(m,\rho)}\pro{  \|(A+Z)u\| \le c\sqrt{\frac{d}{\me}}} = O(e^{-cd}).
\end{equation}

Let $\Sigma_0(m,\rho)\subset \Flat_0(m,\rho)$ be a $\rho$-net for $\Flat_0(m,\rho)$ as in Lemma \ref{lem:flatnet}.
On the event $\event_K(m,\rho)$ we have $\|A+Z\|_{\oneperp}\le K\sqrt{d}$ and $\|(A+Z)v\|\le \rho K\sqrt{d}$ for some $v\in \Flat_0(m,\rho)$. Letting $u\in \Sigma_0(m,\rho)$ such that $\|u-v\|\le \rho$, by the triangle inequality,
\begin{align*}
\|(A+Z)u\| &\le \|(A+Z)v\| + \|(A+Z)(u-v)\|\\
&\le \rho K\sqrt{d} + \rho \|A+Z\|_{\oneperp}\\
& \le 2\rho K\sqrt{d}.
\end{align*}
Thus,
\[
\pro{ \event_K(m,\rho)} \le \pro{ \exists u \in \Sigma_0(\me,\rho): \|(A+Z)u\|\le 2\rho K\sqrt{d}}.
\]
By our choice of $\rho$ (adjusting the constant $c$) we have $2\rho K\sqrt{d}\le c\sqrt{d/m}$.
Thus, we can apply the union bound and the estimate \eqref{very:ufixed} to conclude
\begin{align*}
\pro{ \event_K(m,\rho)}
& \ll O\left(\frac{n}{m\rho^2}\right)^{m} e^{-cd}\\
&\le \expo{ \me \log\left(\frac{n}{\me}\right) + 2\me\left(\log\left(\frac1{\rho}\right)+ O(1)\right) - cd}\\
&\le \expo{ O(\gamma_0\me\log n) - cd}\\
&\le e^{-cd/2}
\end{align*}
where we have substituted the assumed bound on $m$ and the expression for $\rho$. The claim follows.
\end{proof}

\section{Incrementing control on structured vectors}	\label{sec:increment}

In this section we upgrade the control on very flat vectors from Lemma \ref{lem:very} to obtain Proposition \ref{prop:flat} by an iterative argument. Note that this step is not necessary for large degrees $n/\log^2n\ll d\le n/2$.
Whereas Lemma \ref{lem:very} was established by restricting attention to $A\in\mA^{\codeg}(i_1,i_2)$, i.e.\ restricting to the event that $A$ has controlled codegrees, here we will use the expansion property enjoyed by elements of $\mA^{\exp}(\kappa)$ (defined in \eqref{def:goodexp}). 
As in the proof of Lemma \ref{lem:very} we will create a coupled matrix $\tA$ using random switchings. This time the switchings will be applied across several columns rather than just two. While similar in spirit to the coupling from the previous section, the coupling used here requires some more care and notation to properly define.

\subsection{Neighborhood switchings}	\label{sec:switchings}

Let $\bi=\{i,i'\}\subset [n]$ and let  $J, J'\subset [n]$ be disjoint with $|J|=|J'|$.
Define a mapping
\[
\Switch_{\bi,J,J'}: \mA_{n,d}\to\mA_{n,d}
\]
as follows. 
For $A\in \mA_{n,d}$ such that $J\subset \mN_A(i)\setminus \mN_A(i')$ and $J'\subset \mN_A(i')\setminus \mN_A(i)$, let $A'=\Switch_{\bi,J,J'}(A)$ be the element of $\mA_{n,d}$ that agrees with $A$ on all entries $(i_0,j_0)\notin \{i,i'\}\times J\cup J'$, and such that 
\[
\mN_{A'}(i) = (\mN_A(i)\setminus J) \cup J', \quad\quad \mN_{A'}(i')= (\mN_A(i') \setminus J')\cup J.
\]
Similarly, if $J\subset \mN_A(i')\setminus \mN_A(i)$ and $J'\subset \mN_A(i)\setminus \mN_A(i')$, let $A'=\Switch_{\bi,J,J'}(A)$ be the element of $\mA_{n,d}$ that agrees with $A$ on all entries $(i_0,j_0)\notin \bi\times J\cup J'$, and such that 
\[
\mN_{A'}(i) = (\mN_A(i)\setminus J') \cup J, \quad\quad \mN_{A'}(i')= (\mN_A(i') \setminus J)\cup J'.
\]
Otherwise set $\Switch_{\bi,J,J'}(A)=A$.
It is straightforward to verify that $\Switch_{\bi,J,J'}$ is an involution on $\mA_{n,d}$. 
We say that $A$ is \emph{switchable at $(\bi,J,J')$} if $\Switch_{\bi,J,J'}(A)\ne A$. 
See Figure \ref{fig:SwitchiJ}. 

\begin{figure}
\centering
  \includegraphics[width=130mm]{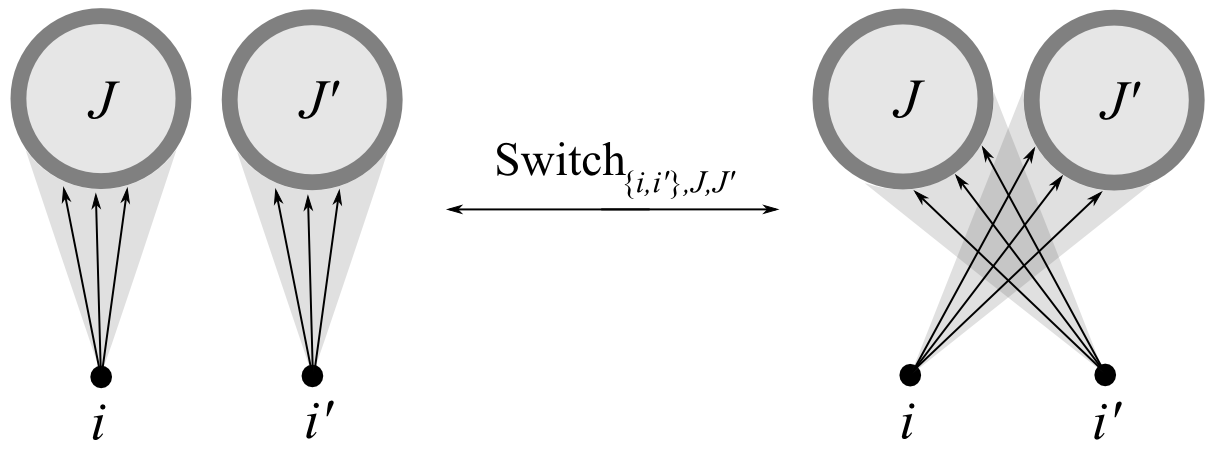}

\caption{Depiction of the effect of applying the switching operation $\Switch_{\{i,i'\},J,J'}$ in the digraph associated to a matrix $A$ that is switchable at $(\{i,i'\},J,J')$. Here we depict only (and all of) the directed edges from $\{i,i'\}$ to $J\cup J'$; in particular, for the configuration on the left, say, it is important that there are no edges from $i$ to $J'$ or from $i'$ to $J$. }

\label{fig:SwitchiJ}
\end{figure}

For $\xi\in \{0,1\}$ we interpret
$\Switch^\xi_{\bi,J,J'}$ to mean $\Switch_{\bi,J,J'}$ when $\xi=1$ and the identity map when $\xi=0$.
We will later need the following:

\begin{lemma}[Stability of $\mA^{\exp}(\kappa)$ under switchings]	\label{lem:Aexp.stable}
Let $\kappa\in (0,1)$ and $A\in \mA^{\exp}(\kappa)$. 
Let $\{(\{i_l,i_l'\},J_l,J_l')\}_{l=1}^k$ be a sequence such that the $2k$ indices $i_1,i_1',\dots, i_k,i_k' \in [n]$ are all distinct, and for all $1\le l\le k$ we have $J_l\cap J_l'=\emptyset$ and $|J_l|=|J'_l|$. Let $\bxi\in \{0,1\}^{k}$ and put 
\[
\tA= \left( \circ_{l\in [k]} \Switch^{\xi_l}_{\{i_l,i_l'\},J_l,J_l'}\right)(A).
\]
Then $\tA\in \mA^{\exp}(\kappa/2)$.
\end{lemma}

\begin{proof}
Fix an arbitrary set $J\subset[n]$ with $|J|\le n/2d$. It suffices to show
\[
|\mN_{\tA^\tran}(J)| \ge \frac12 |\mN_{A^\tran}(J)|.
\]
For any $i\in \mN_{A^\tran}(J)\setminus \mN_{\tA^\tran}(J)$ we must have $i\in \{i_l,i'_l\}$ for some $l\in [k]$ such that $\xi_l=1$. In the case that $i=i_l$, by definition of the switching we must have $i'_l\in \mN_{\tA^\tran}(J)$, and similarly if $i=i'_l$ then $i_l\in \mN_{\tA^\tran}(J)$. Since the indices $i_1,i'_1,\dots, i_k,i'_k$ are all distinct, we have
\[
|\mN_{\tA^\tran}(J)| \ge | \mN_{A^\tran}(J)\setminus \mN_{\tA^\tran}(J)| \ge |\mN_{A^\tran}(J)|- | \mN_{\tA^\tran}(J)|
\]
and the result follows upon rearranging. 
\end{proof}

\subsection{Coupling construction}		\label{sec:couplings}

Let $L,L'\subset[n]$ be fixed nonempty disjoint sets. 
Denote
\[
I=[\lf n/2\rf] \quad\text{and} \quad I'=[\lf n/2\rf +1, 2\lf n/2\rf]
\]
and fix a bijection $\pi:I\to I'$. 
For distinct $\{i,i'\}\subset [n]$ define
\begin{equation}	\label{def:goodplus}
\good_{L,L'}(i,i') = \big\{ A\in \mA_{n,d} : L_A(i')=\emptyset, \, | L'_A(i')\setminus \mN_A(i)| \ge | L_A(i)| \ge 1\big\}
\end{equation}
(recall our notation \eqref{def:LA}).
For $i\in I$ we abbreviate 
\[
\good^+(i):= \good_{L,L'}(i,\pi(i)), \quad \good^-(i) := \good_{L,L'}(\pi(i), i).
\]
Note $\good^+(i)\cap \good^-(i)=\emptyset$.
Denote
\begin{align}
I^+(A) &:= \big\{ i\in I: \,A\in \good^+(i)\big\},	\label{def:iplus}\\
I^-(A) &:=  \big\{ i\in I: \,A\in \good^-(i)\big\}.	\label{def:iminus}
\end{align}

\begin{figure}
\centering
  \includegraphics[width=130mm]{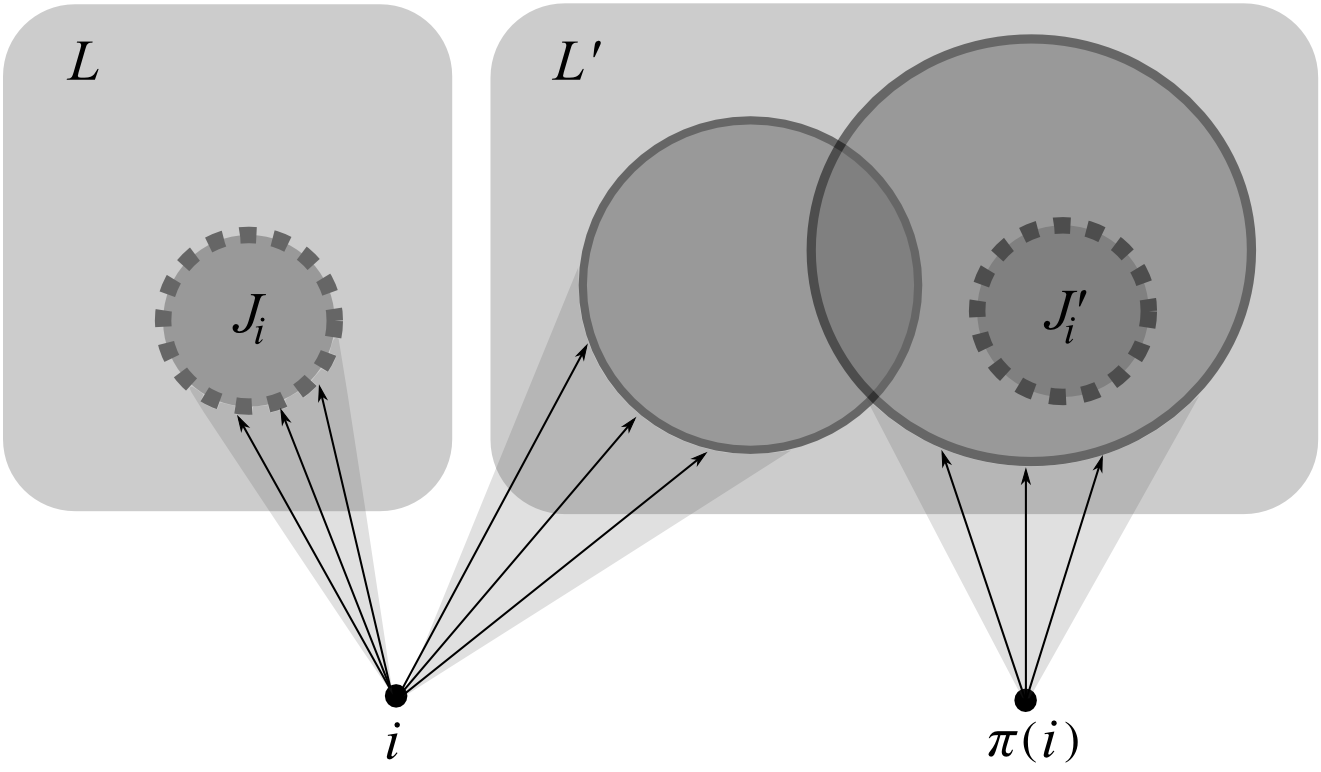}

\caption{Depiction of the out-neighborhoods of vertices $i,\pi(i)$ for the digraph corresponding to a matrix $A\in \good^+(i)=\good_{L,L'}(i,\pi(i))$, where we only depict the directed edges from $\{i,\pi(i)\}$ to $L\cup L'$. Note that $J_i=L_A(i)$, the set of out-neighbors of $i$ in $L$, while $\pi(i)$ has no out-neighbors in $L$. The set $J_i'$ is a random subset of $L'_A(\pi(i))\setminus L'_A(i)=L'\cap (\mN_A(\pi(i))\setminus \mN_A(i))$ of size $|J_i|$. In the formation of the coupled matrix $F_{L,L'}(A)$ we apply $\Switch_{\{i,\pi(i)\}, J_i,J_i'}$ (see Figure \ref{fig:SwitchiJ}), or not, depending on the value of $\xi_i\in \{0,1\}$, and we do this independently for each $i\in I^+(A)$.} 

\label{fig:LJpi}
\end{figure}

Let $A\in \mA_{n,d}$.
We define a deterministic sequence $\mJ(A)=\{J_i(A)\}_{i\in I^+(A) \cup I^-(A) }$ of subsets of $L$, and a sequence of jointly independent random sets $\mJ'(A)=\{J'_i(A)\}_{i\in  I^+(A) \cup I^-(A)}$ in $L'$ as follows.
For each $i\in I^+(A)$ we set $J_i(A)= L_A(i)$ and draw
\begin{equation}	\label{Jprime1}
J'_i(A)\in { L'_A(\pi(i))\setminus \mN_A(i) \choose |L_A(i)|}
\end{equation}
uniformly at random.
For each $i\in I^-(A)$ we set $J_i(A)= L_A(\pi(i))$ and draw
\begin{equation}	\label{Jprime2}
J'_i(A)\in { L'_A(i)\setminus \mN_A(\pi(i)) \choose |L_A(\pi(i))|}
\end{equation}
uniformly at random.
Note that $|J_i(A)|=|J'_i(A)|$ for all $i\in I^+(A) \cup I^-(A)$.
Let $\bxi=(\xi_i)_{i\in [n]} \in \{0,1\}^n$ be drawn uniformly at random, independent of all other random variables. 
Now set
\begin{equation}	\label{def:FLLA}
F_{L,L'}(A)=F_{L,L'}(A; \mJ'(A), \bxi) =  \Big( \circ_{i\in I^+(A)\cup I^-(A)} \Switch^{\xi_i}_{\{i,\pi(i)\}, J_i(A),J_i'(A)} \Big) (A).
\end{equation}
We emphasize that $F_{L,L'}(A)$ is not determined by $A$, as it is also a function of the random variables $\mJ'(A)$ and $\bxi$.

\begin{lemma}	\label{lem:FA}
With $F_{L,L'}(A)$ defined as above, if $A\in \mA_{n,d}$ is uniform random then $F_{L,L'}(A)\in \mA_{n,d}$ is also uniform random. 
\end{lemma}

\begin{proof}
That $F_{L,L'}(A)\in \mA_{n,d}$ follows from the fact that the individual mappings $\Switch^{\xi_i}_{\{i,\pi(i)\}, J_i,J_i'}$ have range in $\mA_{n,d}$.
To show $F_{L,L'}(A)$ is uniformly distributed, by the joint independence of the indicator variables $\xi_i$  and the sets $J_i'(A)$ conditional on $A$ it suffices to consider only one switching operation. 
That is, fix $i\in I$ and let
\begin{equation}
\stA = 
\begin{cases} 
\Switch^{\xi_i}_{\{i,\pi(i)\}, J_i(A),J_i'(A)} & A\in \good^+(i)\cup \good^-(A)\\
A & \text{otherwise}.
\end{cases}
\end{equation}
Furthermore, by the independence of $\xi_i$ from all other variables we may fix $\xi_i=1$ as the claim is immediate for $\xi_i=0$. Since $i$ is fixed we will lighten notation by writing $i'= \pi(i)$, $\good^\pm = \good^\pm (i)$, $J=J_i, J'=J'_i$.

Fix an arbitrary element $A^0\in \mA_{n,d}$. 
Our aim is to show
\begin{equation}
\pr(\stA=A^0) = \pr(A= A^0).
\end{equation}
Since $\stA=A$ when $A\notin \good^+\cup \good^-$, 
\begin{align*}
\pr(\stA=A^0) 
&= \pr( A = A^0, A\notin \good^+\cup \good^-) + \pr(\stA=A^0 | A\in \good^+)\pr(A\in\good^+)\\
&\qquad   + \pr( \stA=A^0| A\in \good^-) \pr(A\in \good^-).
\end{align*}
By symmetry we have $\pr(A\in \good^+) = \pr(A\in \good^-)$.
Thus, by subtracting a similar expression for $\pr(A=A^0)$, it suffices to show
\begin{align}
\pro{ \stA=A^0 \,\middle|\, A\in \good^+} &= \pro{ A = A^0 \,\middle|\, A\in \good^-}, \text{ and}	\label{Fb:goal1}\\
\pro{ \stA=A^0 \,\middle|\, A\in \good^-} &= \pro{ A = A^0 \,\middle|\, A\in \good^+}.
\end{align}
Again by symmetry, it suffices to establish \eqref{Fb:goal1}. 
Note that 
\begin{equation}	\label{ata.outside}
A(i_0,j_0) = \stA(i_0,j_0) \quad \forall (i_0,j_0)\notin \{i,i'\}\times L\cup L'.
\end{equation}
Thus, both sides of \eqref{Fb:goal1} are zero unless the event
\[
\event_0 = \{A(i_0,j_0) = A^0(i_0,j_0) \quad \forall (i_0,j_0)\notin \{i,i'\}\times L\cup L'\}
\]
holds. Our aim is now to show
\begin{equation}	\label{Fb:goal2}
\pro{ \stA=A^0 \,\middle|\, A\in \good^+, \event_0} = \pro{ A = A^0 \,\middle|\, A\in \good^-, \event_0}.
\end{equation}
Now notice that due to the constraints on column sums for elements of $\mA_{n,d}$, restricting to the event $\event_0\wedge \{A\in \good^+\}$ fixes all entries of $A,A^*$ except for those in 
$\{i,i'\}\times L'_0$, where 
\begin{align*}
L'_0 &:= L'_{A^0}(i) \triangle L'_{A^0}(i') \\
&\;= L'_{A}(i) \triangle L'_{A}(i') = L'_{\stA}(i) \triangle L'_{\stA}(i'),
\end{align*}
and the same is true if we restrict to $\event_0\wedge \{A\in \good^-\}$. 
Thus, it suffices to show
\begin{equation}	\label{Fb:goal3}
\pro{ \stA_{(i,i')\times L_0'} =A^0_{(i,i')\times L_0'} \,\middle|\, A\in \good^+, \event_0} = \pro{ A_{(i,i')\times L_0'} = A^0_{(i,i')\times L_0'} \,\middle|\, A\in \good^-, \event_0}.
\end{equation}
Under the conditioning, the submatrix $A_{(i,i')\times L_0'}$ is determined by the random set $L''_A:=L'_A(i)\setminus L'_A(i')\subset L'_0$, and similarly $\stA_{(i,i')\times L_0'}$ is determined by $L''_{\stA}=L'_{\stA}(i)\setminus L'_{\stA}(i')\subset L'_0$.
On the event $\event_0\wedge \{A\in \good^-\}$, since $A$ is uniformly distributed, $L''_A$ is uniformly distributed over subsets of $L'_0$ of fixed size $r:=|L''_{A^0}|$.
Thus, it suffices to show that on $\event_0\wedge \{A\in \good^+\}$, $L''_{\stA}$ is also uniformly distributed over subsets of $L'_0$ of size $r$.

From the definition \eqref{Jprime1}, on $\event_0\wedge \{A\in \good^+\}$ we have
\[
L''_{\stA} = L''_A\cup J'
\]
where, conditional on $A$, $J'$ is drawn uniformly from subsets of $L'_A(i')\setminus L'_A(i)=L_0'\setminus L''_A$ of size $|L_A(i)|=:s$, and we note that $s$ is fixed by the constraints on column sums. From the constraints on row sums we have $|L_A''| = r-s$. 
Also, again by uniformity of $A\in \mA_{n,d}$, $L_A''$ is uniformly distributed over subsets of $L_0'$ of size $r-s$. 
Thus, on $\event_0\wedge \{A\in \good^+\}$, $L_{\stA}''$ is generated by first selecting $L_A''\subset L_0'$ of size $r-s$ uniformly at random, and then adjoining a uniform random set $J'\subset L_0'\setminus L_A''$ of size $s$. 
It follows that $L_{\stA}''$ is uniformly distributed over subsets of $L'_0$ of size $r$, as desired.
\end{proof}

In Lemma \ref{lem:iplus} below we show that when $A$ is drawn uniformly at random, with high probability the number of random switchings $|I^+(A)\cup I^-(A)|$ that are applied to form $F_{L,L'}(A)$ is fairly large. 
To prove this we need a consequence of concentration of measure for the symmetric group, which will also be used in the proof of Lemma \ref{lem:walk}.

\begin{lemma}	\label{lem:conc}
Let $S, T$ be finite sets with $m=|S|\le |T|$ and let $\sigma:S\to T$ be a uniform random injection.
Let $T_1,\dots, T_m\subset T$ be fixed subsets and set
\[
N_\sigma = |\{i\in S: \sigma(i)\in T_i\}|.
\]
We have $\e N_\sigma = \frac1{|T|}\sum_{i\in S} |T_i|$, and for any $\delta>0$,
\begin{equation}	\label{conc1}
\pro{ \left| N_\sigma - \e N_\sigma \right| \ge \delta \e N_\sigma } \le 2\expo{ -\frac{\delta^2}{4+2\delta} \e N_\sigma}. 
\end{equation}
In particular, if $T_0\subset T$ is fixed and $U\subset T$ is a uniform random set of size $m$, then for any $\delta>0$,
\begin{equation}	\label{conc2}
\pro{ \left| |U\cap T_0| - \frac{m|T_0|}{|T|}\right| \ge \delta \frac{m|T_0|}{|T|}} \le 2\expo{ -\frac{\delta^2}{4+ 2\delta} \frac{m|T_0|}{|T|}}.
\end{equation}

\end{lemma}

\begin{proof}
Label the elements of $S,T$ as $S=\{s_1,\dots, s_m\}$, $T=\{t_1,\dots, t_N\}$, and define an $N\times N$ matrix 
$M=(m_{ij})$ with 
\[
m_{ij} = \1(i\in [m], t_j\in T_i).
\]
Letting $\tau:[N]\to [N]$ be a uniform random permutation, we have
\[
N_\sigma \eqd \sum_{i=1}^N m_{i\tau(i)}.
\]
\eqref{conc1} now follows from a concentration bound for Hoeffding's permutation statistic due to Chatterjee; see \cite[Proposition 1.1]{Chatterjee}.
\eqref{conc2} is obtained by setting $T_i\equiv T_0$ for all $i\in S$ and noting that $U\eqd \sigma(S)$.
\end{proof}

\begin{lemma}	\label{lem:iplus}
Let $\kappa\in (0,1)$, and assume 
\begin{equation}	\label{iplus.d}
\frac{C}{\kappa^2} \le d\le \frac{\kappa n}{C}
\end{equation}
for a sufficiently large constant $C>0$.
Let $L,L'\subset[n]$ be fixed disjoint sets satisfying
\begin{equation}	\label{LLbounds}
\frac{8}{\kappa} \le |L| \le \frac{n}{16d}, \qquad \frac{32n}{\kappa^2 d} \le |L'|\le n
\end{equation}
(note there exist such $L,L'$ if $C$ is taken sufficiently large). 
Let $A\in \mA_{n,d}$ be uniform random, and let $I^+(A)$ be as in \eqref{def:iplus}. 
We have
\begin{equation}
\pro{ |I^+(A)|\le c\kappa d|L|, \, A\in \mA^{\exp}(\kappa)} \ll \expo{ -c\kappa d |L|}
\end{equation}
for a sufficiently small constant $c>0$.
\end{lemma}

\begin{proof}
We will use Lemma \ref{lem:goodexpand} and the row-exchangeability of $A$.

Let $r\ge1$ to be chosen later, and denote
\begin{align*}
I_r(A) &= \mN_{A^\tran}^{\le r}(L) \setminus \mN_{A^\tran}^{\ge |L'|/2} (L'),\\
I'_r(A;i) &= \mN_{A^\tran}^{>r}(L'\setminus \mN_A(i)) \setminus \mN_{A^\tran}(L).
\end{align*}
(recall our notation from Lemma \ref{lem:goodexpand}). 
Note that if $i\in I$ is such that $i\in I_r(A)$ and $\pi(i)\in I_r'(A;i)$, then $i\in I^+(A)$. 
Indeed, from $\pi(i)\in I_r'(A;i)$ we have $\pi(i)\notin \mN_{A^\tran}(L)$, i.e.\ $L_A(\pi(i))=\emptyset$, and
\[
|L'_A(\pi(i))\setminus \mN_A(i)|  > r \ge |L_A(i)| \ge 1.
\]
(So far we have not used the disjointness of $I_r(A)$ from $\mN_{A^\tran}^{\ge |L'|/2} (L')$; this will be needed below to obtain a lower bound on the size of $I'_r(A;i)$.)
Thus, setting
\[
I_r^*(A)= \{i\in I\cap I_r(A): \pi(i)\in I_r'(A;i)\}
\]
it suffices to show
\begin{equation}	\label{iplus.goal1}
\pro{ |I_r^*(A)| \le c\kappa d |L|, A\in \mA^{\exp}(\kappa)} \ll \expo{ -c\kappa d|L|}.
\end{equation}

For the remainder of the proof we restrict to the event $\{A\in \mA^{\exp}(\kappa)\}$; note the rows of $A$ are still exchangeable under this restriction.
First we use Lemma \ref{lem:goodexpand} to prove bounds on the sizes of $I_r(A)$ and $I'_r(A;i)$.
From Lemma \ref{lem:goodexpand}(1), by our upper bound \eqref{LLbounds} on $|L|$ and taking
\begin{equation}	\label{take.r1}
r\ge 2/\kappa
\end{equation}
we have
\[
|\mN_{A^\tran}^{\le r}(L)| \ge \kappa d|L|/2.
\]
Furthermore, from
\[
\frac{|L'|}{2} |\mN_{A^\tran}^{\ge |L'|/2}(L')| \le e_A([n],L') = d|L'|
\]
we have $|\mN_{A^\tran}^{\ge |L'|/2}(L')| \le 2d$. Then from the lower bound on $|L|$ in \eqref{LLbounds} we conclude
\begin{equation}	\label{ir.lb}
|I_r(A)| \ge \kappa d|L|/2 - 2d \ge \kappa d|L|/4.
\end{equation}
Now let $i\in I_r(A)$.
Since $i\notin \mN_{A^\tran}^{\ge |L'|/2} (L')$ we have
\[
|L'\setminus \mN_A(i)|> |L'|/2.
\]
Now from our assumption \eqref{LLbounds} it follows $|L'|\ge n/\kappa d$, so $|L'\setminus \mN_A(i)|>n/2\kappa d$.
Thus, taking
\begin{equation}	\label{take.r2}
r \le \frac{\kappa d|L'|}{8n}
\end{equation}
we have $r\le \kappa d|L'\setminus \mN_A(i)|/4n$, and we can apply Lemma \ref{lem:goodexpand}(2) to obtain
\[
|\mN_{A^\tran}^{>r}(L'\setminus \mN_A(i))| >n/8.
\]
Our lower bound on $|L'|$ in \eqref{LLbounds} implies there exists a choice of $r$ satisfying both \eqref{take.r1} and \eqref{take.r2}. We henceforth fix such a choice of $r$.
Since $|\mN_{A^\tran}(L)|\le d|L| \le n/16$, we conclude
\begin{equation}	\label{irprime.lb}
|I'_r(A;i)| \ge n/16 \qquad \forall i\in I_r(A).
\end{equation}

Next we argue that
\begin{equation}	\label{iplus.goal1a}
\pro{ |I_r(A)\cap I| \le \frac1{16}\kappa d|L|} \ll \expo{ -c\kappa d|L|}
\end{equation}
(recall our restriction to the event $\{A\in \mA^{\exp}(\kappa)\}$). 
Let us condition on the size of $I_r(A)$. 
Note that the rows of $A$ are exchangeable under this conditioning and the restriction to $\{A\in \mA^{\exp}(\kappa)\}$.
Thus, $I_r(A)$ is a uniform random subset of $[n]$ of fixed size at least $ \kappa d|L|/4$, and \eqref{iplus.goal1a} follows from \eqref{conc2} (with $\delta =1/4$, say) and the fact that $|I|=\lf n/2\rf \ge n/4$.

Now we may restrict to the event $\{|I_r(A)\cap I| \ge \frac1{16}\kappa d|L|\}$. We condition on $ I_r(A)\cap I$ and write $m=|I_r(A)\cap I|$.
We will be done if we can show
\begin{equation}	\label{iplus.goal2}
\big| \big\{ i\in I_r(A)\cap I: \pi(i) \in I_r'(A;i)\big\}\big| \gg \kappa d|L|
\end{equation}
except with probability $O(\expo{ -c\kappa d|L|})$.
Under all of the conditioning we still have that the rows of $A$ with indices in $(I_r(A)\cap I)^c$ are exchangeable; thus, 
\begin{align*}
&\pro{ \big| \big\{ i\in I_r(A)\cap I: \pi(i) \in I_r'(A;i)\big\}\big| \le c\kappa d|L|} \\
&\qquad\qquad= \pro{ \big| \big\{ i\in I_r(A)\cap I: \pi(i) \in \sigma(I_r'(A;i))\big\}\big| \le c\kappa d|L|}
\end{align*}
where $\sigma$ is a uniform random permutation of $(I_r(A)\cap I)^c$, independent of $A$ under the conditioning. 
We condition on $A$ and proceed using only the randomness of $\sigma$.
Note that the restriction of $\sigma^{-1}\circ \pi$ to $I_r(A)\cap I$ is a uniform random injection into $(I_r(A)\cap I)^c$.
We have
\begin{align*}
\e  \big| \big\{ i\in I_r(A)\cap I: \pi(i) \in \sigma(I_r'(A;i))\big\}\big| &= \frac{1}{n-m} \sum_{i\in I_r(A)\cap I} |I_r'(A;i)| \gg m
\end{align*}
where we used \eqref{irprime.lb} and the fact that $|I_r(A)|\le |\mN_{A^\tran}(L)| \le d|L|\le n/16$.
Applying Lemma \ref{lem:conc} with $S= I_r(A)\cap I$, $T= (I_r(A)\cap I)^c$, $T_i = I_r'(A; i)$, $\sigma^{-1}\circ \pi|_{I_r(A)\cap I}$ in place of $\sigma$, and $\delta =1/2$, say, we conclude \eqref{iplus.goal2} holds except with probability 
\[
O(\expo{-c'm}) = O(\expo{-c\kappa d|L|})
\]
as desired.
\end{proof}

\subsection{Incrementing flatness}

Now we use the coupled pair $(A,F_{L,L'}(A))$ from Lemma \ref{lem:FA} to boost the weak control on structured vectors in Lemma \ref{lem:very} to the stronger control of Proposition \ref{prop:flat}.
We do this by iterative application of the following:

\begin{lemma}[Incrementing flatness]	\label{lem:increment}
There is a constant $c_0>0$ such that the following holds.	
Fix $\gamma_0\ge 1/2$ and let $1\le K\le n^{\gamma_0}$.
Let $\kappa\in (0,1)$ and assume $ 1/(c_0\kappa)^2\le d \le c_0\kappa n$. Let 
\begin{equation}	\label{incr.assumelb}
e^{ -\gamma_0\log^2n} \le \rho < 1, \qquad \frac1{c_0^2\kappa} \le m \le  \frac{c_0 n}{d},
\end{equation}
and let $\rho',m'$ satisfy
\begin{equation}	\label{assume:ratios}
0<\rho'\le \left(\frac{c_0}{Kn}\right) \rho, \qquad m<m' \le \left(\frac{c_0\kappa d}{\gamma_0\log^2n}\right)m	.	
\end{equation}
Let $A\in \mA_{n,d}$ be drawn uniformly at random, and recall the events \eqref{def:event}.
We have
\begin{equation}	\label{incr:goal1}
\pr\Big( \event_K(m',\rho')\,\wedge\, \event_K(m,\rho)^c \,\wedge \,\big\{ A\in \mA^{\expand}(\kappa )\big\}\Big) \le 
\expo{ -c_0 \kappa dm}.
\end{equation}
\end{lemma}

We remark that at least one of the assumptions \eqref{incr.assumelb} and \eqref{assume:ratios} on $m,m'$ is vacuous unless
\begin{equation}	\label{incr:dbds}
\frac{\gamma_0}{c_0\kappa}\log^2n < d \le c_0^3\kappa n.
\end{equation}

\begin{proof}
Let $c_0>0$ to be taken sufficiently small. 
Let $m,m',\rho,\rho'$ be as in the statement of the lemma. 

For fixed $u\in \sph$ and $t>0$ define
\begin{equation}
\mC(u,t) := \big\{A_0\in \mA_{n,d}: \|(A_0+Z)u\|\le t\big\}.
\end{equation}
We first use the couplings $(A,F_{L,L'}(A))$ from Section \ref{sec:couplings} to show that 
\begin{equation}	\label{incr.goal1}
\sup_{u \in \Flat_0(m',\rho')\setminus \Flat_0(m,\rho)} \pro{ A\in \mC\left(u, c_0\rho\sqrt{\frac{\kappa dm}{n}}\right)\cap \mA^{\exp}(\kappa)} \ll \expo{ -c_0\kappa dm}
\end{equation}
if $c_0>0$ is taken sufficiently small.
Fix an arbitrary element $u\in \Flat_0(m',\rho')\setminus \Flat_0(m,\rho)$. 
Since
\begin{align*}
\Flat_0(m',\rho')\setminus \Flat_0(m,\rho) 
&= \sph_0 \cap \Flat(m',\rho') \cap \Flat(m,\rho)^c\\
&\subset \sph\setminus \Flat(m,\rho)
\end{align*}
we can apply Lemma \ref{lem:gap} to obtain disjoint sets $L,L'\subset[n]$ with $|L|\gg m$ and $|L'| \gg n-m \gg n$ such that for any nonempty sets $J\subset L, J'\subset L'$ with $|J|=|J'|$, 
\begin{equation}	\label{incr:avglb}
\left| \sum_{j\in J}u_j- \sum_{j\in J'} u_j\right| \ge \frac{\rho}{4\sqrt{n}}.
\end{equation}
By deleting elements we may take $|L|=\lf cm\rf$ and $|L'|=\lf cn\rf$ for some constant $c>0$.

Now let $\tA= F_{L,L'}(A)$ be as in \eqref{def:FLLA}. By Lemma \ref{lem:FA} we have $\tA\eqd A$, so
\[
\pro{ A\in \mC\left(u, c_0\rho\sqrt{\frac{\kappa dm}{n}}\right)\cap \mA^{\exp}(\kappa)} = \pro{\tA\in \mC\left(u, c_0\rho\sqrt{\frac{\kappa dm}{n}}\right), \tA\in \mA^{\exp}(\kappa)}.
\]
From \eqref{def:FLLA} and the fact that the switching mappings are their own inverses it follows that
\[
A = \Big( \circ_{i\in I^+(A)\cup I^-(A)} \Switch^{\xi_i}_{\{i,\pi(i)\}, J_i,J_i'}\Big)(\tA).
\]
From Lemma \ref{lem:Aexp.stable} we deduce that $\{\tA\in \mA^{\exp}(\kappa)\}\subset \{A\in \mA^{\exp}(\kappa/2)\}$, so
\[
\pro{ A\in \mC\left(u, c_0\rho\sqrt{\frac{\kappa dm}{n}}\right)\cap \mA^{\exp}(\kappa)} \le \pro{ \tA\in \mC\left(u, c_0\rho\sqrt{\frac{\kappa dm}{n}}\right), A\in \mA^{\exp}(\kappa/2)} .
\]
For the right hand side, letting $c'>0$ a sufficiently small constant and applying the union bound,
\begin{align*}
&\pro{ \tA\in \mC\left(u, c_0\rho\sqrt{\frac{\kappa dm}{n}}\right), A\in \mA^{\exp}(\kappa/2)} \\
&\qquad \le \pro{ \tA\in \mC\left(u, c_0\rho\sqrt{\frac{\kappa dm}{n}}\right), |I^+(A)|>c'\kappa dm}
 + \pro{ A\in \mA^{\exp}(\kappa/2), |I^+(A)|\le c'\kappa dm} .
 \end{align*}
By our assumed bounds on $d$ and $m$ and taking $c_0,c'$ sufficiently small we can apply Lemma \ref{lem:iplus} to bound the second term above by $O(\expo{ -c'\kappa dm})$.
Thus, taking $c_0\le c'$, to obtain \eqref{incr.goal1} it suffices to show
\begin{equation}	\label{incr.goal2}
\pro{ \tA\in \mC\left(u, c_0\rho\sqrt{\frac{\kappa dm}{n}}\right), |I^+(A)|>c_0\kappa dm} \ll \expo{ -c_0\kappa dm}.
\end{equation}

Condition on $A$ satisfying $|I^+(A)|>c_0\kappa dm$, and consider an arbitrary element $i\in I^+(A)$. 
Writing $R_i,\wR_i$ for the $i$th rows of $A+Z$ and $\wA+Z$, respectively, we have
\[
\wR_{i} = R_{i} + \xi_i \bigg( \sum_{j\in J_i' } u_j - \sum_{j\in J_i} u_j \bigg).
\]
From \eqref{incr:avglb} it follows that
\begin{equation}
\pr_{\xi_i}\bigg( \big|\wR_{i} \cdot u\big|<\frac{\rho}{8\sqrt{n}} \bigg) \le 1/2.
\end{equation}
By independence of the components of $\bxi$ we can apply Lemma \ref{lem:tensorize} to obtain that if $c_0$ is sufficiently small, 
\begin{align*}
\pr_{\bxi}\left( \|(\tA+Z) u\| \le c_0\rho\sqrt{\frac{\kappa dm}{n}} \right)
&= \pr_{\bxi}\bigg( \sum_{i=1}^n \big|\wR_i\cdot u\big|^2 \le c_0^2  \rho^2 \kappa dm/n  \bigg)\\
&\le  \pr_{\bxi}\bigg( \sum_{i\in I^+(A)}\big|\wR_{i_l}\cdot u\big|^2 \le c_0 \rho^2 |I^+(A)|/n \bigg)\\
&\le 2^{-|I^+(A)|} \le e^{-c_0\kappa dm}.
\end{align*}
Undoing the conditioning on $A$ we obtain \eqref{incr.goal2}, and hence \eqref{incr.goal1}.

By Lemma \ref{lem:flatnet} we may fix a $\rho'$-net $\Sigma_0=\Sigma_0(m',\rho')\subset \Flat_0(m',\rho')$ for $\Flat_0(m',\rho')$ with $|\Sigma_0|\le O(n/m'\rho'^2)^{m'}$.
By definition, on the event $\event_K(m',\rho')$ we have $A\in \mC(u,\rho'K\sqrt{d})$ for some $u\in \Flat_0(m',\rho')$. Arguing as in the proof of Lemma \ref{lem:very} it follows that $A\in \mC(u',2\rho'K\sqrt{d})$ for some $u'\in \Sigma_0$.
Thus, by our assumption $\rho'<\rho/2$,
\begin{align}
&\pro{ \event_K(m',\rho')\wedge \event_K(m,\rho)^c \wedge\big\{A\in  \mA^{\expand}(\kappa)\big\}}	\notag\\
&\qquad\le \sum_{u\in \Sigma_0} \pro{ \event_K(m,\rho)^c \wedge\Big\{ A\in \mC\big(u,2\rho'K\sqrt{d}\big) \cap  \mA^{\expand}(\kappa) \Big\}}	\notag\\
&\qquad\le \sum_{u\in \Sigma_0\setminus \Flat_0(m,\rho)} \pro{ A\in \mC\big(u,2\rho'K\sqrt{d}\big)  \cap \mA^{\expand}(\kappa)}.	\label{incr.sumover}
\end{align}
Now by the assumed bound \eqref{assume:ratios} on $\rho'/\rho$, we have
$
2\rho'K\sqrt{d} \le c_0\rho\sqrt{\kappa dm/n}
$
for all $n$ sufficiently large,
and hence
\[
\mC\big(u,2\rho'K\sqrt{d}\big) \subset  \mC\left(u, c_0\rho\sqrt{\frac{\kappa dm}{n}}\right)
\]
for any $u\in \sph$.
Combining \eqref{incr.sumover} with \eqref{incr.goal1} we have
\begin{align*}
&\pr\Big( \event_K(m',\rho')\,\wedge\, \event_K(m,\rho)^c \,\wedge \,\big\{ A\in \mA^{\expand}(\kappa )\big\}\Big) \\
&\qquad\qquad\qquad \ll |\Sigma_0| \expo{ -c_0\kappa dm}\\
&\qquad\qquad\qquad \le O\left(\frac{n}{m'\rho'^2}\right)^{m'} \expo{ -c_0\kappa  dm}\\
&\qquad\qquad\qquad \le \expo{ O(m' \log(n/\rho') )- c_0\kappa d m}\\
&\qquad\qquad\qquad \le \expo{ O(m'( \log (Kn) + \log (1/\rho))) -c_0\kappa dm} \\
&\qquad\qquad\qquad \le \expo{ O(m'\gamma_0\log^2n -c_0\kappa dm} \\
&\qquad\qquad\qquad \le \expo{ -\frac12c_0 \kappa dm}
\end{align*}
where in the fifth and final lines we applied the upper bounds \eqref{assume:ratios}, and in the sixth line the lower bound \eqref{incr.assumelb} on $\rho$.
The result follows by replacing $c_0$ with $2c_0$.
\end{proof}

\begin{proof}[Proof of Proposition \ref{prop:flat}]

We may and will assume throughout that $n$ is sufficiently large depending on $\gamma_0$.
Since Lemma \ref{lem:very} already implies Proposition \ref{prop:flat} when $d \gg n/\log^2n$, we may assume
\begin{equation}	\label{struct:dub}
d\le \frac{n}{\log^2n}.
\end{equation}
In the sequel, we will frequently apply the observation that the events $\event_K(m,\rho)$ are monotone increasing in the parameters $m$ and $\rho$.

For $k\ge 0$ set
\begin{equation}
m_k:= \left(\frac{c'd}{\gamma_0\log^3n}\right)^k, \qquad \rho_k:= n^{-10\gamma_0k}
\end{equation}
where $c'>0$ is a sufficiently small constant, and denote
\[
\event_k := \event_K(m_k,\rho_k).
\]
Note that the sequence $m_k$ is monotone increasing (if $n$ is sufficiently large) by our assumption $d\ge \log^4n$. 
From Lemma \ref{lem:very} and monotonicity,
\begin{equation}	\label{bd:event1}
\pro{ \event_1} = O(e^{-cd}).
\end{equation}
Let $k^*\ge 0$ be the integer such that
\begin{equation}	\label{kstar}
m_{k^*}\le \frac{c_0n}{d} < m_{k^*+1}
\end{equation}
where the constant $c_0$ is as in Lemma \ref{lem:increment}.

Let us denote $\kappa=c/\log n$ for a sufficiently small constant $c>0$.
By \eqref{struct:dub} we have $d\le  c_0 \kappa n$ for all $n$ sufficiently large. Also by \eqref{struct:dub},
\[
\frac{c_0n}{d} \ge c_0 \log^2n\ge \frac1{c_0^2\kappa}
\]
for all $n$ sufficiently large, so the range for $m$ in \eqref{incr.assumelb} is nonempty. 
From the definitions of $k^*$ and $m_k$ we have
\begin{equation}	\label{struct:kasymp}
k^* \ll \frac{\log n}{\log d}.
\end{equation}
Let $m'>c_0n/d$ to be specified later. By the union bound and monotonicity of $\event_K(m,\rho)$ in $m$ and $\rho$,
\begin{align}
\pro{ \event_K(m',\rho_{k^*+2})} &\le \pro{ A\notin \mA^{\exp}(\kappa)} + \pro{ \event_1} 
+ \sum_{k=1}^{k^*} \pro{ \event_{k+1}\wedge \event_k^c \wedge \big\{ A \in \mA^{\exp}(\kappa)\big\}}	\notag\\
&\quad + \pro{ \event_K(m',\rho_{k^*+2})\wedge \event_{k^*+1}^c\wedge \big\{ A \in \mA^{\exp}(\kappa)\big\}}	 .\label{struct:split}
\end{align}
From \eqref{struct:kasymp} and the assumed lower bound on $d$ we have
\[
\rho_{k^*+1} = n^{-10\gamma_0(k^*+1)} \ge \expo{ -\gamma_0\log^2n}
\]
for all $n$ sufficiently large. 
Thus, taking	
\[
m'= \frac{c_0^2 c n}{\gamma_0\log^3n} = \frac{c_0n}{d}\times \frac{c_0\kappa d}{\gamma_0\log^2n}
\]
we can apply Lemma \ref{lem:increment} to bound
\[
\pro{ \event_K(m',\rho_{k^*+2})\wedge \event_{k^*+1}^c\wedge \big\{ A \in \mA^{\exp}(\kappa)\big\}} \le \expo{-c_0^2\kappa n}.
\]
From monotonicity of $\event_K(m,\rho)$ in $m$ and \eqref{kstar}, 
\begin{equation}	\label{struct:end}
\pro{ \event_K(m',\rho_{k^*+2})\wedge \event_K\left( \frac{c_0n}{d}, \rho_{k^*+1}\right)^c\wedge \big\{ A \in \mA^{\exp}(\kappa)\big\}} \le \expo{-c_0^2\kappa n}.
\end{equation}
For each $1\le k\le k^*$, since $\rho_k\ge \rho_{k^*+1}$ we can also apply Lemma \ref{lem:increment} to bound
\begin{equation}	\label{struct:eventk}
\pro{ \event_{k+1}\wedge \event_k^c \wedge \big\{ A \in \mA^{\exp}(\kappa)\big\}}  \le \expo{ -c_0\kappa dm_k}.
\end{equation}
Assembling the bounds \eqref{struct:split}, \eqref{bd:event1}, \eqref{struct:end}, \eqref{struct:eventk} and applying Lemma \ref{lem:expand} we conclude
\begin{align*}
\pro{ \event_K(m',\rho_{k^*+2})} &\ll e^{-cd} + e^{-cn/\log n} + \sum_{k=1}^{k^*} \expo{ -c_0\kappa dm_k} \le e^{-cd/2}.
\end{align*}
Since $\rho_{k^*+2} = n^{-10\gamma_0(k^*+2)} = n^{-O(\gamma_0\log_dn)}$ by \eqref{struct:kasymp}, the claim follows after adjusting $c$.
\end{proof}

\section{Invertibility over non-flat vectors}		\label{sec:unstructured}

In this section we conclude the proof of Theorem \ref{thm:ssv}.
Throughout this section $A$ and $Z$ are as in the theorem statement.

\subsection{An averaging argument} 

The aim of this subsection is to establish Lemma \ref{lem:goodoverlap} below, which essentially reduces our task to bounding the probability that the difference between two rows of $A+Z$ has small inner product with a fixed unit vector. 
The statement and proof were inspired by arguments in \cite{LLTTY} for the invertibility problem, which in turn are an intricate refinement of a basic averaging argument that goes back to Koml\'os in his work on the invertibility of random $\pm1$ matrices with iid entries \cite{Komlos77}.

\begin{definition}[Good overlap events]	\label{def:mO}
For $i_1,i_2\in [n]$, $\ell \ge1$ and $\rho, t >0$ we define $\mO_{i_1,i_2}(\ell, \rho,t)$ to be the event that there exists $u\in \sph_0$ and disjoint sets $L_1,L_2\subset  \mN_A(i_1) \triangle \mN_A(i_2)$ such that the following hold:
\begin{enumerate}
\item $|L_1|=|L_2|=\ell$,
\item $|u_{j_1} - u_{j_2}| \ge \frac{\rho}{\sqrt{n}}$ for all $j_1\in L_1,j_2\in L_2$, 
\item $\|(A+Z)^{(i_1,i_2)}u\| \le \frac{t}{\sqrt{n}}$, and
\item $|(R_{i_1}+ R_{i_2}) \cdot u| \le \frac{2t}{\sqrt{n}}$.
\end{enumerate}
Recall our notation $A^{(i_1,i_2)}$ for the matrix obtained by removing the rows with indices $i_1,i_2$, and write $\langle A^{(i_1,i_2)}\rangle$ for the sigma algebra it generates. 
Crucially, we note that $\mO_{i_1,i_2}(\ell,\rho,t)$ is a $\langle A^{(i_1,i_2)}\rangle$-measurable event. Indeed, conditioning on $A^{(i_1,i_2)}$ fixes $R_{i_1}+R_{i_2}$ and $\mN_A(i_1) \triangle \mN_A(i_2)$ by the constraint on column sums.

For each pair of distinct indices $i_1,i_2\in [n]$ we choose a $\langle A^{(i_1,i_2)}\rangle$-measurable random vector 
\[
u^{(i_1,i_2)}\in \sph_0
\]
and $\langle A^{(i_1,i_2)}\rangle$-measurable sets 
\[
L_1(i_1,i_2), \, L_2(i_1,i_2)\subset  \mN_A(i_1) \triangle \mN_A(i_2)
\]
which, on the event $\mO_{i_1,i_2}(\tau,\rho,\ell)$, satisfy the properties (1)--(4) stated for $u, L_1,L_2$;
off this event we define $u^{(i_1,i_2)}, L_1(i_1,i_2)$, and $L_2(i_1,i_2)$ in some arbitrary (but $\langle A^{(i_1,i_2)}\rangle$-measurable) fashion. 	
\end{definition}

Informally, $\mO_{i_1,i_2}(\ell,\rho,\tau)$ is the event that there is a unit vector which is ``almost normal" to the span of the $n-1$ vectors $\{R_i: i\notin \{i_1,i_2\}\} \cup \{R_{i_1}+R_{i_2}\}$, and which also has ``high variation" on $\mN_A(i_1) \triangle \mN_A(i_2)$.
The key property of $u^{(i_1,i_2)}, L_1(i_1,i_2)$, and $L_2(i_1,i_2)$ is that they are fixed upon conditioning on $A^{(i_1,i_2)}$.
We will eventually be able to restrict to these events with parameters $\ell\ge d/\log^{O(1)}n$ and $\rho,t \ge n^{-O(\Gamma)}$.

For $m\ge1$ and $\rho,t>0$ we define the good event
\begin{equation}	\label{def:notflat}
\good(m,\rho,t) = \left\{ \forall u, v\in \Flat_0(m,\rho),\; \min\big(\|(A+Z)u\|,\|(A+Z)^*v\|\big)> \frac{t}{\sqrt{n}}\right\}.
\end{equation}
Recall also the set $\mA^{\edge}(n_0,\delta)\subset\mA_{n,d}$ from \eqref{event:discrep}.

\begin{lemma}[Good overlap on average]	\label{lem:goodoverlap}
Assume $1\le d\le n/2$. 
Let $1\le m\le c_1n$ for a sufficiently small constant $c_1>0$, and put $\ell = md/(8n)$.
Then for all $\rho>0$ and $0<t\le n^{-10}$  we have
\begin{align}
&\pr\bigg( \good(m,\rho,t) \wedge \left\{ s_n(A+Z) \le \frac{t}{\sqrt{n}}, \; A\in \mA^{\edge}\left(\frac{m}{8},\frac12\right)\right\}\bigg) \notag\\
&\qquad \le \frac{2}{mn} \sum_{i_1,i_2=1}^n\pro{ \mO_{i_1,i_2}(\ell,  \rho/2,t) \wedge\left\{ \big|(R_{i_1}-R_{i_2})\cdot u^{(i_1,i_2)} \big| \le \frac{8t}{\rho}\right\}}.	\label{bd:overlap}
\end{align}
\end{lemma}

\begin{proof}
Suppose the event on the left hand side of \eqref{bd:overlap} holds.
Let $u,v\in \sph$ be the eigenvectors of $(A+Z)^*(A+Z)$, $(A+Z)(A+Z)^*$, respectively, associated to the eigenvalue $s_n(A+Z)^2$. By our hypotheses and \eqref{pf} we have
\[
(A+Z)^*(A+Z)\1 = (A+Z)(A+Z)^*\1 = |d+\zeta|^2\1.
\]
Then since
\[
s_n(A+Z)\le \frac{t}{\sqrt{n}}<n^{-10}\le |d+\zeta|,
\]
it follows that $u$ and $\1$ are associated to distinct eigenvalues of $(A+Z)^*(A+Z)$ and hence $u\perp \1$; we similarly have $v\perp \1$. Thus, we have located $u,v\in \sph_0$ such that 
\[
\|(A+Z)u\|, \, \|(A+Z)^*v\| \le \frac{t}{\sqrt{n}}.
\]
Furthermore, by our restriction to $\good(m,\rho,t)$ we have that $u,v\in \sph_0\setminus \Flat(m,\rho)$.

Our next step is to find a large set $\mI^*(u,v)$ of pairs of row indices $(i_1,i_2)$ for which $\mO_{i_1,i_2}(\ell,\rho/2,t)$ holds, and for which $|v_{i_1}-v_{i_2}|$ is reasonably large. 
We begin with the former -- that is, counting pairs of row indices that are ``good" for $u$.
By Lemma \ref{lem:gap} there are disjoint sets $J_1,J_2\subset[n]$ with $|J_1|\ge m$, $|J_2|\gg n-m$ such that
\begin{equation}
|u_{j_1}- u_{j_2}| \ge \frac{\rho}{2\sqrt{n}} \qquad \forall \;j_1\in J_1,\; j_2\in J_2.
\end{equation}
Let us take the constant $c_1$ sufficiently small that $|J_2|\ge m$.
For $\alpha=1,2$ put 
\[
I_\alpha^0(u) = \left\{ i\in [n] : \big|\mN_A(i)\cap J_\alpha \big|\le \frac{3d}{2n}|J_\alpha|\right\}.
\]
By our restriction to $\mA^{\edge}(m/8,1/2)$ we have $|I_\alpha^0(u)|\ge n-m/8$ for $\alpha=1,2$.
Indeed, if this were not the case we would have $|I_\alpha^0(u)^c|>m/8$, so
\[
\frac{3d}{2n} |I_\alpha^0(u)^c||J_\alpha| > e_A\big( I_\alpha^0(u)^c, J_\alpha\big) = \sum_{i\in I_\alpha^0(u)^c} \big|\mN_A(i) \cap J_\alpha\big| > \frac{3d}{2n} |I_\alpha^0(u)^c||J_\alpha| \,,
\]
a contradiction. 
Now for $\alpha\in \{1,2\}$ and $i_1\in [n]$ put
\[
I_\alpha^1(u; i_1) = \left\{ i\in [n]\setminus \{i_1\}: \big| \mN_A(i)\cap\mN_A(i_1)^c\cap  J_\alpha \big|\ge \frac{d}{2n}|\mN_A(i_1)^c\cap J_\alpha|\right\}.
\]
For $i_1\in I_\alpha^0(u)$ we have
\[
|\mN_A(i_1)^c\cap J_\alpha| \ge \left( 1-\frac{3d}{2n} \right) |J_\alpha| \ge \frac{m}4
\]
where in the last inequality we applied our assumption $d\le n/2$.
Thus, using our restriction to $\mA^{\edge}(m/8,1/2)$ and arguing by contradiction as above, we find that $|I_\alpha(u;i_1)| \ge n-m/8$ for $\alpha\in \{1,2\}$ and any $i_1\in I_\alpha^0(u)$.
Setting 
\[
\mI_\alpha(u) = \left\{ (i_1,i_2)\in [n]^2 : \big| \mN_A(i_2)\cap \mN_A(i_1)^c \cap   J_\alpha\big| \ge \frac{md}{8n}\,\right\}
\]
we have
\[
|\mI_\alpha(u)| \ge \sum_{i_1\in I_\alpha^0(u) } |I_\alpha^1(u;i_1)| \ge \left(n-\frac{m}8\right)^2>n^2-\frac{mn}{4}.
\]
Setting $\mI(u):= \mI_1(u)\cap \mI_2(u)$, 
\begin{equation}	\label{lb:uready}
|\mI(u)^c|\le |\mI_1(u)^c|+ |\mI_2(u)^c| < 2\frac{mn}{4} = \frac{mn}2.
\end{equation}
Now for each $\alpha=1,2$ and $(i_1,i_2)\in \mI(u)$, since
\[
|\mN_A(i_2) \cap  \mN_A(i_1)^c\cap J_\alpha| \ge \frac{md}{8n}= \ell
\]
we may select a subset $L'_\alpha (i_1,i_2)$ of cardinality $\ell$.
We have that $L'_1(i_1,i_2),L'_2(i_1,i_2)$ are disjoint subsets of $\mN_A(i_1)\triangle \mN_A(i_2)$ such that
\begin{equation}
|u_{j_1}-u_{j_2}|\ge \frac{\rho}{2\sqrt{n}} \quad \forall j_1\in L'_1(i_1,i_2), \;j_2\in L'_2(i_1,i_2).
\end{equation}
Furthermore, for each $(i_1,i_2)\in [n]^2$,
\[
\|(A+Z)^{(i_1,i_2)}u\| \le \|(A+Z)u\| \le \frac{t}{\sqrt{n}}
\]
and 
\[
|(R_{i_1}+R_{i_2})\cdot u| \le |R_{i_1}\cdot u|+ |R_{i_2}\cdot u| \le 2\|(A+Z)u\|\le \frac{2t}{\sqrt{n}}.
\]
Thus, $u, L'_1(i_1,i_2), L'_2(i_1,i_2)$ satisfy the conditions on $u, L_1, L_2$ in Definition \ref{def:mO}, and it follows that $\mO_{i_1,i_2}(\ell,\rho/2,t)$ holds for each $(i_1,i_2)\in \mI(u)$. 

Now we count pairs of row indices that are ``good" with respect to $v$.
By Lemma \ref{lem:ac} and the fact that $v\in \sph\setminus\Flat(m,\rho)$  we have that for any $i_1\in [n]$, 
\begin{equation}	\label{v:notflat}
|E_v(v_{i_1}\sqrt{n}, \rho)^c| = \left|\left\{i\in [n]\setminus\{i_1\}: |v_{i_1}-v_i|\ge \frac{\rho}{\sqrt{n}}\right\}\right| \ge m.
\end{equation}
Setting
\begin{equation}
\mI^*(u,v) := \left\{ (i_1,i_2)\in \mI(u): |v_{i_1}-v_{i_2}|\ge \frac{\rho}{\sqrt{n}}\right\}
\end{equation}
from \eqref{v:notflat} and \eqref{lb:uready} we have
\begin{equation}
|\mI^*(u,v)| \ge mn-\frac{mn}2 =\frac{mn}2.
\end{equation}

Fix $(i_1,i_2)\in \mI^*(u,v)$. 
By several applications of the fact that $\mO_{i_1,i_2}(\ell,\rho/2,t)$ holds and the Cauchy--Schwarz inequality,
\begin{align*}
\frac{t}{\sqrt{n}} & \ge \big\|v^*(A+Z)\big\| \ge \big|v^*(A+Z)u^{(i_1,i_2)}\big|  = \left| \sum_{i=1}^n \overline{v}_i R_i\cdot u^{(i_1,i_2)}\right|	\\
&\ge \big|(\overline{v}_{i_1} R_{i_1}+ \overline{v}_{i_2}R_{i_2})\cdot u^{(i_1,i_2)}\big| -  \left| \sum_{i\in [n]\setminus\{i_1,i_2\}} \overline{v}_i R_i\cdot u^{(i_1,i_2)}\right| \\
&\ge \big|(\overline{v}_{i_1} R_{i_1}+ \overline{v}_{i_2}R_{i_2})\cdot u^{(i_1,i_2)}\big| - \big\|(A+Z)^{(i_1,i_2)}u^{(i_1,i_2)}\big\|\\
&\ge \big|(\overline{v}_{i_1} R_{i_1}+ \overline{v}_{i_2}R_{i_2})\cdot u^{(i_1,i_2)}\big| - \frac{t}{\sqrt{n}}\\
&= \frac12\left|(\overline{v}_{i_1} + \overline{v}_{i_2})(R_{i_1}+R_{i_2}) \cdot u^{(i_1,i_2)} + (\overline{v}_{i_1}-\overline{v}_{i_2})(R_{i_1}-R_{i_2})\cdot u^{(i_1,i_2)}\right| - \frac{t}{\sqrt{n}}\\
&\ge \frac12 \left|(\overline{v}_{i_1}-\overline{v}_{i_2})(R_{i_1}-R_{i_2})\cdot u^{(i_1,i_2)}\right| - \left|(R_{i_1}+R_{i_2}) \cdot u^{(i_1,i_2)}\right|-\frac{t}{\sqrt{n}}\\
&\ge \frac12 \left|(\overline{v}_{i_1}-\overline{v}_{i_2})(R_{i_1}-R_{i_2})\cdot u^{(i_1,i_2)}\right| - \frac{3t}{\sqrt{n}}\\
&\ge \frac{\rho}{2\sqrt{n}} \left|(R_{i_1}-R_{i_2})\cdot u^{(i_1,i_2)}\right|-\frac{3t}{\sqrt{n}}.
\end{align*}
Rearranging we have
\[
\big|(R_{i_1}-R_{i_2})\cdot u^{(i_1,i_2)}\big|\le \frac{8t}{\rho}.
\]
In summary, we have shown that on the event 
\[
\good(m,\rho,t) \wedge \left\{ s_n(A+Z) \le \frac{t}{\sqrt{n}},\, A\in \mA^{\edge}\left(\frac{m}8,\frac12\right)\right\},
\]
the event $\mO_{i_1,i_2}(\ell,\rho/2,t)\wedge \{ |(R_{i_1}-R_{i_2})\cdot u^{(i_1,i_2)}|\le 8t/\rho\}$ holds for at least $mn/2$ pairs $(i_1,i_2)\in [n]^2$ -- that is,
\begin{align*}
&\frac{mn}{2}\un_{\good(m,\rho,t)} \un\left( s_n(A+Z) \le \frac{t}{\sqrt{n}},\, A\in \mA^{\edge}\left(\frac{m}8,\frac12\right)\right) \\
&\qquad\qquad\qquad\qquad\le 
\sum_{i_1,i_2=1}^n \un_{\mO_{i_1,i_2}(\ell,\rho/2,t)}\un\left(|(R_{i_1}-R_{i_2})\cdot u^{(i_1,i_2)}|\le \frac{8t}{\rho} \right).
\end{align*}
The bound \eqref{bd:overlap} now follows by taking expectations on each side and rearranging.
\end{proof}

\subsection{Anti-concentration for row-pair random walks}

The aim of this subsection is to prove the following: 

\begin{lemma}	\label{lem:walk}
Assume $1\le d\le n/2$.
Let $i_1,i_2\in [n]$ distinct, and suppose the event $\mO_{i_1,i_2}(\ell,\rho,t)$ holds for some $\ell\ge1$, $\rho,t>0$. Then for all $r\ge 0$,
\begin{equation}
\pro{ \big|(R_{i_1}-R_{i_2})\cdot u^{(i_1,i_2)} \big|\le r \,\Big|\, A^{(i_1,i_2)} } \ll  \left( 1+ \frac{r\sqrt{n}}{\rho}\right) \left(\frac{\log(n/\rho)}{\ell}\right)^{1/2}+  e^{-c\ell}.
\end{equation}
\end{lemma}

We will need the following standard anti-concentration bound of Berry--Ess\'een-type; see for instance \cite[Lemma 2.7]{Cook:ssv} (the condition there of $\kappa$-controlled second moment is easily verified to hold with $\kappa=1$ for a Rademacher variable).

\begin{lemma}[Berry--Ess\'een-type small ball inequality]	\label{lem:be}
Let $v\in \C^n$ be a fixed nonzero vector and let $\xi_1,\dots, \xi_n$ be iid Rademacher variables. 
Then for $r\ge 0$,
\[
\sup_{z\in \C} \pr\bigg( \bigg| z+ \sum_{j=1}^n \xi_j v_j \bigg| \le r\bigg) \ll \frac{r+ \|v\|_\infty}{\|v\|}.
\]
\end{lemma}

\begin{proof}[Proof of Lemma \ref{lem:walk}]
By symmetry we may take $(i_1,i_2)=(1,2)$. 
We condition on a realization of $A^{(1,2)}$ satisfying the event $\mO_{1,2}(\ell,\rho,t)$ for the remainder of the proof. 
For ease of notation we will write $u, L_1,L_2$ instead of $u^{(1,2)}, L_1(1,2), L_2(1,2)$.
Note that conditioning on $A^{(1,2)}$ fixes $u, L_1, L_2$. Let $r\ge0$ be arbitrary. Our aim is to show
\begin{equation}	\label{walk:goal0}
 \pro{ \left| \big( R_1-R_2\big) \cdot u \right| \le r \ \middle|\ A^{(1,2)}} 
\ll \left( 1+ \frac{r\sqrt{n}}{\rho}\right)\left(\frac{\log(n/\rho)}{\ell}\right)^{1/2} + e^{-c\ell}.
\end{equation}

We now construct a coupled matrix $\tA\in \mA_{n,d}$ on an augmented probability space. 
Under the conditioning on $A^{(1,2)}$, by the constraint on column sums the only remaining randomness is in the submatrix $A_{(1,2)\times \mN_A(1)\triangle \mN_A(2)}$. We create $\tA$ by resampling a sequence of $2\times 2$ submatrices of this submatrix uniformly and independently.
Specifically, we fix a bijection $\pi:L_1\to L_2$ in some arbitrary (but measurable) fashion under the conditioning on $A^{(1,2)}$, and let $\bxi=(\xi_j)_{j=1}^n$ be iid Rademacher variables independent of $A$. 
Conditional on $A$, for each $j\in L_1$ such that
\begin{equation}	\label{switchable2by2}
A_{(1,2)\times(j,\pi(j))} \in \left\{\begin{pmatrix} 1 & 0\\ 0 & 1\end{pmatrix},\, \begin{pmatrix} 0 & 1\\ 1 & 0 \end{pmatrix} \right\}
\end{equation}
we put
\[
\wA_{(1,2)\times (j,\pi(j))} = \begin{pmatrix} 1 & 0\\ 0 & 1\end{pmatrix} \un(\xi_j = +1) +\begin{pmatrix} 0 & 1\\ 1 & 0 \end{pmatrix} \un(\xi_j=-1);
\]
if \eqref{switchable2by2} does not hold then we set $\wA_{(1,2)\times(j,\pi(j))} = A_{(1,2)\times(j,\pi(j))}$. Finally, we set $\wA(i,j)=A(i,j)$ for all $(i,j)\notin \{1,2\}\times L_1\cup L_2$. 
Note that $\tA_{(1,2)\times \mN_A(1)\triangle \mN_A(2)}$ is generated by first sampling $A_{(1,2)\times \mN_A(1)\triangle \mN_A(2)}$ uniformly under the conditioning on $A^{(1,2)}$, and then independently and uniformly resampling the $2\times2$ submatrices $A_{(1,2)\times (j,\pi(j))}$. It readily follows that $\tA_{(1,2)\times [n]}\eqd A_{(1,2)\times [n]}$ conditional on $A^{(1,2)}$.
Denoting the first two rows of $\wA+Z$ by $\wR_1,\wR_2$, by replacing $A$ with $\wA$ in \eqref{walk:goal0} it suffices to show
\begin{equation}	\label{walk:goal1}
 \pro{ \left| \big( \wR_1-\wR_2\big) \cdot u \right| \le r \ \middle|\ A^{(1,2)}} 
\ll \left( 1+ \frac{r\sqrt{n}}{\rho}\right)\left(\frac{\log(n/\rho)}{\ell}\right)^{1/2} + e^{-c\ell}.
\end{equation}

Define
\begin{equation}	\label{def:L1prime}
L_1'= L_1'(A):=\big\{ j\in L_1: \text{\eqref{switchable2by2} holds}\big\}.
\end{equation}
Using this set we can express the dot products $\wR_1\cdot u, \wR_2\cdot u $ in a manner which exposes the dependence on the Rademacher variables $\xi_j$:
\begin{align*}
\wR_1\cdot u & = z_1(A) + \sum_{j\in L_1'} u_j \un (\xi_j=+1) + u_{\pi(j)}\un(\xi_j=-1),\\
\wR_2\cdot u & = z_2(A) + \sum_{j\in L_1'} u_{\pi(j)} \un (\xi_j=+1) + u_{j}\un(\xi_j=-1),
\end{align*}
where $z_1(A), z_2(A)\in \C$ denote quantities that depend only on $A$. Subtracting we obtain
\begin{align}
(\wR_1-\wR_2)\cdot u 
&= z(A) + \sum_{j\in L_1'} (u_j-u_{\pi(j)}) \un(\xi_j=1) + (u_{\pi(j)}- u_j) \un(\xi_j=-1) 	\notag\\
&= z(A) + \sum_{j\in L_1'} \xi_j \partial_j(u)		\label{walkrep}
\end{align}
where $\partial_j(u):= u_j-u_{\pi(j)}$ and $z(A)$ depends only on $A$. 
Since $\pi(j)\in L_2$ for each $j\in L_1'\subset L_1$ we have
\begin{equation}	\label{partialbounds}
\frac{\rho}{\sqrt{n}}\le |\partial_j(u)|\le 2 \qquad \forall j\in L_1'.
\end{equation}

Our first task is to show that
$|L_1'|\gg \ell$
with high probability in the randomness of the first two rows of $A$. The conditioning on $A^{(1,2)}$ has fixed all entries but the submatrix $A_{(1,2)\times \mN_A(1)\triangle \mN_A(2)}$, and has also fixed $m:=|\mN_A(1)\setminus \mN_A(2)|$. Moreover, the submatrix $A_{(1,2)\times \mN_A(1)\triangle \mN_A(2)}$ is determined by the set $\mN_A(1)\setminus \mN_A(2)$, which is uniformly distributed over subsets of $\mN_A(1)\triangle \mN_A(2)$ of size $m$.
Define 
\begin{equation}
L_1'' = \mN_A(1)\cap L_1,\qquad L_2'' = \mN_A(2)\cap \pi(L_1'')\subset L_2.
\end{equation}
We have $|L_1'|\ge |L_2''|$; indeed, $L_2''$ is the image under $\pi$ of the elements of $j\in L_1$ for which the first alternative in \eqref{switchable2by2} holds.
Now we obtain a high probability lower bound on $|L_2''|$ by two applications of the bound \eqref{conc2} in Lemma \ref{lem:conc}. Applying \eqref{conc2} with $T= \mN_A(1)\triangle \mN_A(2)$, $U= \mN_A(1)\setminus \mN_A(2)$, $T_0= L_1$ and $\delta = 1/2$ (say) gives
\begin{equation}	\label{L1prime}
\pro{ |L_1''| \ge \ell/4} \ge 1-2\exp(-c\ell).
\end{equation}
Next, conditional on $L_1''$ we can apply \eqref{conc2} again with $T= (\mN_A(1)\triangle \mN_A(2))\setminus L_1$, $U= (\mN_A(2)\setminus \mN_A(1))\setminus L_1$, $T_0 = \pi(L_1'')$ and $\delta=1/2$. With these choices we have
\[
|T| = 2m-\ell, \qquad |U| = m-(\ell-|L_1''|), \qquad |T_0| = |L_1''|
\] 
so on the event $|L_1''|\ge \ell/4$,
\[
\frac{|U||T_0|}{|T|} = \frac{m-\ell+|L_1''|}{2m-\ell} |L_1''| \ge \frac{\max(m-\ell,\ell/4)}{2m}\frac{\ell}{4} \gg \ell
\]
and we obtain
\begin{align*}
\pro{ |L_2''| \ge c\ell } \un( |L_1''|\ge \ell/4) &\ge \pro{ |L_2''| \ge \frac12 \frac{m-\ell+|L_1''|}{2m-\ell} |L_1''|}\un(|L_1''|\ge \ell/4) \\
&\ge 1 - 2\exp( -c\ell)
\end{align*}
for a sufficiently small constant $c>0$. Combining this bound with \eqref{L1prime} and the lower bound $|L_1'|\ge |L_2''|$ we obtain
\begin{equation}	\label{lb:L2prime}
\pro{ |L_1'| \ge c\ell } = 1-O(e^{-c\ell})
\end{equation}
as desired.
We henceforth condition on a realization of $A$ satisfying $|L_1'|\ge c\ell$. Now it suffices to show
\begin{equation}	\label{walk:goal2}
 \pro{ \left| \big( \wR_1-\wR_2\big) \cdot u \right| \le r }
\ll \left( 1+ \frac{r\sqrt{n}}{\rho}\right)\left(\frac{\log(n/\rho)}{\ell}\right)^{1/2}.
\end{equation}
At this point the only randomness is in the Rademacher variables $\xi_j$.

Next we locate a large subset of $L_1'$ on which the discrete derivatives $\partial_j(u)$ have roughly the same size.
For $k\ge -1$ define $L^{(k)} = \{j\in L_1': 2^{-(k+1)}<|\partial_j(u)|\le 2^{-k}\}$. 
From \eqref{partialbounds} we have 
\[
L_1'\subset \bigcup_{k=-1}^{O(\log(n/\rho))}L^{(k)}
\]
so by the pigeonhole principle we have
\begin{equation}
|L^*|:= |L^{(k)}| \gg \frac{\ell}{\log(n/\rho)}
\end{equation}
for some $k$. Now define $v\in \C^n$ to have components $v_j= \partial_j(u)\un(j\in L^*)$. 
From \eqref{partialbounds} and the fact that the components of $v$ vary by a factor at most $2$ on $L^*$ we have 
\begin{equation}
|v_j|\ge \max\left( \frac{\rho}{\sqrt{n}}, \frac12\|v\|_\infty\right) \qquad \forall \,j\in L^*.
\end{equation}
From this we obtain 
\begin{equation}
\|v\| \ge \frac{\rho}{\sqrt{n}}|L^*|^{1/2} \gg \frac{\rho}{\sqrt{n}}\left(\frac{\ell}{\log(n/\rho)}\right)^{1/2}
\end{equation}
and
\begin{equation}
\|v\|_\infty \le \frac{2\|v\|}{|L^*|^{1/2}} \ll \left(\frac{\log(n/\rho)}{\ell}\right)^{1/2} \|v\|.
\end{equation}
From the expression \eqref{walkrep}, we condition on the variables $(\xi_j)_{j\notin L^*}$ and apply Lemma \ref{lem:be} to get
\begin{align*}
 \pr_{(\xi_j)_{j\in L^*}}\Big( \left| \big( \wR_1-\wR_2\big) \cdot u \right| \le r \Big)
&\le \sup_{z\in \C} \pro{ \bigg| z + \sum_{j\in L^*} \xi_j \partial_j(u)\bigg|\le r }  \\
&= \sup_{z\in \C} \pro{ \left| z + \sum_{j=1}^n \xi_j v_j\right|\le r }  \\
&\ll \frac{r}{\|v\|} + \frac{\|v\|_\infty}{\|v\|}\\
&\ll \left( 1+ \frac{r\sqrt{n}}{\rho}\right)\left(\frac{\log(n/\rho)}{\ell}\right)^{1/2}.
\end{align*}
Undoing the conditioning on $(\xi_j)_{j\notin L^*}$ gives \eqref{walk:goal2} as desired.
\end{proof}

\subsection{Conclusion of proof of Theorem \ref{thm:ssv}}

Fix $\gamma\ge 1$ and let $C_1,\Gamma>0$ to be chosen sufficiently large. We may and will assume that $n$ is sufficiently large depending on $\gamma$. We may also assume
\begin{equation}	\label{concl:dlb}
d\ge \log^{2C_1}n
\end{equation}
as the desired bound holds trivially otherwise. 

From our hypotheses, \eqref{pf} and the triangle inequality
\[
\|A+Z\|_{\oneperp} \le d + n^\gamma \le 2n^\gamma \le n^{\gamma+1}
\]
with probability 1.
Thus, we may restrict to the event $\mB(n^{\gamma+1/2})$. 
Set
\begin{equation}	\label{end.m}
m=\frac{cn}{\gamma \log^3n}
\end{equation}
for a sufficiently small constant $c>0$.
With this choice of $m$ and the lower bound \eqref{concl:dlb}, taking $C_1$ sufficiently large ($C_1\ge 4$ will do),
we can apply Lemma \ref{lem:discrep} to bound
\begin{equation}	\label{end.disc}
\pro{ A\notin \mA^{\edge}\left( \frac{m}8, \frac12\right)} \le n^{O(1)} e^{-cd}.
\end{equation}
Furthermore, recalling the events \eqref{def:notflat}, by applying Proposition \ref{prop:flat} to $A+Z$ and $(A+Z)^*$ and taking the union bound we have
\begin{equation}	\label{end.good}
\pro{ \good\left( m,\, n^{-\Gamma_0}, \, n^{-\Gamma+1/2}\right)} \ge 1-2e^{-cd}
\end{equation}
for some $\Gamma_0\ll \gamma \log_dn$, as long as
\[
n^{-\Gamma} < n^{-\Gamma_0} n^{\gamma+1/2}\sqrt{d}.
\]
We assume henceforth that $\Gamma\ge \Gamma_0$, so that \eqref{end.good} holds.

Applying Lemma \ref{lem:goodoverlap} (taking $\Gamma \ge 10.5$) followed by Lemma \ref{lem:walk}, 
\begin{align*}
& \pro{\good(m, n^{-\Gamma_0}, n^{-\Gamma+1/2}) \wedge \left\{ s_n(A+Z)\le n^{-\Gamma}, \, A\in \mA^{\edge}\left( \frac{m}8, \frac12\right)\right\}}\\
&\quad \le \frac{2}{mn} \sum_{i_1,i_2=1}^n \pr\bigg(  \mO_{i_1,i_2}\left( \frac{md}{8n} , \frac{1}{2}n^{-\Gamma_0}, n^{-\Gamma+1/2}\right) \\
&\qquad\qquad\qquad\qquad \qquad \wedge \left\{ \big| (R_{i_1}-R_{i_2}) \cdot u^{(i_1,i_2)}\big| \le  8n^{ -\Gamma+\Gamma_0 + 1/2}\right\} \bigg) \\
&\quad \ll \frac{n}{m} \left( 1+ n^{-\Gamma+ 2\Gamma_0+1}\right) \left(\frac{n}{md}\right)^{1/2} (\Gamma_0+1/2)^{1/2}\log^{1/2}n + \frac{n}me^{-cmd/n}\\
&\quad \ll_\gamma \frac{1}{\sqrt{d}}\left(\frac{n}m\right)^{3/2}  (1+n^{-\Gamma+2\Gamma_0+1}) \log n + \frac{n}me^{-cmd/n}.
\end{align*}
Taking
\begin{equation}
\Gamma= 2\Gamma_0 + 1 \ll \gamma \log_dn
\end{equation}
and substituting the bounds \eqref{end.disc}, \eqref{end.good} and our choice \eqref{end.m} for $m$, we conclude
\begin{align*}
\pro{ s_n(A+Z) \le n^{-\Gamma}\,} 
&\ll_\gamma  \frac{\log^{5.5}n}{\sqrt{d}} + n^{O(1)}e^{-cd} + (\log^3n) e^{-cd/(\gamma \log^3n)}\\
&\ll  \frac{\log^{5.5}n}{\sqrt{d}} .
\end{align*}
The proof of Theorem \ref{thm:ssv} is complete.

\section{Proof of Theorem \ref{thm:main}: reduction to asymptotics for singular value distributions}		\label{sec:highlevel}

We turn now to the proof of Theorem \ref{thm:main}.
We begin by recalling a lemma concerning the logarithmic potential, which allows us to convert the question of convergence of the ESDs $\mu_{\Abar_n}$ to questions about empirical \emph{singular value} distributions. 
For a Borel probability measure $\mu$ over $\C$ integrating $\log|\cdot |$ in a neighborhood of infinity, the \emph{logarithmic potential} $U_\mu:\C\to (-\infty,\infty]$ is defined
\begin{equation}
U_\mu(z):= -\int_\C\log |\lambda-z| \dd\mu(\lambda).
\end{equation}
Recall from \eqref{def:esvd} our notation $\nu_M$ for the empirical singular value distribution of a matrix $M$.
The following is taken from \cite[Lemma 4.3]{BoCh:survey} (see also \cite[Theorem 1.20]{TaVu:esd}).

\begin{lemma}	\label{lem:logpot}
For each $n\ge1$ let $M_n$ be a random $n\times n$ matrix with complex entries. 
Suppose that for a.e. $z\in \C$, 
\begin{enumerate}
\item there exists a probability measure $\nu_z$ on $\R_+$ such that $\nu_{M_n-z}\to \nu_z$ in probability;
\item the measures $\nu_{M_n-z}$ uniformly integrate the function $s\mapsto \log (s)$ in probability, i.e.\
for every $\eps>0$ there exists $T<\infty$ such that
\begin{equation}	\label{ui:log}
\sup_{n\ge1} \;\pro{ \int_{\{|\log(s)|>T\}} |\log(s)|\,\dd \nu_{M_n-z}(s)  > \eps} \le \eps.
\end{equation}
\end{enumerate}
Then $\mu_{M_n}$ converges in probability to the unique probability measure $\mu$ on $\C$ satisfying
\begin{equation}	\label{Umu}
U_\mu(z) = - \int_0^\infty \log(s) \,\dd\nu_z(s)
\end{equation}
for all $z\in \C$.
\end{lemma}

(We note that \cite{BoCh:survey} uses the weak topology on the space of measures -- defined in terms of bounded continuous test functions -- rather than the vague topology used in this article. However, under the uniform integrability assumption \eqref{ui:log} the assumption (1) above is equivalent to weak convergence in probability.)

We introduce the centered and rescaled matrix
\begin{equation}	\label{def:Y}
Y_n = \frac{1}{\sqrt{p(1-p)}} (A_{n} - p \1\1^\tran) = \sqrt{n}\,\Abar_n - \sqrt{\frac{p}{1-p}}\1\1^\tran
\end{equation}
where we recall our notation $p=d/n$.
The following two propositions, along with Corollary \ref{cor:ssv}, are the main ingredients for establishing the conditions (1) and (2) in Lemma \ref{lem:logpot}, and hence for proving Theorem \ref{thm:main}. The proofs are deferred to later sections.

\begin{prop}[Weak convergence of singular value distributions]	\label{prop:weak}
Assume $\log^4n\le d\le n/2$. For each $z\in \C$ there exists a probability measure $\nu_z$ on $\R_+$ such that $\nu_{\frac1{\sqrt{n}}Y_n-z}\to \nu_z$ in probability. Moreover, the family $\{\nu_z\}_{z\in \C}$ satisfies the relation \eqref{Umu} with $\mu=\muc$.
\end{prop}

\begin{prop}[Anti-concentration of the spectrum]		\label{prop:wegner}
Assume $\log^4n\le d\le n/2$. There are absolute constants $C,c>0$ such that with probability $1-O(e^{-n})$, for all $\eta\in (0, 1]$,
\begin{equation}	\label{anti_spec}
\nu_{\frac1{\sqrt{n}}Y_n-z}([0,\eta]) \le C( \eta + d^{-1/48}).
\end{equation}
\end{prop}

\begin{remark}
From the existence of the weak limit $\nu_z$ provided by Proposition \ref{prop:weak} we obtain $\nu_{\frac1{\sqrt{n}}Y_n-z}(I) =O_z(|I|) + o(1)$ for any fixed interval $I\subset\R_+$. Proposition \ref{prop:wegner} gives a quantitative improvement of this bound for intervals near zero, providing nontrivial estimates down to the ``mesoscopic scale" $\eta \sim d^{-1/48}$. 
Recently there have been major advances establishing quantitative versions of the Kesten--McKay \cite{BHY:km} and semicircle \cite{BKY:semicircle} laws for undirected random regular graphs of (large) fixed degree or growing degree, proving that these limiting laws are a good approximation for the finite $n$ ESDs on intervals $I$ at the near-optimal scale $|I|\sim n^{\eps-1}$. In particular, we expect that the arguments in \cite{BKY:semicircle} could be used to obtain a similar local law for $\nu_{\frac1{\sqrt{n}}Y_n-z}$.
\end{remark}

Now we conclude the proof of Theorem \ref{thm:main} on Propositions \ref{prop:weak} and \ref{prop:wegner}.
We let $C_0>0$ to be taken sufficiently large. 
For the duration of the proof we abbreviate 
$$\nu_{n,z}:= \nu_{\Abar_n-z}.$$
For any $z\in \C$, $\frac1{\sqrt{n}}Y_n-z$ and $\Abar_n-z$ differ by a matrix of rank one.
By a standard eigenvalue interlacing bound it follows that their singular value distributions are close in Kolmogorov distance, specifically:
\begin{equation}	\label{main.kol}
 \sup_{\eta\in \R} \left|\nu_{n,z}([0,\eta]) - \nu_{\frac1{\sqrt{n}}Y_n-z}([0,\eta])\right| \le 1/n.
\end{equation}
By Lemma \ref{lem:logpot}, the above estimate and Proposition \ref{prop:weak}, it suffices to show the measures $\nu_{n,z}$ uniformly integrate the logarithm function, i.e.\ for every $\eps>0$ there exists $T=T(\eps,z)<\infty$ such that 
\begin{equation}	\label{ui:goal1}
\sup_{n\ge1} \pro{ \int_{\{|\log(s)|>T\}} |\log(s)|\dd\nu_{n,z}(s)>\eps}\le \eps.
\end{equation}

First we address the singularity of $\log$ at infinity. We claim that for any $\eps>0$ and $z\in \C$ there exists $T'=T'(\eps,z)>0$ such that
\begin{equation}	\label{logtight}
\int_{e^{T'}}^\infty \log(s) \, \dd \nu_{n,z}(x) \le \eps/2
\end{equation}
almost surely for all $n\ge 1$. 
Indeed, note that for any fixed $z\in \C$,
\begin{align*}
\int_0^\infty s^2 \,\dd\nu_{n,z}(s) 
&= \frac1n \tr\left( \frac1{\sqrt{d(1-d/n)}}A_n - z\I_n\right)^*\left( \frac1{\sqrt{d(1-d/n)}}A_n  - z\I_n\right)\\
&= \frac1n \sum_{i,j=1}^n \left|\frac1{\sqrt{d(1-d/n)}}a_{ij} - z\delta_{ij}\right|^2\\
&\le \frac{2}{n} \sum_{i,j=1}^n \frac{a_{ij}}{ d(1-d/n)} + |z|^2 \delta_{ij}\\
&\le \left(\frac4{nd}\sum_{i,j=1}^n a_{ij}\right)  + 2|z|^2\\
&\ll 1+ |z|^2
\end{align*}
with probability one, where in the fourth line we used our assumption $d\le n/2$ (see Remark \ref{rmk:reduced}), and in the final line we used that $\sum_{i,j=1}^na_{ij} =dn$ for any element $A=(a_{ij})\in \mA_{n,d}$.
Thus, with probability one, for all $T'$ sufficiently large and for all $n$,
\[
\int_{e^{T'}}^\infty \log (s) \, \dd \nu_{n,z}(s) 
\le T'e^{-2T'} \int_{e^{T'}}^\infty s^2 \, \dd \nu_{n,z}(s)
\ll_z T'e^{-2T'}
\]
and we can take $T'$ sufficiently large to obtain \eqref{logtight}.

It remains to control the contribution of small singular values, i.e.\ to show that for some $T=T(\eps,z)\ge T'$,
\begin{equation}	\label{ui:goal2}
\sup_{n\ge1}\pro{ \int_0^{e^{-T}} |\log(s)| d\nu_{n,z}(s)>\eps/2} \le \eps.
\end{equation}
By Corollary \ref{cor:ssv} and taking $C_0\ge C_1'$, there exists $\Gamma=o( \log n)$ such that
\begin{equation}	\label{ui:fromssv}
\pro{ s_n\left( \frac1{\sqrt{d(1-d/n)}}A_n -z\right) \le n^{-\Gamma}} =o(1) .
\end{equation}
Thus, it suffices to show
\begin{equation}	\label{ui:goal3}
\sup_{n\ge1}\pro{ \int_{n^{-\Gamma}}^{e^{-T}} |\log(s)| d\nu_{n,z}(s)>\eps/2} \le \eps/2.
\end{equation}
(While the term $o(1)$ in \eqref{ui:fromssv} may not be smaller than $\eps/2$ for small values of $n$, we can take $T$ larger, if necessary, in \eqref{ui:goal3} to make the left hand side equal to zero unless $n$ is sufficiently large.)
Now since
\[
\sup_{s\in [n^{-\Gamma}, e^{-T}]} |\log(s)|\le \Gamma\log n =o(\log^2n),
\]
from \eqref{main.kol} we have
\begin{equation*}
\left| \int_{n^{-\Gamma}}^{e^{-T}} |\log(s)| d\nu_{n,z}(s) - \int_{n^{-\Gamma}}^{e^{-T}} |\log(s)| d\nu_{\frac1{\sqrt{n}}Y_n-z}(s) \right| =o\left(\frac{\log^2n}{n}\right) = o(1)
\end{equation*}
so again enlarging $T$ depending on $\eps$ if necessary, it suffices to show
\begin{equation}	\label{ui:goal4}
\sup_{n\ge1}\pro{ \int_{n^{-\Gamma}}^{e^{-T}} |\log(s)| d\nu_{\frac1{\sqrt{n}}Y_n-z}(s)>\eps/4} \le \eps/2.
\end{equation}
From the pointwise bound
\[
|\log(s)| \ll \sum_{m=0}^\infty 1_{[0,2^{-m}]}(s),\quad s\in [0,1]
\]
and Fubini's theorem, we have 
\begin{align*}
\int_{n^{-\Gamma}}^{e^{-T}} |\log(s)| d\nu_{\frac1{\sqrt{n}}Y_n-z}(s)
&\ll \sum_{m=0}^\infty \nu_{\frac1{\sqrt{n}}Y_n-z}([n^{-\Gamma}, e^{-T}\wedge 2^{-m}])\\
&= \sum_{m=0}^{O(\Gamma\log n)} \nu_{\frac1{\sqrt{n}}Y_n-z}([n^{-\Gamma}, e^{-T}\wedge 2^{-m}])\\
&\le \sum_{m=0}^{O(\Gamma\log n)} \nu_{\frac1{\sqrt{n}}Y_n-z}([0, e^{-T}\wedge 2^{-m}]).
\end{align*}
Now by Proposition \ref{prop:wegner} (and assuming $C_0\ge 4$), except with probability $O(e^{-n})$, for all $m\ge 0$,
\[
\nu_{\frac1{\sqrt{n}}Y_n-z}([0, e^{-T}\wedge 2^{-m}]) \ll e^{-T}\wedge 2^{-m} + d^{-1/48}.
\]
Summing these bounds over $m$, we have that on this event,
\begin{equation*}
\int_{n^{-\Gamma}}^{e^{-T}} |\log(s)| d\nu_{\frac1{\sqrt{n}}Y_n-z}(s) \ll e^{-T} + d^{-1/48}\Gamma\log n \ll e^{-T} + o(d^{-1/48}\log^2n).
\end{equation*}
Assuming $C_0\ge 96$ (in addition to our previous assumptions on $C_0$), we have that except with probability $O(e^{-n})$,
\begin{equation}
\int_{n^{-\Gamma}}^{e^{-T}} |\log(s)| d\nu_{\frac1{\sqrt{n}}Y_n-z}(s) \ll e^{-T} + o(1).
\end{equation}
We can now take $T$ sufficiently large depending on $\eps$ to make the right hand side smaller than $\eps/4$ for all $n$ sufficiently large, giving
\[
\sup_{n\ge n_0(\eps)}\pro{ \int_{n^{-\Gamma}}^{e^{-T}} |\log(s)| d\nu_{\frac1{\sqrt{n}}Y_n-z}(s)>\eps/4} \ll e^{-n}.
\]
The right hand side is smaller than $\eps/2$ for all sufficiently large $n$.
We take $T$ larger if necessary to make the integral zero for all other values of $n$, which yields \eqref{ui:goal4}. 
Finally, taking $T\ge T'(\eps,z)$ completes proof of Theorem \ref{thm:main} on Propositions \ref{prop:weak} and \ref{prop:wegner}.

The remainder of the paper is organized as follows.
In Section \ref{sec:hermitization} we recall the approach to studying empirical singular value distributions via Hermitization and the resolvent method.
In Section \ref{sec:BG} we introduce two iid models to be compared with $Y_n$ -- a Bernoulli matrix $X_n$ and a Gaussian matrix $G_n$ -- and state lemmas giving quantitative comparisons between the singular value distributions of $Y_n$ and $X_n$, and between $X_n$ and $G_n$.
In the remainder of Section \ref{sec:comparison} we use these comparison lemmas to prove Propositions \ref{prop:weak} and \ref{prop:wegner}.
The comparison lemmas are proved in Section \ref{sec:compare.proofs}. 
In the appendix we prove a bound on the local density of small singular values for perturbed Gaussian matrices.

\section{The comparison strategy} 	\label{sec:comparison}

\subsection{Hermitization and the Stieljes transform}			\label{sec:hermitization}

To prove Propositions \ref{prop:weak} and \ref{prop:wegner} we will use a popular linearization trick and Stieltjes transform techniques, which we now briefly outline; see \cite{BoCh:survey} for additional background and motivation.

For an $n\times n$ matrix $M$ and $z\in \C$ we define the $2n\times 2n$ matrix
\begin{equation}	\label{def:her}
\Her_z(M):= \begin{pmatrix} 0 & \frac1{\sqrt{n}}M - z\I_n \\ \frac{1}{\sqrt{n}}M^* - \overline{z}\I_n &0 \end{pmatrix}.
\end{equation}
which we refer to as the (shifted and rescaled) \emph{Hermitization} of $M$. 
It is routine to verify that the $2n$ eigenvalues of $\Her_z(M)$ (counted with multiplicity) are the signed singular values $\pm s_1(\frac1{\sqrt{n}}M-z), \cdots, \pm s_n(\frac1{\sqrt{n}}M-z)$. 
Consequently, the ESD $\mu_{\Her_z(M)}$ is the symmetrization of the empirical singular value distribution $\nu_{\frac{1}{\sqrt{n}}M-z}$, i.e.\ 
\begin{equation}	\label{symmetrize}
\mu_{\Her_z(M)}= \frac12\Big(\nu_{\frac{1}{\sqrt{n}}M-z}(\cdot) + \nu_{\frac{1}{\sqrt{n}}M-z}(-\,\cdot)\Big).
\end{equation}
Thus, to prove weak convergence and anti-concentration for the measures $\nu_{\frac1{\sqrt{n}}Y_n-z}$, it will suffice to study the ESDs $\mu_{\Her_z(M)}$.
The key advantage of considering the ESDs of $\Her_z(M)$ over $(\frac1{\sqrt{n}}M - z\I_n)^*(\frac1{\sqrt{n}}M - z\I_n)$ is that $\Her_z(M)$ is a linear function of $M$.

We denote the resolvent of $\Her_z(M)$ at $w\in \C_+$ by
\begin{equation}	\label{def:res}
\Res_{z,w}(M) := (\Her_z(M)-w\I_{2n})^{-1}
\end{equation}
and its normalized trace 
\begin{equation}	\label{def:g}
g_{z,w}(M):= \frac1{2n} \tr \Res_{z,w}(M).
\end{equation}
We will frequently apply the bound
\begin{equation}	\label{bd:resop}
\|\Res_{z,w}(M)\|\le \frac1{\Im w}
\end{equation}
which is immediate from \eqref{def:res} as $\Im w$ is a lower bound on the distance from $w$ to the spectrum of $\Her_z(M)$ in the complex plane.

Denote the Stieltjes transform of a probability measure $\mu$ on $\R$ by
\begin{equation}
m_{\mu}: \C_+\to \C_+, \qquad m_\mu(w) = \int_\R \frac{d\mu(x)}{x-w}.
\end{equation}
If $H$ is an $n\times n$ Hermitian matrix with real eigenvalues $\lambda_1, \cdots, \lambda_n$, the Stietjes transform of its empirical spectral distribution is given by
\begin{equation}	\label{spectralcalc}
m_{\mu_{H}}(w) = \frac1n\sum_{i=1}^n \frac{1}{\lambda_i-w} = \frac{1}{n}\tr (H-w)^{-1}.
\end{equation}
In particular, 
\begin{equation}
g_{z,w}(M) = m_{\mu_{\Her_z(M)}}(w).
\end{equation}

The Stieltjes transforms $w\mapsto g_{z,w}(Y_n)$ will be a key tool in the proofs of Propositions \ref{prop:weak} and \ref{prop:wegner}.
Indeed, it is a standard fact in random matrix theory that weak convergence of the ESDs $\mu_{\Her_z(Y_n)}$ follows from pointwise convergence of their Stieltjes transforms. 
Furthermore, a key advantage of considering Stieltjes transforms over moments is that the former provide good quantitative control of ESDs of Hermitian matrices at short scales, which will be key for obtaining Proposition \ref{prop:wegner}. 
For more discussion of the Stieltjes transform method in Hermitian random matrix theory we refer to the books \cite{AGZ:book, Tao:book} and the review article \cite{BeKn:locallaw}.

\subsection{Comparison with Gaussian and Bernoulli ensembles}	\label{sec:BG}

For each $n\ge 1$ we let $G_n$ denote an $n\times n$ matrix with iid standard real Gaussian entries.
We also let $B_n=(b_{ij}^{(n)})$ be an $n\times n$ matrix with independent Bernoulli$(p)$-distributed 0--1 entries, and set
\begin{equation}	\label{def:X}
X_n = \frac1{\sqrt{p(1-p)}}(B_n-p\1_n\1_n^\tran) 
\end{equation}
where we recall $p:=d/n\le 1/2$.
We denote the entries of $X_n$ by $\xi_{ij}^{(n)}$. 
As with $A_n,Y_n$ we will often suppress the dependence on $n$ and write $G,X,\xi_{ij}$. 
Note that the variables $\xi_{ij}$ are iid centered variables with unit variance. 
We further note that for any $q>2$ we have the moment bound
\begin{equation}	\label{Lq}
\e |\xi_{ij}|^q \le  \left({\frac{1-p}{p}}\right)^{(q-2)/2} \e |\xi_{ij}|^2\le \left({\frac{1}{p}}\right)^{(q-2)/2} = \left(\frac{n}{d}\right)^{(q-2)/2}. 
\end{equation}

The following two lemmas give quantitative comparisons between the ESDs of $\Her_z(Y)$ and $\Her_z(X)$, and of $\Her_z(X)$ and $\Her_z(G)$. The proofs are deferred to Sections \ref{sec:YtoX} and \ref{sec:XtoG}.

\begin{lemma}[Comparison with iid Bernoulli]	\label{lem:YtoX}
Let $z\in \C$ and let $f: \R\to \R$ be an $L$-Lipschitz function supported on a compact interval $I\subset \R$. 
Let $X,Y$ be $n\times n$ matrices as in \eqref{def:X}, \eqref{def:Y}, respectively. Assume $\log^4n\le d\le n/2$.
For any $\eps>0$,
\begin{align}
\pr\bigg( \bigg|  \int_\R f \,\dd \mu_{\Her_z(Y)} - \,&\,\e \int_\R f\,\dd \mu_{\Her_z(X)} \bigg| \ge \eps\bigg) 	\le \frac{|I|}{\eps} \expo{ O(d^{2/3}n\log n)- \frac{cnd\eps^4}{L^2|I|^2}}
\end{align}
for some constant $c>0$.
\end{lemma}

\begin{lemma}[From iid Bernoulli to iid Gaussian]		\label{lem:XtoG}
Let $z\in \C$ and $w\in \C_+$. We have
\begin{equation}
\big| \e g_{z,w}(X) - \e g_{z,w}(G) \big| \ll  \frac1{d^{1/2}(\Im w)^4}\left( 1+ \frac1{(n\Im w)^2}\right)
\end{equation}
where the implied constant is absolute.
\end{lemma}

In the remainder of this section we use the above lemmas to establish Propositions \ref{prop:weak} and \ref{prop:wegner}.

\subsection{Proof of Proposition \ref{prop:weak}} 	\label{sec:weak}

Our starting point is the following, which gives the desired limiting behavior for the Gaussian matrices $G_n$ in place of $Y_n$. We will then use Lemmas \ref{lem:YtoX} and \ref{lem:XtoG} to transfer this limiting property to $Y_n$.

\begin{lemma}[Convergence of singular value distributions, Gaussian case]	\label{lem:dosi}
For each $z\in \C$ there exists a probability measure $\nu_z$ on $\R_+$ such that $\nu_{\frac1{\sqrt{n}}G_n-z}\to \nu_z$ in probability and in expectation. Moreover, the family $\{\nu_z\}_{z\in \C}$ satisfies the relation \eqref{Umu} with $\mu=\muc$.
\end{lemma}

\begin{proof}
The existence of the measures $\nu_z$ follows as a special case of a result of Dozier and Silverstein \cite{DoSi07a} (which allows more general entry distributions and more general shifts than $-z\I_n$). See \cite[Lemma 3]{PaZh} for the verification that \eqref{Umu} holds with $\mu=\muc$.
\end{proof}

\begin{proof}[Proof of Proposition \ref{prop:weak}]
By Lemma \ref{lem:dosi} it suffices to show $\nu_{\frac1{\sqrt{n}}Y_n-z}-\e\nu_{\frac1{\sqrt{n}}G_n-z}$ converges in probability to the zero measure, i.e.\ to show that for any $f\in C_c(\R)$ and any $\eps>0$,
\begin{equation}	\label{weak.goal1}
\pro{\left| \int_\R f\, \dd \nu_{\frac1{\sqrt{n}}Y_n-z} - \e \int_\R f \, \dd \nu_{\frac1{\sqrt{n}}G_n - z} \right|\ge \eps} =o(1)
\end{equation}
where here and in the remainder of the proof we implicitly allow quantities $o(1)$ to tend to zero at a rate depending on $f$ and $\eps$.

Fix $f\in C_c(\R)$ and $\eps>0$.
From Lemma \ref{lem:XtoG} and our assumption that $d$ grows to infinity with $n$ it follows that
\begin{equation*}
\e \int_\R f \, \dd \nu_{\frac1{\sqrt{n}}G_n - z} = \e \int_\R f \, \dd \nu_{\frac1{\sqrt{n}}X_n - z} + o(1)
\end{equation*}
(see for instance \cite[Theorem 2.4.4]{AGZ:book}). 
Thus, it suffices to show
\begin{equation}	\label{weak.goal2}
\pro{\left| \int_\R f\, \dd \nu_{\frac1{\sqrt{n}}Y_n-z} - \e \int_\R f \, \dd \nu_{\frac1{\sqrt{n}}X_n - z} \right|\ge \eps} =o(1).
\end{equation}
There exists a compactly supported Lipschitz function $f_\eps$ with support and Lipschitz constant depending only on $f$ and $\eps$ such that $\|f-f_\eps\|_\infty\le \eps/4$. 
Now it suffices to show
\begin{equation}	\label{weak.goal2}
\pro{\left| \int_\R f_\eps\, \dd \nu_{\frac1{\sqrt{n}}Y_n-z} - \e \int_\R f_\eps \, \dd \nu_{\frac1{\sqrt{n}}X_n - z} \right|\ge \eps/2} =o(1).
\end{equation}
But the above is immediate from Lemma \ref{lem:YtoX} and our assumption $d\ge \log^4n$.
\end{proof}

\subsection{Proof of Proposition \ref{prop:wegner}}	\label{sec:wegner}

As in the proof of Proposition \ref{prop:weak}, our starting point is a result for Gaussian matrices. The proof is deferred to Appendix \ref{app:wegner}.

\begin{lemma} 	\label{lem:wegner.gaussian}
There are constants $C,c>0$ such that the following holds.
Let $1\le k\le n$ and let $M\in \mM_n(\C)$ be a deterministic matrix.
Except with probability $O(n^2e^{-ck})$, for all $k\le j\le n-1$ we have $s_{n-j}(\frac1{\sqrt{n}}G +M) \ge cj/n$. 
\end{lemma}

\begin{remark}
We will apply the lemma with $k=\sqrt{n}$, but we note that it gives a nontrivial result down to much smaller scales:
namely, with high probability, all but the smallest $k$ singular values are within a constant of their expected values, as long as $k$ grows faster than $\log n$. 
\end{remark}

\begin{proof}[Proof of Proposition \ref{prop:wegner}]
It will be more convenient to work with $\mu_{\Her_z(Y)}$, the symmetrized probability measure for $\nu_{\frac1{\sqrt{n}}Y-z}$ (see Section \ref{sec:hermitization}). Thus, it suffices to show that with probability $1-O(e^{-n})$, for all $\eta\in (0,1]$, 
\begin{equation}	\label{wegner.goal1}
\mu_{\Her_z(Y)}([-\eta,\eta]) \ll \eta + d^{-1/48}.
\end{equation}

First we show there is a constant $C$ such that for any fixed $\eta\in (d^{-1/48},1]$,
\begin{equation}	\label{wegner.step1}
\pro{ \mu_{\Her_z(Y)}([-\eta,\eta]) \le C(\eta+d^{-1/48})} = 1-O(\eta e^{-n}).
\end{equation}
Fix $\eta\in (d^{-1/48},1]$ and denote $I=[-\eta,\eta]$. 
Let $f_\eta$ be the piece-wise linear $\eta^{-1}$-Lipschitz function which is equal to $1$ on $[-\eta,\eta]$ and zero outside of $(-2\eta,2\eta)$. We have $1_I\le f_\eta$ pointwise, and by Lemma \ref{lem:YtoX} (taking $\eps=C'd^{-1/12}\log^{1/4}n$ for a sufficiently large constant $C'>0$),
\[
\mu_{\Her_z(Y)}(I) \le \int f_\eta \, \dd\mu_{\Her_z(Y)} \le \e \int f_\eta \,\dd\mu_{\Her_z(X)} + O(d^{-1/12}\log^{1/4}n)
\]
except with probability $O(\eta\exp( -cd^{2/3}n\log n))= O(\eta e^{-n})$ (adjusting the constant $c$ to absorb the prefactor $d^{1/12}/\log^{1/4}n$). The error term on the right hand side above is $O(d^{-1/48})$ by our assumption on $d$. 
Now to obtain \eqref{wegner.step1} it suffices to show 
\begin{equation}	\label{wegner:goal2}
\e \int f_\eta \,\dd \mu_{\Her_z(X)} \ll \eta + d^{-1/48}.
\end{equation}
From the pointwise bound
\[
\frac1{5\eta} 1_{[-2\eta,2\eta]}(x) \le \frac{\eta}{x^2+ \eta^2}  = \Im \frac{1}{x-\ii\eta}  , \quad x\in \R
\]
together with Lemma \ref{lem:XtoG} we have
\begin{align}
\e \int f_\eta \,\dd \mu_{\Her_z(X)} 
&\le \e \int 1_{[-2\eta,2\eta]} \,\dd \mu_{\Her_z(X)} \notag \\
&\le 5\eta \e \Im g_{z, \ii\eta}(X) \notag \\
&\le 5\eta \e \Im g_{z,\ii\eta}(G) + O\left( \frac{1}{\eta^4\sqrt{d} }\left( 1+  \frac1{(n\eta)^{2}}\right)\right) 	\\
&\le 5\eta \e \Im g_{z,\ii\eta}(G) + O\left( \frac{1}{\eta^4\sqrt{d} }\right), \label{weg:1}
\end{align}
where in the last line we used the assumption $\eta\ge d^{-1/48}\ge n^{-1}$ to bound $(n\eta)^{-1}=O(1)$.
Now we use Lemma \ref{lem:wegner.gaussian} to bound the first term.
First, by Fubini's theorem,
\begin{align}
\Im g_{z,\ii\eta}(G) &= \int_{\R} \frac{\eta}{x^2+\eta^2} \,\dd \mu_{\Her_z(G)} (x)	\notag\\
& =\int_0^{1/\eta} \int_{\R} \mathbf{1}_{\left\{y\le \frac{\eta}{x^2+\eta^2}\right\}}(x,y) \,\dd \mu_{\Her_z(G)} (x)\,\dd y	\notag\\
&= \int_0^{1/\eta} \mu_{\Her_z(G)} \left( \left[ -\sqrt{\eta\left(y^{-1}-\eta\right)},\sqrt{\eta\left(y^{-1}-\eta\right)}\right]\right) \,\dd y	\notag\\
&= \frac1n \int_0^{1/\eta} \left| \left\{ j\in [n]: s_{n-j}\left( \frac1{\sqrt{n}}G-z\right) \le \sqrt{\eta\left( y^{-1}-\eta\right)} \right\}\right| \,\dd y.	\label{img.int}
\end{align}
Let constants $C,c$ be as in Lemma \ref{lem:wegner.gaussian}, let $k\ge 1$ to be chosen later, and let $\good_k$ denote the event that $s_{n-j}(\frac1{\sqrt{n}}G-z)\ge cj/n$ for all $k\le j\le n-1$.
On $\good_k$ the integrand in \eqref{img.int} is bounded by
\[
k + O(n\sqrt{\eta(y^{-1} -\eta)}).
\]
Inserting this estimate in \eqref{img.int}, we have 
\begin{align*}
\Im g_{z,\ii \eta}(G) \un_{\good_k} 
&\ll \frac{k}{\eta n} + \int_0^{1/\eta} \sqrt{\eta(y^{-1} - \eta)} \,\dd y \\
&= \frac{k}{\eta n} + \int_0^{1/\eta}\int_0^\infty \1_{\left\{ y\le \frac{\eta}{x^2+ \eta^2} \right\}}(x,y)d xdy \\
&= \frac{k}{\eta n} + \int_0^\infty \frac{\eta}{x^2+ \eta^2} d x\\
&= O\left( 1+ \frac{k}{\eta n}\right).
\end{align*}
From the deterministic bound $\Im g_{z,\ii \eta}(G) \le 1/\eta$ and Lemma \ref{lem:wegner.gaussian} (taking $M=-z\I_n$) we conclude
\begin{equation}
\e \Im g_{z,\ii \eta}(G) \ll  1 + \frac{k}{\eta n} + \frac1\eta\pr(\good_k^c) \le  1 + \frac{k}{\eta n}+ \frac1\eta n^2e^{-ck} .
\end{equation}
Substituting into \eqref{weg:1} yields
\[
\e \int f_\eta\, \dd \mu_{\Her_z(X)}  \ll \eta + \frac{k}{n} + n^2e^{-ck} + \frac{1}{\eta^4\sqrt{d} }.
\]
Taking $k=\sqrt{n}$, by our assumption on $\eta$ the last three terms are of lower order, which yields \eqref{wegner:goal2} and hence \eqref{wegner.step1}. (By optimizing $\eta$ to balance the first and last terms above, one sees that we actually showed the stronger bound $\e \int f_\eta \,\dd \mu_{\Her_z(X)} \ll \eta + d^{-1/10}$.)

Now consider the events
\begin{equation}
\bad_m(t)  =  \left\{ \mu_{\Her_z(Y)}([-2^{-m},2^{-m}])> t \right\}, \quad m\in \N_{\ge0},\; t\in \R_+
\end{equation}
and put
\[
\mB=\bigvee_{m\ge 0} \bad_m\Big( 2C\big(2^{-m} + d^{-1/48}\big)\Big).
\]
Denoting $m^*=\lf \frac1{48} \log_2d\rf$, from the union bound,
\begin{align*}
\pro{ \bad} 
&\le \sum_{m=0}^{m^*} \pro{ \mB_m\Big( 2C\big(2^{-m} + d^{-1/48}\big)\Big) } + \pr \bigg( \bigvee_{m> m^*} \bad_m\Big( 2C\big(2^{-m} + d^{-1/48}\big)\Big) \bigg) \\
&\le \sum_{m=0}^{m^*} \pro{ \mB_m\Big( 2C\big(2^{-m} + d^{-1/48}\big)\Big) } + \pro{ \mB_{m^*+1} (2Cd^{-1/48})}\\
&\le \sum_{m=0}^{m^*+1} \pro{ \mB_m\Big( C\big(2^{-m} + d^{-1/48}\big)\Big) },
\end{align*}
where we in the second and third lines we used that the events $\bad_m(t)$ are monotone in the parameter $t$.
Applying \eqref{wegner.step1}, 
\[
\pro{ \bad}\ll e^{-n} \sum_{m=0}^\infty 2^{-m} \ll e^{-n}.
\]
Fix $\eta\in (0,1]$ arbitrarily and suppose $\mB$ does not hold. Let $m$ be the integer such that $2^{-(m+1)}<\eta\le 2^{-m}$.
We have
\[
\mu_{\Her_z(Y)}([-\eta,\eta]) \le \mu_{\Her_z(Y)}([-2^{-m},2^{-m}]) \le C(2^{-m} + d^{-1/48}) \le 2C(\eta + d^{-1/48}). 
\]
The result follows after adjusting the constant $C$.
\end{proof}

\section{Proofs of comparison lemmas}	\label{sec:compare.proofs}

\subsection{Proof of Lemma \ref{lem:YtoX}}		\label{sec:YtoX}

Recall that the matrix $B=B_n$ has iid Bernoulli($p$) entries, with $p=d/n$.
Here we follow a strategy that was used by Tran, Vu and Wang in \cite{TVW} to prove a local semicircle law for adjacency matrices of random (undirected) regular graphs of growing degree. 
The idea is to use sharp concentration estimates for linear eigenvalue statistics of Hermitian random matrices together with a lower bound on the the probability that the iid Bernoulli matrix $B_n$ lies in $\mA_{n,d}$. 
For the former we have the next lemma, which is easily obtained from the arguments of Guionnet and Zeitouni in \cite{GuZe}:

\begin{lemma}[Concentration of linear statistics]	\label{lem:guze}
Let $H=(h_{ij})_{i,j=1}^n$ be a Hermitian random matrix with entries on and above the diagonal jointly independent and uniformly bounded by $K/\sqrt{n}$ for some $K<\infty$.
Let $f:\R\to\R$ be an $L$-Lipschitz function supported on a compact interval $|I|\subset \R$, and let $H_0$ be an arbitrary $n\times n$ deterministic Hermitian matrix. 
For any $\eps>0$,
\begin{equation}
\pro{ \left| \int_\R f \dd\mu_{H+H_0}  - \e \int_\R f\dd\mu_{H+H_0}\right| \ge \eps} \le (C|I|/\eps) \expo{ -\frac{cn^2\eps^4}{K^2 L^2 |I|^2}}
\end{equation}
for some constants $C,c>0$.
\end{lemma}

\begin{proof}
The case $H_0=0$ follows directly from \cite[Theorem 1.3(a)]{GuZe}. The general case follows from a slight modification of the proof in \cite{GuZe} -- see \cite[Lemma 3.2]{Cook:circpm}.
\end{proof}

The following is established in Appendix \ref{app:band}, following an argument of Shamir and Upfal for undirected $d$-regular graphs \cite{ShUp}.

\begin{lemma}	\label{lem:band}
Assume $\log^4n\le d\le n/2$.  
Then
\[
\pro{ B\in \mA_{n,d}} \ge \expo{  -O(d^{2/3}n\log n)}.
\]
\end{lemma}

\begin{remark}
The more accurate asymptotic
\begin{equation}
\pro{ B\in \mA_{n,d}} = (1+o(1))\sqrt{2\pi d(n-d)} \expo{ -n\log\left(\frac{2\pi d(n-d)}{n}\right) }
\end{equation}
was established for the range $d=o(\sqrt{n})$ by McKay and Wang \cite{McWa} and $\min(d,n-d)\gg n/\log n$ by Canfield and McKay \cite{CaMc}.
For the proof of Theorem \ref{thm:main} it is only important that we have a bound of the form $\pr(B\in \mA_{n,d})\ge \expo{ -o(nd)}$.
\end{remark}

In Appendix \ref{app:band} we actually prove a slightly stronger version of Lemma \ref{lem:band} allowing $d$ to grow as slowly as $\log^{1+\eps}n$. 

\begin{proof}[Proof of Lemma \ref{lem:YtoX}]
Recall that $A$ denotes a uniform random element of $\mA_{n,d}$, and $B$ denotes an $n\times n$ matrix with independent Bernoulli($p$) entries, where $p=d/n$.
For fixed $B_0\in \{0,1\}^{n\times n}$ we denote
\begin{equation}
M(B_0) := \frac1{\sqrt{p(1-p)}} (B_0-p \1_n\1_n^\tran).
\end{equation}
In this notation we have $X=M(B)$ and $Y=M(A)$ (see \eqref{def:X}, \eqref{def:Y}). 
For $f\in C_c(\R)$, $z\in C$ and $\eps>0$ we denote the corresponding ``bad set" of 0--1 matrices
\begin{equation}
\mB(f,z,\eps) := \left\{ B_0\in \{0,1\}^{n\times n} : \ \left| \int_\R f \dd\mu_{\Her_z(M(B_0))} - \e \int_\R f \dd\mu_{\Her_z(M(B))} \right| \ge \eps\right\}.
\end{equation}
Our aim is to show that for any $L$-Lipschitz function $f:\R\to \R$ supported on a compact interval $I$ and any $\eps>0$, 
\begin{equation}	\label{YtoX:goal1}
\pro{ A\in \mB(f,z,\eps) } \ll \frac{|I|}{\eps} \expo{ O(nd^{1/2}) - \frac{cnd\eps^4}{L^2|I|^2}}.
\end{equation}
Fix such $f$ and $\eps$. 
We can apply Lemma \ref{lem:guze} (with $n$ replaced by $2n$), taking $H_z(X)$ for $H$, $\begin{pmatrix} 0 & z\\ \overline{z} & 0\end{pmatrix}\otimes \I_n$ for $H_0$, and $K=1/\sqrt{p}$ to obtain
\begin{equation}	\label{apply:guze}
\pro{ B\in \mB(f,z,\eps)} \ll \frac{|I|}{\eps} \expo{ - \frac{cn^2p\eps^4}{L^2|I|^2}} = \frac{|I|}{\eps} \expo{ - \frac{cnd\eps^4}{L^2|I|^2}}.
\end{equation}
Now notice that conditional on the event $\{B\in \mA_{n,d}\}$, $B$ is uniformly distributed over $\mA_{n,d}$.
Thus,
\begin{align*}
\pro{ A\in \mB(f,z,\eps)} 
&= \pro{ B\in \mB(f,z,\eps) \mid B\in \mA_{n,d}} \\
&= \frac{\pro{ B\in \mB(f,z,\eps)\cap   \mA_{n,d}} }{ \pro{ B\in \mA_{n,d}}}\\
&\le \frac{\pro{ B\in \mB(f,z,\eps) }}{ \pro{ B\in \mA_{n,d}}}, 
\end{align*}
and \eqref{YtoX:goal1} follows from \eqref{apply:guze} and Lemma \ref{lem:band}.
\end{proof}

\subsection{Proof of Lemma \ref{lem:XtoG}}	\label{sec:XtoG}

Here we make use of the Lindeberg replacement strategy, which was introduced to random matrix theory by Chatterjee \cite{Chatterjee:invariance1,Chatterjee:invariance2}, who used it to prove the semicircle law for random symmetric matrices with exchangeable entries above the diagonal.
It has since become a widely used tool in universality theory for random matrices, most notably with its use by Tao, Vu and others to establish universality of local eigenvalue statistics for various models; see e.g.\ \cite{TaVu:4MT_survey2,ErYa12:survey, TaVu:4MT_iid} and references therein.

In particular we will apply the following \emph{invariance principle}:

\begin{theorem}[cf.\ {\cite[Theorem 1.1 and Corollary 1.2]{Chatterjee:invariance1}}]	\label{thm:invar}
Let $X$ and $W$ be independent random vectors in $\R^N$ with independent components having finite third moment and satisfying $\e X_i = \e W_i$ and $\e X_i^2 = \e W_i^2$ for $1\le i\le N$.
Denote
\begin{equation}
\gamma_3= \max_{i\in [n]} \max\{ \e |X_i|^3, \e |W_i|^3\}.
\end{equation}
Let $f\in C^3(\R^N\to \R)$, and denote
\begin{equation}
\lambda_3(f) = \sup_{x\in \R^N} \max_{r \in \{1,2,3\}}  \max_{i\in [n]} |\partial_i^{r} f(x)|^{3/r}.
\end{equation}
Write $U=f(X)$, $V=f(W)$, and let $h\in C^3(\R\to \R)$. 
Then
\begin{equation}	\label{bd:invar}
 |\e h(U) - \e h(V)| \ll_h \gamma_3\lambda_3(f) N.
\end{equation}
\end{theorem}

With Theorem \ref{thm:invar} in hand, the proof of Lemma \ref{lem:XtoG} boils down to estimating the partial derivatives of the resolvent $\Res_{z,w}(M)$ from \eqref{def:res}, viewed as a function of $M$. 
The proof is similar an argument that was sketched in the appendix of \cite{TaVu:esd}, and subsequently applied in the sparse setting by Wood \cite{Wood:sparse}, and to matrices with exchangeable entries by Adamczak, Chafa\"i and Wolff in \cite{ACW:exchangeable} (who used a more general invariance principle from \cite{Chatterjee:invariance2} for exchangeable sequences). 

Here we will follow similar lines to the above-mentioned works; the only difference is that we will need to quantify errors in order to obtain the bound in Proposition \ref{prop:wegner}. 
The aforementioned works all obtained estimates like Proposition \ref{prop:wegner} by a different geometric argument, also introduced in \cite{TaVu:esd}. 
Adapting that argument to the setting of random regular digraphs appears to be of comparable difficulty to the proof of Theorem \ref{thm:ssv} due to the dependencies among entries. Instead, we have opted to make our comparisons in Lemmas \ref{lem:YtoX} and \ref{lem:XtoG} quantitative, and apply the geometric argument from \cite{TaVu:esd} in the simpler Gaussian setting (see the proof of Lemma \ref{lem:wegner.gaussian}).

\begin{proof}[Proof of Lemma \ref{lem:XtoG}]
Fix $z\in \C$ and $w\in \C_+$. 
For a differentiable matrix-valued function $H:\R\to \R^{2n\times 2n}$ and $R(t) = (H(t)-w\I_{2n})^{-1}$,
the following well-known identity is easily verified:
\begin{equation}	\label{id:res}
\frac{\dd}{\dd t} R = - R \frac{\dd H}{\dd t} R.
\end{equation}
Now let $\Her(M)=\Her_z(M)$ and $\Res(M)=\Res_{z,w}(M)$ be as in \eqref{def:her}, \eqref{def:res}.
By iterating \eqref{id:res} we obtain the following formulas for the partial derivatives of $\Res$ with respect to entries $\alpha_1,\alpha_2,\alpha_3\in [n]^2$ of $M$:
\begin{align}
\partial_{\alpha_1}\Res &= \Res \big(\partial_{\alpha_1}\Her\big)\Res,	\label{deriv1}\\
\partial_{\alpha_2}\partial_{\alpha_1}\Res &= \sum_{\sigma\in \sym(2)}\Res \big(\partial_{\alpha_{\sigma(1)}}\Her\big)\Res \big(\partial_{\alpha_{\sigma(2)}}\Her\big)\Res,	\label{deriv2} \\
\partial_{\alpha_3}\partial_{\alpha_2}\partial_{\alpha_1}\Res &= \sum_{\sigma\in \sym(3)}
\Res \big(\partial_{\alpha_{\sigma(1)}}\Her\big)
\Res \big(\partial_{\alpha_{\sigma(2)}}\Her\big)\Res \big(\partial_{\alpha_{\sigma(3)}}\Her\big)\Res,	\label{deriv3}
\end{align}
where the sums in \eqref{deriv2}, \eqref{deriv3} run over the symmetric group on two and three labels, respectively.
(Here we have used the fact that $\partial_{\alpha_1}\partial_{\alpha_2}\Her=0$ for any $\alpha_1,\alpha_2$ as $\Her$ is a linear function of $M$.)
Now if $\alpha=(i,j)$ we have
\begin{equation}	\label{derivH}
\partial_\alpha \Her = \frac1{\sqrt{n}}\big( \mathbf{E}_{i,j+n} + \mathbf{E}_{j+n,i}\big) ,
\end{equation}
where $\mathbf{E}_{i,j}$ is the $2n\times 2n$ matrix with entries $(\mathbf{E}_{i,j})_{k,l}= \delta_{i=k}\delta_{j=l}$. 
In particular,
\begin{equation}	\label{dHHS}
\max_{\alpha\in [n]^2}\|\partial_\alpha \Her\|_\HS =O(n^{-1/2}).
\end{equation}
Combining \eqref{deriv1}, \eqref{derivH} and the definition of $g(M)=g_{z,w}(M)$ we obtain for any $\alpha=(i,j)\in [n]^2$,
\begin{align*}
\partial_{\alpha}g 
&= -\frac1{2n^{3/2}} \tr \Res \big( \mathbf{E}_{i,j+n} + \mathbf{E}_{j+n,i}\big) \Res\\
&= -\frac1{2n^{3/2}} \tr \big( \mathbf{E}_{i,j+n} + \mathbf{E}_{j+n,i}\big) \Res^2.
\end{align*}
Since each of $\mathbf{E}_{i,j+n}$, $\mathbf{E}_{j+n,i}$ have one nonzero entry with value 1, we obtain
we have
$
|\partial_\alpha g| \le \|R\|^2/n^{3/2},
$
and from \eqref{bd:resop} we conclude
\begin{equation}
\max_{\alpha\in [n]^2}\sup_{M\in \R^{n\times n}}|\partial_{\alpha}g(M) |\le \frac{1}{n^{3/2}(\Im w)^2}.
\end{equation}

We turn to the second order partial derivatives of $g$.
By repeated application of the inequalities
\begin{equation}	\label{opHS}
|\tr(AB)|\le \|A\|_\HS \|B\|_\HS, \qquad \|AB\|_\HS \le \|A\|\|B\|_\HS
\end{equation}
we can bound
\begin{align*}
|\partial_{\alpha_2}\partial_{\alpha_1} g| 
&\le \frac1{2n}\sum_{\sigma\in \sym(2)} \left| \tr \Res \big(\partial_{\alpha_{\sigma(1)}}\Her\big)\Res \big(\partial_{\alpha_{\sigma(2)}}\Her\big)\Res \right| \\
&= \frac1{2n}\sum_{\sigma\in \sym(2)} \left| \tr  \big(\partial_{\alpha_{\sigma(1)}}\Her\big)\Res \big(\partial_{\alpha_{\sigma(2)}}\Her\big)\Res^2 \right| \\
&\le \frac1{2n}\sum_{\sigma\in \sym(2)}  \|\partial_{\alpha_{\sigma(1)}}\Her\|_\HS \| \Res \big(\partial_{\alpha_{\sigma(2)}}\Her\big)\Res^2\|_\HS\\
&\le  \frac1{2n}\sum_{\sigma\in \sym(2)}  \|\partial_{\alpha_{\sigma(1)}}\Her\|_\HS \| \partial_{\alpha_{\sigma(2)}}\Her\|_\HS \|\Res\|^3\\
&\ll \frac1{n^{2}} \|\Res\|^3,
\end{align*}
where in the final line we used \eqref{dHHS}. Applying \eqref{bd:resop} we conclude
\begin{equation}
\max_{\alpha_1,\alpha_2\in [n]^2}\sup_{M\in \R^{n\times n}} \left| \partial_{\alpha_1}\partial_{\alpha_2} g(M) \right| 
\ll \frac{1}{n^2 (\Im w)^3}.
\end{equation}

By similar steps applied to the trace of the identity \eqref{deriv3} (cyclically permuting the trace, repeated application of \eqref{opHS}, and the bounds \eqref{dHHS} and \eqref{bd:resop})
one obtains
\begin{equation}
\max_{\alpha_1,\alpha_2,\alpha_3\in [n]^2} \sup_{M\in \R^{n\times n}} \left| \partial_{\alpha_1}\partial_{\alpha_2}\partial_{\alpha_3} g(M) \right| 
\ll \frac{1}{n^{5/2} (\Im w)^4}.
\end{equation}

Since $\partial_\alpha$ commutes with $\Re(\cdot)$ the above estimates give (with notation as in Theorem \ref{thm:invar})
\begin{align*}
\lambda_3(\Re g_{z,w}) 
&\ll \frac{1}{n^{9/2}(\Im w)^{6}}+ \frac{1}{n^{3}(\Im w)^{9/2}}+ \frac1{n^{5/2}(\Im w)^4} \\
&= O\left(\frac{1}{n^{5/2}(\Im w)^4}\right)\left(\frac{1}{(n\Im w)^{2}}+ \frac{1}{(n\Im w)^{1/2}} +1\right).
\end{align*}
Now we apply Theorem \ref{thm:invar} with $\Re g_{z,w}$ in place of $f$, $n^2$ in place of $N$ (identifying $\R^{n\times n}$ with $R^{n^2}$), and matrices $X,G$ in place of the vectors $X,W$.
From \eqref{Lq} we have
\[
\gamma_3 \ll (n/d)^{1/2}.
\]
Taking $h$ to simply be the identity mapping, \eqref{bd:invar} gives
\begin{align*}
&\left|\e\Re g_{z,w}(X) - \e \Re g_{z,w}(G)\right| \\
&\qquad\qquad\ll \left(\frac{n}{d}\right)^{1/2} \times \left(\frac{1}{n^{5/2}(\Im w)^4}\right)\left(1+ \frac{1}{(n\Im w)^{1/2}}+ \frac{1}{(n\Im w)^{2}} \right) \times n^2 \\
&\qquad\qquad\ll \frac{1}{d^{1/2}(\Im w)^4}\left(1+\frac{1}{(n\Im w)^{2}} \right).
\end{align*}
One obtains the same bound for the imaginary parts by the same lines. 
\end{proof}

\appendix

\section{Proof of Lemma \ref{lem:wegner.gaussian}}	\label{app:wegner}

In this appendix we establish the estimate of Lemma \ref{lem:wegner.gaussian} for the local density of small singular values of a perturbed real Gaussian matrix.
The argument is a (by now standard) application of an approach introduced in \cite{TaVu:esd}. In fact, a weaker version of the lemma (but still sufficient for our purposes) follows directly from \cite[Lemma 6.7]{TaVu:esd}, which applies to any matrix with iid standardized entries with finite second moment.
However, the argument is simpler in the Gaussian case and gives a stronger bound, so we include the proof below.

Fix $M\in \mM_n(\C)$ and denote $\tG= G+ \sqrt{n}M$. 
By the union bound it suffices to show that for some constants $C,c>0$,
\begin{equation}	\label{app:goal1}
\pro{ s_{n-k}(\tG) \le ck/\sqrt{n}} =O(ne^{-ck})
\end{equation}
for all $1\le k\le n-1$. 
Fix such a $k$; since the desired bound is trivial for small values of $k$ we may assume $k$ is larger than any fixed constant. Let $R_i$ denote the $i$th row of $\tG$.
Put $m=n-\lceil k/2\rceil$, and for each $i\in [m]$ set
\[
V_{i} = \Span( R_j: j\in [m]\setminus \{i\}).
\]
We claim 
\begin{equation}	\label{app:claim1}
s_{n-k}(\tG) \gg \sqrt{\frac{k}{n}} \min_{i\in [m]}\dist(R_i,V_i).
\end{equation}
Indeed, letting $\tG'$ be the $m\times n$ matrix obtained by removing the last $\lf k/2\rf$ rows from $\tG$, by the Cauchy interlacing law we have
\begin{equation}	\label{app:cauchy}
s_{n-k}(\tG)\ge s_{n-k}(\tG').
\end{equation}
On the other hand, from the inverse second moment identity (cf.\ \cite[Lemma A.4]{TaVu:esd}) we have
\[
\sum_{i=1}^m s_i(\tG')^{-2} = \sum_{i=1}^m \dist(R_i, V_i)^{-2}.
\]
Thus,
\begin{align*}
n \max_{i\in [m]} \dist(R_i,V_i)^{-2} &\ge \sum_{i=1}^m s_i(\tG')^{-2}
\ge \sum_{i=n-k}^m s_i(M')^{-2}
\ge \frac{k}{2} s_{n-k}(\tG')^{-2}
\end{align*}
and \eqref{app:claim1} now follows from \eqref{app:cauchy} and rearranging.
By the union bound, to obtain \eqref{app:goal1} it suffices to show that for each fixed $i\in [m]$, 
\begin{equation}
\pro{ \dist(R_i,V_i) \le c\sqrt{k}} \le e^{-ck}
\end{equation}
for a sufficiently small (adjusted) constant $c>0$.

Fix $i\in [m]$.
For the remainder of the proof we condition on the rows $\{R_j: j\in [m]\setminus \{i\}\}$, which fixes the subspace $V_i$. 
Now let $V_i'= \Span(V_i,\e R_i)$ (since $G$ is centered, $\e R_i$ is just the $i$th row of $\sqrt{n}M$). 
Writing $G_i=R_i-\e R_i$ for the $i$th row of $G$, we have $\dist(G_i,V_i')\le \dist(R_i,V_i)$, so it suffices to show
\begin{equation}	\label{app:goal}
\pro{\dist(G_i,V_i')\le c\sqrt{k}} \le e^{-ck}
\end{equation}
for a suitable constant $c>0$.
Also,
\begin{equation}	
\dim V_i'\le \dim V_i+ 1\le m.
\end{equation}
where here and in the following we mean dimension over $\C$.
By rotational symmetry of the distribution of $G_i$ we may assume $V_i'$ is spanned by the last $\dim V_i'$ standard basis vectors in $\C^n$. 
Then we have
\begin{equation}
\dist(G_i,V_i')^2 = \sum_{j=1}^{n-\dim V_i'} G_{ij}^2.
\end{equation}
In particular,
\begin{equation}
\e \dist(G_i,V_i')^2 = n-\dim V_i' \ge n-m \ge k/2.
\end{equation}
\eqref{app:goal} now follows from a standard concentration inequality for Gaussian measure (see for instance \cite{Ledoux:phenom}).

\section{Proof of Lemma \ref{lem:band}}	\label{app:band}

In this appendix we establish the following:

\begin{lemma}	\label{lem:band}
Fix $\alpha\in (0,1/2)$ and $\eps>0$, and let $(\log n)^{\frac1{1-2\alpha}+\eps} \le d\le n/2$. Set $p=d/n$ and let $B\in \{0,1\}^{n\times n}$ have iid Bernoulli$(p)$ entries. With $\mA_{n,d}$ as in \eqref{def:And}, there is a constant $C>0$ such that
\[
\pro{ B\in \mA_{n,d}} \ge \expo{ -Cn\log n\max\big\{ \sqrt{d\log n}, d^{2/3}, d^{1-\alpha}\big\}}
\]
for all $n$ sufficiently large depending on $\eps$. 
\end{lemma}

Lemma \ref{lem:band} follows from the above by setting $\alpha=1/3$ and $\eps=1$.

To prove Lemma \ref{lem:band} we follow an argument of Shamir and Upfal from \cite{ShUp}, who established estimates on the probability that an (undirected) Erd\H{o}s--R\'enyi graph is a $d$-regular graph.
The heart of the proof is to show that with high probability, the matrix $B$ contains a \emph{$d$-regular factor} of slightly smaller density. 
Recall that a $d$-regular factor in $B=(b_{ij})$ is an element $B'=(b_{ij}')\in \mA_{n,d}$ such that $b_{ij}=0\Rightarrow b_{ij}'=0$ (usually this terminology is applied to the associated graphs/digraphs rather than adjacency matrices). 

\begin{lemma}	\label{lem:factor}
Let $n\ge 1$, $p\in (0,1)$, and let $B\in \{0,1\}^n$ have iid Bernoulli($p$) entries. Let
\begin{equation}	\label{lbdelta}
\frac12\ge \delta \ge  C\max\left[\left(\frac{\log n}{pn}\right)^{1/2}, \frac1{(pn)^{1/3}}\right] 
\end{equation}
for a sufficiently large constant $C>0$, 
and put $d=(1-\delta)pn$ (we assume $\delta$ is such that $d$ is an integer, and $n$ and $p$ are such that the range for $\delta$ is nonempty). 
Then $B$ contains a $d$-regular factor except with probability at most $\expo{ -c\delta^2 pn}$ for some constant $c>0$. 
\end{lemma}

In the proof of Lemma \ref{lem:factor} we will write
\begin{equation*}
e(S,T) := \sum_{i\in S, j\in T} B(i,j)
\end{equation*}
for $S,T\subset[n]$ (as in \eqref{def:edge} but suppressing the subscript $B$). For $i\in [n]$, $T\subset[n]$ we write
\begin{equation*}
\deg_T(i):= |\mN_B(i)\cap T|
\end{equation*}
where $\mN_B(i)$ is as in \eqref{def:nbhd}.
We will apply the following consequence of the Ore--Ryser theorem \cite{Ore}:
\begin{prop}
Let $B\in \{0,1\}^{n\times n} $ and let $d\in [n]$. Then $B$ contains a $d$-regular factor $B'$ if and only if 
\begin{equation}	\label{cond:OR}
\forall T\subset[n], \quad X_T:= \sum_{i=1}^d \min(d, \deg_T(i) ) \ge d|T|.	
\end{equation}
\end{prop}

\subsection{Proof of Lemma \ref{lem:factor}}		\label{sec:factor}

We let the constant $C>0$ to be taken sufficiently large over the course of the proof. 
To $T\subset[n]$, associate the set of \emph{heavy} vertices
\begin{equation}
H_T:= \{ i\in [n]: \deg_T(i) \ge d\}.
\end{equation}
We have
\begin{equation}
X_T = d|H_T| + e(H_T^c, T).
\end{equation}
Thus, we see the inequality in \eqref{cond:OR} automatically holds for $T\subset[n]$ such that $|H_T|\ge |T|$, so it suffices to show
\begin{equation}	\label{factor.ets}
\pro{ \exists T\subset[n]: |H_T|\le |T|, X_T< d|T|} \le \expo{-c\delta^2pn}.
\end{equation}
First we will control the event that $X_T< d|T|$ for some $T$ of size less than $(1-\delta/2)n$.
By the union bound and exchangeability of the rows and columns of $B$,
\begin{align*}
&\pro{ \exists T\subset[n]: |H_T|\le |T| \le (1-\delta/2)n, X_T<d|T|}\\
&\qquad\qquad\qquad\le \sum_{\substack{T\subset[n]: \\|T|\le (1-\delta/2)n}}\sum_{\substack{H\subset[n]:\\ |H|\le |T|}} \pr\big( e(H^c,T)<d(|T|-|H|)\big)\\
&\qquad\qquad\qquad\le \sum_{t\le (1-\delta/2)n} \sum_{h\le t} {n\choose t} {n\choose h} \pr\big( e([n-h], [t]) < d(t-h)\big).
\end{align*}
Now since
\begin{align*}
d(t-h) &= (1-\delta)pn(t-h) \\
&= (1-\delta) p(n-h)t -(1-\delta)ph(n-t) \\
&\le  (1-\delta) p(n-h)t\\
&= (1-\delta) \e e([n-h],[t]),
\end{align*}
we can use Bernstein's inequality to bound
\begin{align*}
 \pr\big( e([n-h], [t]) < d(t-h)\big) &\le  \pr\big( e([n-h], [t]) <(1-\delta) \e e([n-h],[t])\big) \\
 &\le \expo{ -c \delta^2p (n-h)t}
\end{align*}
for some constant $c>0$.
Thus,
\begin{align*}
\pro{ \exists T: |H_T|\le |T| \le (1-\delta/2)n, X_T<d|T|}
&\le \sum_{t\le (1-\delta/2)n} \sum_{h\le t} {n\choose t} {n\choose h}\expo{ -c \delta^2p (n-h)t}.
\end{align*}
For $t\le n/2$ we can bound
\begin{align*}
 \sum_{h\le t} {n\choose t} {n\choose h}\expo{ -c \delta^2 p(n-h)t}
 &\le n^t\expo{ -\frac12c\delta^2pnt}\sum_{h\le t} n^h\\
 &\le  n^{2t+1} \expo{ -\frac12c \delta^2 pnt}\\
 &\le \expo{ -\frac14c\delta^2pnt}
\end{align*}
where we have applied the first lower bound in \eqref{lbdelta}, taking $C$ sufficiently large. For $n/2\le t\le (1-\delta/2)n$ we have
\begin{align*}
 \sum_{h\le t} {n\choose t} {n\choose h}\expo{ -c \delta^2 p(n-h)t}
 &\le 10^n \expo{ -\frac12c \delta^2 p(n-t)n}\\
 &\le \expo{ -\frac14c\delta^2pn(n-t)}
\end{align*}
where we have lower bounded $n-t\ge \delta n/2$ and applied the second lower bound in \eqref{lbdelta}
(taking $C$ larger, if necessary).
Summing these bounds gives
\begin{equation}	\label{smallT}
\pro{ \exists T: |H_T|\le |T| \le (1-\delta/2)n, X_T<d|T|} \le \expo{ -c'\delta^2pn}.
\end{equation}

It remains to bound the probability that $X_T<d|T|$ for some $T$ of size $|T|\ge (1-\delta/2)n$.
For $T\subset[n]$ define the set of \emph{light} vertices
\[
L_T:= \{i\in [n]: \deg_T(i) \le (1-\delta/2)p|T|\}.
\]
For fixed $T$ and $i$, since $\e \deg_T(i)=p|T|$, from Bernstein's inequality we have
\[
\pro{ i\in L_T} \le \expo{ -c\delta^2 p|T|}.
\]
Thus, for any $\ell\ge1$ we have
\[
\pro{ |L_T|\ge \ell} \le n^\ell \pro{ [\ell]\subseteq L_T} \le \expo{ \ell\big( \log n - c\delta^2 p|T|\big)}.
\]
Thus, taking $C$ larger if necessary, and if $|T|\ge n/2$, say, then 
\[
\pro{ |L_T|\ge \ell} \le \expo{ -(c/2)\delta^2p|T|\ell}.
\]
By the union bound, for each $1\le m\le n/2$,
\begin{align*}
\pro{ \exists T\in {[n]\choose n-m}: |L_T|\ge \ell} 
&\le n^m \expo{ -(c/2)\delta^2 p(n-m)\ell}\\
&\le  n^m\expo{ -(c/4)\delta^2pn\ell}\\
&\le \expo{ -c'\delta^2pn\ell}
\end{align*}
if
\begin{equation}
\ell= \ell^*(m):=\frac{C m \log n}{\delta^2 pn}.
\end{equation}
By another union bound over the choices of $m$ (and adjusting the constants $C,c'$) we conclude that except with probability at most $\expo{ -c'\delta^2 pn}$, 
\begin{equation}	\label{good.big}
\forall\, 1\le m\le n/2, \quad \forall \,T\in {[n]\choose n-m}, \qquad |L_T|\le \ell^*(m).
\end{equation}
We may restrict to this event. Now consider $T\subset[n]$ with $|T|=n-m$ for some $m\le \delta n/2$. 
We have
\begin{equation}	\label{XT.above}
X_T  \ge \sum_{i\notin L_T} \min(d, \deg_T(i)) \ge (n-\ell^*(m)) \min \big(d, (1-\delta/2)p(n-m)\big).
\end{equation}
Note that under \eqref{lbdelta} we have $m\le \ell^*(m)$. Also, since $m\le \delta n/2$ we have 
\[
(1-\delta/2)p(n-m)\ge (1-\delta)pn = d.
\]
Hence, the right hand side of \eqref{XT.above} is bounded below by $d(n-m)=d|T|$, as desired. 
Thus, we have shown that 
\begin{equation*}	
\pro{ \exists T\subset[n]: n(1-\delta/2)\le |T|\le n, X_T<d|T|} \le \expo{ -c'\delta^2 pn}.
\end{equation*}
which together with \eqref{smallT} completes the proof.

\begin{remark}
The second lower bound in \eqref{lbdelta} could likely be improved, or removed entirely, by separately controlling $X_T$ for sets $T$ of ``intermediate" size. However, we do not pursue such an improvement as it not necessary for the purposes of this work.
\end{remark}

\subsection{Proof of Lemma \ref{lem:band}}

Set $\delta$ equal the lower bound in \eqref{lbdelta}, let $p'$ such that $d= (1-\delta)p'n$, and let $B'\in \{0,1\}^{n\times n}$ have iid Bernoulli($p'$) entries. 

Let $\alpha\in (0,1/2)$. From Bernstein's inequality, the probability that a given row or column of $B'$ has support of size differing from $p'n$ by at least $d^{1-\alpha}$ is at most $2\expo{ -cd^{1-2\alpha}}$ for some constant $c>0$. By the union bound we have that all rows and columns have supports of size $p'n + O(d^{1-\alpha})$ with probability at least $1/2$, say, for all $n$ sufficiently large depending on $\eps$. 
By Lemma \ref{lem:factor}, $B'$ contains a $d$-regular factor with probability at least $3/4$. 
Denoting the intersection of these events by $\mG$, we have $\pr(\mG)\ge 1/4$.

Identifying $\mG$ with a subset of $\{0,1\}^{n\times n}$, we see that every element of $\mG$ can be obtained by taking an appropriate element of $\mA_{n,d}$ and adding at most 
\begin{equation}	\label{m.def}
m:=\delta p'n+ d^{1-\alpha} \ll \sqrt{d\log n} + d^{2/3} + d^{1-\alpha}
\end{equation}
entries to each row and column. Thus, letting $\mB(B,m)\subset \{0,1\}^{n\times n}$ denote the set of matrices that can be obtained from $B\in \{0,1\}^{n\times n}$ by adding at most $m$ entries to each row and column, we have
\[
1/4\le \pr(B'\in \mG) \le \sum_{B\in \mA_{n,d}} \pr(B'\in \mB(B,m)).
\]
Since $p'\le 1/2$ and each element of $\mB(B,m)$ has at least $nd$ nonzero entries, 
\[
\pr(B'\in \mB(B,m))\le {n\choose m}^{2n} (p')^{nd}(1-p')^{n^2-nd}
\]
which yields the estimate
\begin{equation}	\label{And.est}
|\mA_{n,d}| \ge \frac14{n\choose m}^{-2n} (p')^{-nd}(1-p')^{-(n^2-nd)}.
\end{equation}

Now let $B$ be as in Lemma \ref{lem:band}. We first note that
\begin{equation}	\label{goal:ge}
\pr(B\in \mA_{n,d})\ge \pr(B'\in \mA_{n,d}).
\end{equation}
This is easily seen from the following:
\begin{align*}
\frac{\pr(B\in \mA_{n,d})}{ \pr(B'\in \mA_{n,d})}
&= \left(\frac{p}{p'}\right)^{nd} \left(\frac{1-p}{1-p'}\right)^{n^2-nd} = \expo{ n^2 D_{KL}(\mu_p||\mu_{p'})},
\end{align*}
where $\mu_q$ is the Bernoulli$(q)$ measure on $\{0,1\}$ and $D_{KL}(\mu_p||\mu_{p'})$ is the Kullback--Leibler divergence from $\mu_{p'}$ to $\mu_p$. Since the latter is strictly positive for $p\ne p'$, \eqref{goal:ge} follows.
Combining \eqref{goal:ge}, \eqref{And.est} and \eqref{m.def}, we conclude
\begin{align*}
\pr(B\in \mA_{n,d})&\ge \pr(B'\in \mA_{n,d})\\
&= |\mA_{n,d}|(p')^{nd}(1-p')^{n^2-nd}\\
&\ge \frac14 {n\choose m}^{-2n}\\
&\ge \frac14 \expo{ -2nm\log n}\\
&\ge \expo{ -O(n\log n(\sqrt{d\log n} + d^{2/3} + d^{1-\alpha}))}
\end{align*}
as desired.

\bibliographystyle{alpha}
\bibliography{rrd_circ}

\end{document}